\documentclass[12pt,oneside]{book}
\usepackage{epsfig}
\usepackage{latexsym}
\usepackage{amssymb, latexsym}
\usepackage{amsmath}
\usepackage{mathrsfs}
\usepackage{amsxtra}
\usepackage{amsthm}
\usepackage{setspace}

\newtheorem{theorem}{Theorem}[chapter]
\newtheorem{corollary}[theorem]{Corollary}

\newtheorem{lemma}[theorem]{Lemma}

\newtheorem{claim}[theorem]{Claim}

\newtheorem{proposition}[theorem]{Proposition}

\theoremstyle{definition}
\newtheorem{definition}[theorem]{Definition}

\newtheorem{remark}[theorem]{Remark}

\usepackage{graphicx}

\title{Geometric Interpretation of \\ the Two Dimensional Poisson Kernel \\ And Its Applications}
\author{Sergei Artamoshin} 
\date{2009}

\begin{document}




\pagestyle{empty}

\vspace*{4cm}

\centerline{\Large\textbf{Geometric Interpretation of}}
\bigskip
\centerline{\Large\textbf{the Two Dimensional Poisson Kernel}}
\bigskip
\centerline{\Large\textbf{And Its Applications}}

\bigskip
\bigskip
\bigskip

\centerline{\Large{Sergei Artamoshin}}

\bigskip
\bigskip

\centerline{The Graduate Center}
\centerline{The City University of New York}
\centerline{Department of Mathematics}
\centerline{365 Fifth Avenue, New York, NY, 10016.}
\centerline{e-mail: kolmorpos@yahoo.com}

\bigskip

\bigskip
\bigskip

\centerline{\large{November, 2009}}

\newpage


\vspace*{4cm}

\begin{spacing}{1.5}

\centerline{\textbf{Abstract}}

\bigskip

\noindent Hermann Schwarz, while studying complex analysis, introduced the geometric interpretation for the Poisson 
kernel in 1890, see \cite{Schwarz} (pp. 359-361) or V. Ahlfors, \cite{Ahlfors} (p.168). We shall see here that the 
geometric interpretation can be useful to develop a new approach to some old classical problems as well as to obtain 
several new results, mostly related to hyperbolic geometry. For example, we obtain One Radius Theorem saying that any 
two radial eigenfunctions of a Hyperbolic Laplacian assuming the value 1 at the origin can not assume any other 
common value within some interval [0, p], where the length of this interval depends only on the location of the 
eigenvalues on the complex plane and does not depend on the distance between them.



 \chapter*{Acknowledgements} 

 \thispagestyle{empty}

I am very grateful to professors A. Koranyi and Z.I.
Szabo who pointed out to me that $\omega^{\alpha}$ as well as its radialization (Spherical Mean) is an eigenfunction of the hyperbolic
Lapalcian. This is why all applications of the Main Theorem and the geometric interpretation of $\omega$ in itself to the
hyperbolic geometry became possible. This PhD program has been
supported by NSF Grant DMS-0604861. Besides them, I am equally grateful
to J\'ozef Dodziuk for his interest and valuable comments supporting me since
the very beginning these researches started.

Also I would like to thank Eliezer Hameiri and Leon Karp, whose lectures in Applied Mathematics and PDE inspired me to consider a number of applications, as well as Alexander Kheyfits for some helpful discussions.

\newpage

\pagestyle{headings}

\pagenumbering{arabic}

\setcounter{page}{1}

\tableofcontents


\chapter{Introduction}
\section{General comments.}

Recall that the Poisson kernel
\begin{equation}\label{Intro-1}
    P_{\,1}(x,y)=\frac{1}{2\pi R}\cdot
    \frac{||y|^2-|x|^2|}{|x-y|^2}\,,
\end{equation}
where $|y|=R$ is used as a tool to solve the Dirichlet problem for
a plane disk $\mathbf{B}^2_R (O)$ or for its exterior. In this paper we shall see that
the normalized expression $\omega = 2\pi R \, P_1(x,y)$ considered as a function of $x,y\in
\mathbb{R}^{k+1}$ has a geometric interpretation that can serve as a useful tool to obtain a number of results related to Euclidean and Hyperbolic geometry. Originally, for $x,y\in\mathbb{R}^2$ and $|x|<|y|$, this interpretation was observed by Hermann Schwarz, see \cite{Schwarz} (pp. 359-361) or \cite{Ahlfors} (p. 168). Here we present the uniform description of the geometric interpretation, which includes both cases: $|x|<|y|$ and $|x|>|y|$. As a first consequence of this interpretation we shall see that $\omega(x,y)$ satisfies an integral identity that leads to direct relationship between the Newtonian potential in $\mathbb{R}^{k+1}$ and $(k+1)-$dimensional Poisson kernel. Then, we shall present the following applications.

\medskip

  \section{Analysis in the Euclidean space and \\ applications to electrostatics.}

  Below is the list of applications related to the Euclidean space.

  \begin{enumerate}
    \item An algebraic way to compute certain integrals arising in electrostatics.
    \item A sufficient integral condition for a function depending only on distance to be harmonic.
    \item A new derivation for a solution of the classical
Dirichlet problem in the $(k+1)$-dimensional ball of radius $R$ or
in its exterior.

The standard way of obtaining the Poisson kernel involves the normal derivative of Green's function, see, for example, \cite{Kellog} (pp. 240-241) or \cite{Fritz} (pp. 106-108). The derivation presented here does not involve Green's identity or Green's function and relies on an integral identity of the two-dimensional Poisson kernel introduced here and, for real $\alpha\neq\beta$, expressed
by the following equivalence

\begin{equation}\label{1-ff}
    \alpha+\beta=k\quad \iff \quad
\int\limits_{S^k} \omega^\alpha\,
dS_y=\int\limits_{S^k} \omega^\beta\,
dS_y\,.
\end{equation}
where
\begin{equation}\label{3}
    \omega=2\pi R\cdot
    P_{\,1}(x,y)=\frac{|R^2-|x|^2|}{|x-y|^2}=\frac{||x|^2-|y|^2|}{|x-y|^2}
\end{equation}
with $x\in \mathbb{R}^{k+1}\setminus S^k$ and $y\in
S^k$. $S^k$ is a $k$-dimensional sphere of
radius $R$ centered at the origin $O$.
    \item In the Main Theorem, p.~\pageref{Basic-theorem}, we derive the restrictions for $R, |x|, \alpha, \beta$ so that the equivalence \eqref{1-ff} remains true for complex values $\alpha$ and $\beta$.
    \item Some non-trivial inequalities will be established as a consequences of the integral identity \eqref{1-ff} described above.
\end{enumerate}

\section {Analysis in the Hyperbolic space and \\ applications to the Dirichlet eigenvalue problem.}

The function $\omega(x,y)$ introduced above is useful for studying the Hyperbolic Laplacian, since $\omega^\alpha$ is a radial eigenfunction of the Hyperbolic Laplacian in the ball model. For example, using the integral identity \eqref{1-ff} extended to complex powers $\alpha$ and $\beta$, together with the geometric interpretation of $\omega$, we can obtain the following results concerning eigenfunctions and eigenvalues of Hyperbolic Laplacian.

\begin{enumerate}
  \item Consider the set of all radial eigenfunctions of the Hyperbolic Laplacian assuming the value 1 at the origin, i.e. the set of all solutions for the following system
\begin{equation}\label{Intro-2}
\left\{
  \begin{array}{ll}
     & \hbox{$\varphi^{''}(r)+\frac{k}{\rho}\coth\left(\frac{r}{\rho}\right)\varphi^{'}(r)+\lambda\varphi(r)=0$, $\lambda\in\mathbf{C}$;} \\
     & \hbox{$\varphi(0)=1$\,,}
  \end{array}
\right.
\end{equation}
written in the geodesic polar coordinates of the hyperbolic space of constant sectional curvature $\kappa=-1/\rho^2$.
Recall that for every $\lambda\in\mathbf{C}$ there exists a unique solution $\varphi_\lambda(r)$ such that $\varphi_\lambda(0)=1$, see \cite{Chavel} (p. 272). If we choose two radial eigenfunctions $\varphi_\mu(r)$ and $\varphi_\nu(r)$ such that $\varphi_\mu(0)=\varphi_\nu(0)=1$, then, how many values of $r>0$ such that $\varphi_\mu(r)=\varphi_\nu(r)$ are sufficient to ensure that $\mu=\nu$ and $\varphi_\mu(r)\equiv\varphi_\nu(r)$? We will see that according to the Main Theorem, introduced on page~\pageref{Basic-theorem}, we need only one such a value of $r$ if the value is small enough. In other words, if $\mu\neq\nu$, then there exists an interval $(0,p(\mu,\nu)]$, such that $\varphi_\mu(r)\neq\varphi_\nu(r)$ for all $r\in(0,p(\mu,\nu)]$. In particular, we prove that if $\mu\neq \nu$ are real and $\mu, \nu\leq k^2/4$ then $\varphi_\mu(r)\neq\varphi_\nu(r)$ for all $r\in(0,\infty)$. So, if $\mu, \nu\leq k^2/4$ and $\varphi_\mu(r)=\varphi_\nu(r)$ just for one arbitrary $r>0$, then $\mu=\nu$ and $\varphi_\mu(r)\equiv\varphi_\nu(r)$.
  \item What is the lower bound and the upper bound for the smallest positive eigenvalue of a Dirichlet Eigenvalue Problem? This question has been discussed in many papers. For the bibliography, see, for example, \cite{Chavel} or \cite{Cheng}. Recall that the Dirichelt Eigenvalue Problem for a disc of radius $\delta$ is formulated as follows. Find all $\lambda\in\mathbb{R}$ and corresponding functions $\varphi_\lambda(\upsilon, r)$, which are eigenfunctions of Hyperbolic Laplacian in $(k+1)-$dimensional hyperbolic space with a constant sectional curvature $\kappa=-1/\rho^2$ such that
\begin{equation}\label{Dirichlet_Original}
\left\{
  \begin{array}{ll}
     & \hbox{$\triangle\varphi_\lambda(\upsilon,r)+\lambda \varphi_\lambda(\upsilon,r)=0 \quad \forall r\in [0,\delta], \,\,\lambda - \text{real}$;} \\
     & \hbox{$\varphi_\lambda(\upsilon, \delta)=0$,}
  \end{array}
\right.
\end{equation}
where $\upsilon$ is a point from the unit sphere $\sigma_O^k$ centered at the origin. H.P. McKean showed that
\begin{equation}\label{Mckean_Estimation}
\begin{split}
    & \lambda\geq -\frac{\kappa k^2}{4} \quad\text{for all}\quad \delta>0
    \\& \text{and}\quad\lim\limits_{\delta\rightarrow +\infty}\lambda(\delta)=-\frac{\kappa k^2}{4}\,,
\end{split}
\end{equation}
see \cite{McKean} or \cite{Chavel} (p.46). Relatively recent result was obtained by Jun Ling in \cite{JunLING}. His lower bound of $\lambda$ from the problem \eqref{Dirichlet_Original} can be written as
\begin{equation}\label{Intro-4-1}
    \lambda\geq  \frac{\kappa}{2}+\left(\frac{\pi}{2\delta}\right)^2\,.
\end{equation}
Note that since $\kappa=-1/\rho^2$ is negative, the estimate \eqref{Intro-4-1} becomes trivial for large value of $\delta$. In this paper we will obtain an explicit representation of a radial solution for the problem and will see that the smallest eigenvalue $\lambda$ in \eqref{Dirichlet_Original} must satisfy the following inequality
\begin{equation}\label{Intro-4-2}
    \lambda>  \frac{-\kappa k^2}{4}+\left(\frac{\pi}{2\delta}\right)^2\,,
\end{equation}
which for $k\geq3$ can be improved to
\begin{equation}\label{Intro-4-3}
    \lambda>  \frac{-\kappa k^2}{4}+\left(\frac{\pi}{\delta}\right)^2\,
\end{equation}
and for $k=2$ the smallest eigenvalue $\lambda$ can be computed precisely, i.e.,
\begin{equation}\label{Intro-4-4}
    \lambda_{\min}= -\kappa+\left(\frac{\pi}{2\delta}\right)^2\,.
\end{equation}

As a bonus, the technique developed in this paper yields also the upper bound for $\lambda_{\min}$ in the two dimensional space. The estimate obtained here is a bit stronger than the upper bound in two dimensional space obtained by S.Y. Cheng, see \cite{Cheng} or \cite{Chavel} (p.82), but still, remains weaker then the upper bound obtained by M. Gage, see \cite{Gage} or \cite{Chavel} (p. 80). On the other hand, in the three dimensional space \eqref{Intro-4-4} is much stronger than the estimates obtained by Gage and Cheng in their works mentioned above.

  \item The next question is to figure out whether the inequality and the limit in \eqref{Mckean_Estimation} obtained by H. McKean remain true for $\delta=\infty$. By the Dirichlet eigenvalue problem at $\infty$ we understand the following system of conditions.
\begin{equation}\label{Dirichlet_at_infty}
\left\{
  \begin{array}{ll}
     & \hbox{$\triangle\varphi_\lambda(\upsilon,r)+\lambda \varphi_\lambda(\upsilon,r)=0 \quad \forall r\in [0,\infty), \,\,\lambda - \text{real}$;} \\
     & \hbox{$\lim\varphi_\lambda(\upsilon, r)=0 \quad \text{as}\,\, r\rightarrow\infty\,,$}
  \end{array}
\right.
\end{equation}
which is the natural extension of \eqref{Dirichlet_Original} for $\delta=\infty$. We shall see in Theorem~\ref{Infinity-Behavior}, p.~\pageref{Infinity-Behavior} that for all $\lambda\in(0,\infty)$ this problem has a solution. At least, for every $\lambda\in(0,\infty)$ there exists the unique radial eigenfunction assuming the value 1 at the origin and satisfying \eqref{Dirichlet_at_infty}. We can even write down such an eigenfunction explicitly. So, the lower bound for $\lambda_{\min}$ from \eqref{Dirichlet_at_infty} is 0, and then, the results listed in \eqref{Mckean_Estimation} are not appropriate for eigenvalues from \eqref{Dirichlet_at_infty}. Note, however, that according to Harold Donelly, see \cite{Donelly} 
none of the solutions belongs to $L^2(B_\rho^{k+1})$ with respect to the hyperbolic measure.
  \item We shall see in Theorem~\ref{Explicit-Solution-Theorem}, p.~\pageref{Explicit-Solution-Theorem} that in the hyperbolic space of three dimensions all the radial Dirichlet eigenfunctions together with their eigenvalues can be computed explicitly, i.e., the set of the following formulae
\begin{equation}\label{Intro-6}
    \lambda_j=-\kappa+\left(\frac{\pi j}{\delta}\right)^2\,,\,\,j=1,2,...
\end{equation}
yields the whole spectrum for the Dirichlet eigenvalue problem \eqref{Dirichlet_Original} restricted by a non-zero condition at the origin. The radial solution assuming the value 1 at the origin for each $\lambda_j$ can be written as  \begin{equation}\label{Intro-7}
    \varphi_{\lambda_j}(r)=\frac{\delta(\rho^2-\eta^2)}{2\pi\rho^2 \eta j}\cdot\sin\left(\frac{\pi jr}{\delta}\right)=
    \frac{\delta\sin(\pi jr/\delta)}{\pi j\rho\sinh(r/\rho)}\,,
\end{equation}
where $0\leq r\leq\delta<\infty$ and $\eta=\rho\tanh(r/2\rho)$.
    \item We shall see in Theorem~\ref{Lambda-Positively-Large-Theorem}, p.~\pageref{Lambda-Positively-Large-Theorem} that the geometric interpretation is useful to investigate the asymptotic behavior of a radial eigenfunction $\varphi_\lambda(r)$ as $\lambda\rightarrow\infty$ and $r$ is fixed. We compute the leading terms of the asymptotic decompositions in both cases: as $\lambda\rightarrow\infty$ (Theorem~\ref{Lambda-Positively-Large-Theorem}, p.~\pageref{Lambda-Positively-Large-Theorem}) and $\lambda\rightarrow-\infty$ (Theorem~\ref{Asymptotic-behavior-Theorem}, p.~\pageref{Asymptotic-behavior-Theorem}).
\end{enumerate}

\chapter {Some geometric results.}

\medskip

\section {Geometric Interpretation of $\omega(x,y)$.}

\smallskip

Let $x,y\in\mathbb{R}^{k+1}$ and assume that $|x|\neq|y|$. Let $S^k(R)$ denotes the $k$-dimensional sphere of radius $R$ centered at the origin $O$. Then, let us define $x^*$ and $y^*$ in such a way that
\begin{equation}\label{GeoInte-1}
    x^*\in S^k(|x|),\,\,\,y^*\in S^k(|y|)\quad\text{and}\quad x^*, y^*,x,y\,\,\,\text{are collinear}\,.
\end{equation}

\begin{proposition}
\begin{equation}\label{GeoInte-2}
    |x-y^*|=|x^*-y|\,.
\end{equation}
\end{proposition}
\begin{proof}
Consider the $k$-dimensional plane passing though the origin and orthogonal to the line($x,y$). Clearly, $x^*$ is the reflection of $x$ with respect to the plane. For the same reason, $y^*$ is the reflection of $y$, and therefore, segment $x^*y$ is the reflection of $xy^*$. Hence, $|x-y^*|=|x^*-y|$.
\end{proof}

\bigskip

Now, using the proposition above, we may introduce notation for the following segments:
\begin{equation}\label{GeoInte-3}
    q=|x-y|\quad\text{and}\quad l=|x-y^*|=|x^*-y|\,.
\end{equation}

\begin{theorem}[Geometric Interpretation]\label{Geometric-Interpretation-theorem}
\begin{equation}\label{GeoInte-4}
    \omega(x,y)=\frac{||x|^2-|y|^2|}{|x-y|^2}=\frac{l}{q}\,.
\end{equation}
This expression will be referred to as the Geometric Interpretation of $\omega$.
\end{theorem}

\begin{proof} Figure \ref{Geometric_Interpretation} below represents the two dimensional plane defined by three points: $O,x,y$. The segment $MP$ is the tangent chord to the smaller sphere, say, $S^k(|x|)$ at point $x$. Then, clearly, $a=|Mx|=|xP|$ is one half of the length of the chord. Pythagorean theorem implies that



\begin{figure}[!h]
    \centering
    \epsfig{figure=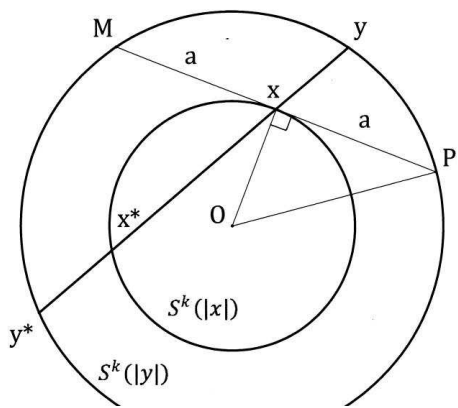,height=5.5cm}
    \caption{Geometric Interpretation.}\label{Geometric_Interpretation}
\end{figure}

\begin{equation}\label{GeoInte-5}
    ||x|^2-|y|^2|=a^2\,.
\end{equation}
In addition, $\triangle y^*xM$ is similar to $\triangle Pxy$, which yields
\begin{equation}\label{GeoInte-6}
    a^2=|y^*-x|\cdot |x-y|=lq\,.
\end{equation}
Combining \eqref{GeoInte-5} and \eqref{GeoInte-6}, we may write that
\begin{equation}\label{GeoInte-7}
    \omega(x,y)=\frac{||x|^2-|y|^2|}{|x-y|^2}=\frac{a^2}{q^2}=\frac{lq}{q^2}=\frac{l}{q}\,.
\end{equation}
The same argument applies if the sphere $S^k(|y|)$ is the smaller one. This completes the proof of Theorem~\ref{Geometric-Interpretation-theorem}.
\end{proof}

\begin{corollary} Using the geometric interpretation obtained above, we observe that $\omega(x,y)$ remains constant as long as $x$ stays at a circle tangent to $S^k(|y|)$ and centered at some point of the line $Oy$. This follows from the Lemma \ref{omega_constant-lemma} below.
\end{corollary}

\begin{lemma}\label{omega_constant-lemma} Let $T$ be any point on the line $Oy$ such that $T\neq y$ and let $S^k_T(|Ty|)$ be the sphere of radius $\delta=|Ty|$ centered at $T$, where $|Ty|=|T-y|$ is the distance from $T$ to $y$. Then
\begin{equation}\label{Omega-constanta}
    \omega(x,y)=\frac{|OT|}{\delta}\quad\text{for all}\,\,\,x\in S^k_T(|Ty|)\setminus \{y\}\,.
\end{equation}
\end{lemma}

\begin{proof}[Proof of Lemma \ref{omega_constant-lemma}]
Let us choose $T\in\text{line}(Oy)$ as it is shown on Figure~\ref{Omega-constant} below.

\begin{figure}[!h]
    \centering
    \epsfig{figure=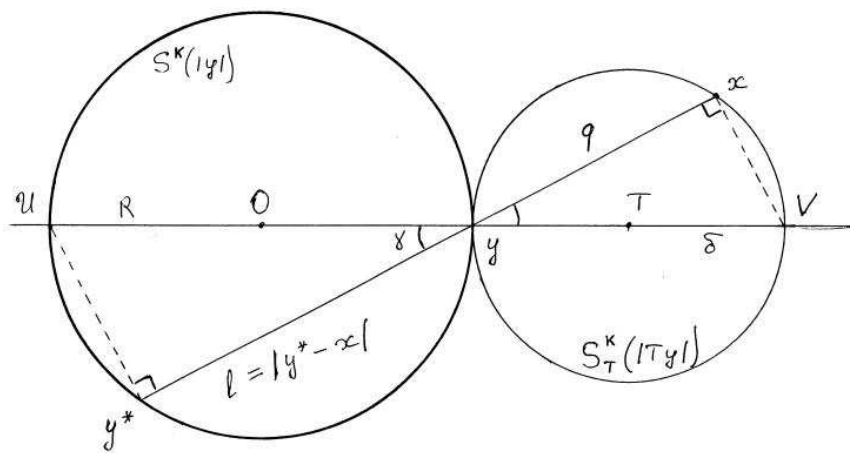,height=6cm}
    \caption{Level hypersurface for $\omega$.}\label{Omega-constant}
\end{figure}

Let $|Oy|=R$, $U$ and $V$ be the intersections of line $Oy$ with the spheres $S^k(|y|)$ and $S^k_T(|Ty|)$ respectively. $y^*\in S^k(|y|)$ is defined by \eqref{GeoInte-1}. Let $\gamma=\angle Oyy^*=\angle xyT$. Observe that $\angle Uy^*y=\angle yxV=\pi/2$ since the segments $Uy$ and $yV$ are diameters. Therefore,
\begin{equation}\label{l-for-level-hypersurface}
    l=|y^*-x|=|y^*y|+|yx|=2R\cos\gamma+2\delta\cos\gamma
\end{equation}
and
\begin{equation}\label{q-for-level-hypersurface}
    q=|y-x|=2\delta\cos\gamma\,.
\end{equation}
Hence, the combination of \eqref{GeoInte-4}, \eqref{l-for-level-hypersurface} and \eqref{q-for-level-hypersurface} yields
\begin{equation}
    \omega(x,y)=\frac{l}{q}=1+\frac{R}{\delta}=\frac{\delta+R}{\delta}=\frac{|OT|}{\delta}\,.
\end{equation}
Therefore, for the chosen point $T$ formula \eqref{Omega-constanta} is justified. A similar argument yields \eqref{Omega-constanta} for all other positions of $T$ on the line $Oy$. This completes the proof of Lemma~\ref{omega_constant-lemma}.
\end{proof}

\begin{corollary}\label{limit-for-omega}
    Let $p$ be the $\mathcal{C}^1$ path through $y$ and let $x\in p$. Then
\begin{equation}\label{Omega-limit-formula}
    \lim\limits_{p\ni x\rightarrow y} \omega(x,y)= A(p)\in[0, \infty]\,,
\end{equation}
where the value $A(p)$ depends on the path $p$ chosen and can assume any value from the closed interval $[0,\infty]$.
\end{corollary}

\begin{proof}[Proof of Corollary \ref{limit-for-omega}]
Note first, if $p\subseteq S^k(|y|)$, then $\omega(x,y)\equiv0$ for every $x\in p\setminus \{y\}$ and therefore, $A(p)=0$. If $p\subseteq S^k_T(|yT|)$, then, by \eqref{Omega-constanta},
\begin{equation}
    A(p)=\frac{|OT|}{|\delta|}=\frac{|OT|}{|yT|}\,,
\end{equation}
which can be any number from $(0,\infty)$ depending on $T$ chosen on the line $Oy$. In particular, if $T=\infty$, then $A(p)=1$, which corresponds to the case when $p$ belongs to the tangent hyperplane to $S^k(|y|)$ at point $y$. In this case $\omega(x,y)\equiv1$ for every $x\in p\setminus \{y\}$. Finally, to get $A(p)=\infty$, $p$ must be a $\mathcal{C}^1$ curve non-tangential to $S^k(|y|)$ at point $y$. Indeed, let $\gamma_0$ be the angle between $p$ and line$(Oy)$ at point $y$ as it is pictured on Figure~\ref{Omega-Limit-picture} below.
\begin{figure}[!h]
    \centering
    \epsfig{figure=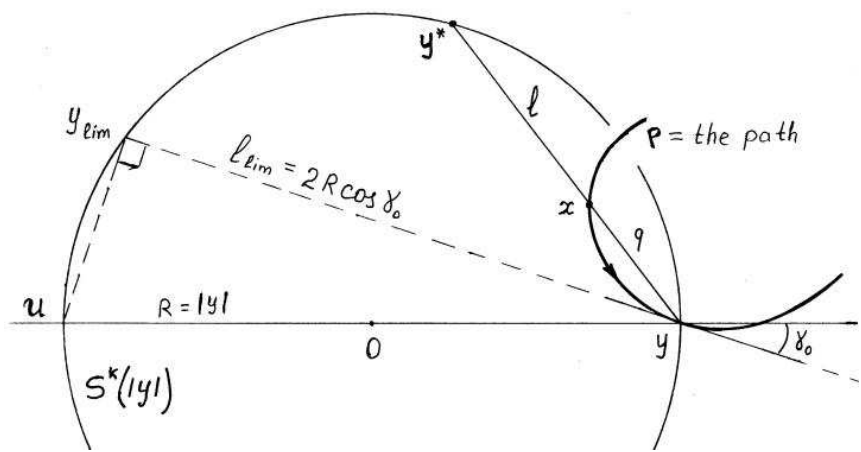,height=6cm}\\
    \caption{Non-tangential limit for $\omega$.}\label{Omega-Limit-picture}
\end{figure}

Note that $Uy$ is the diameter which implies that $\angle Uy_{\lim}y=\pi/2$, and then,
\begin{equation}
    \lim\limits_{p\ni x\rightarrow y} l(x)=l_{\lim}=2R\cos\gamma_0\,,
\end{equation}
while
\begin{equation}
    \lim\limits_{p\ni x\rightarrow y} q(x)=\lim\limits_{p\ni x\rightarrow y} |x-y|=0\,.
\end{equation}
Therefore,
\begin{equation}
    \lim\limits_{p\ni x\rightarrow y} \omega(x,y)= \frac{l(x)}{q(x)}=\infty\,,
\end{equation}
which completes the proof of Corollary \eqref{limit-for-omega}.
\end{proof}

\section {Sphere exchange rule.}  

The following lemma describes the important rule that can be used instead of successive application of inversion and then dilation in the case of a function depending only on distance is integrated over sphere.

\begin{lemma}[Sphere exchange rule]\label{Basic_lemma}

Let $S^k(r)$ and $S^k(R)$ be
two $k$-dimensional spheres of radii $r$ and $R$ respectively. Let $x,x_1,y,y_1\in \mathbb{R}^{k+1}$ be arbitrary points satisfying $|x|=|x_1|=r$ and $|y|=|y_1|=R$. We assume that $x,y$ are fixed, while $x_1,y_1$ be the parameters of integration. If $g:\mathbb{R}\rightarrow \mathbf{C}$ is an integrable complex-valued function on
$[|R-r|,|R+r|]$, then
\begin{equation}\label{ExchaSphere-1}
    R\,^k\cdot \int\limits_{S^k(r)}
    g(|x_1-y|)\,d\,S_{x_1}=r^k\cdot \int\limits_{S^k(R)}
    g(|x-y_1|)\,d\,S_{y_1}
\end{equation}
and
\begin{equation}\label{ExchaSphere-2}
    R\,^k\cdot \int\limits_{S^k(r)} g\circ \omega(x_1,y)
    d\,S_{x_1}=r^k\cdot \int\limits_{S^k(R)}
    g\circ\omega(x,y_1)d\,S_{y_1}\,,
\end{equation}
where $g\circ\omega$ is the composition of two functions.

Moreover, for every integrable complex-valued function $f$ defined on the unit
$k$-dimensional sphere centered at the origin $O$ we have
\begin{equation}\label{ExchaSphere-3}
    R\,^k\cdot \int\limits_{S^k(r)}
    g(|x-y_0|)\cdot f\left(\frac{x}{|x|}\right)d\,S_x=
    r^k\cdot \int\limits_{S^k(R)}
    g(|x_0-y|)\cdot f\left(\frac{y}{|y|}\right)d\,S_y\,,
\end{equation}
if $x_0\in S^k(r)$, $y_0\in S^k(R)$ and $O$
are collinear.

\end{lemma}

\begin{proof}[Proof of Lemma \ref{Basic_lemma}] Let us prove successfully all formulae listed in the Lemma.

\begin{proof}[Proof of \eqref{ExchaSphere-1}.]

Let us fix $y_0\in S^k(R)$ and define $x_0$ as the following intersection:
\begin{equation}\label{ExchaSphere-4}
    x_0=S^k(R)\cap \text{Ray}(Oy_0)\,.
\end{equation}
Then, let $\widetilde{y}\in S^k(R)$ be a variable point and
\begin{equation}\label{ExchaSphere-5}
    \widetilde{x}=S^k(r)\cap \text{Ray}(O\widetilde{y})\,.
\end{equation}
The figure below shows the plane defined by Ray($O\widetilde{y}$) and Ray($Oy_0$). The points $x$ and $y$ denoted on the picture below need not to be on the cross-sectional plane. Clearly, such a construction yields $\triangle O\widetilde{x}y_0$ and $\triangle O\widetilde{y}x_0$ are congruent, and then,
\begin{equation}\label{ExchaSphere-5.1}
    |\widetilde{y}-x_0|=|\widetilde{x}-y_0|\,.
\end{equation}

\begin{figure}[!h]
    \centering
    \epsfig{figure=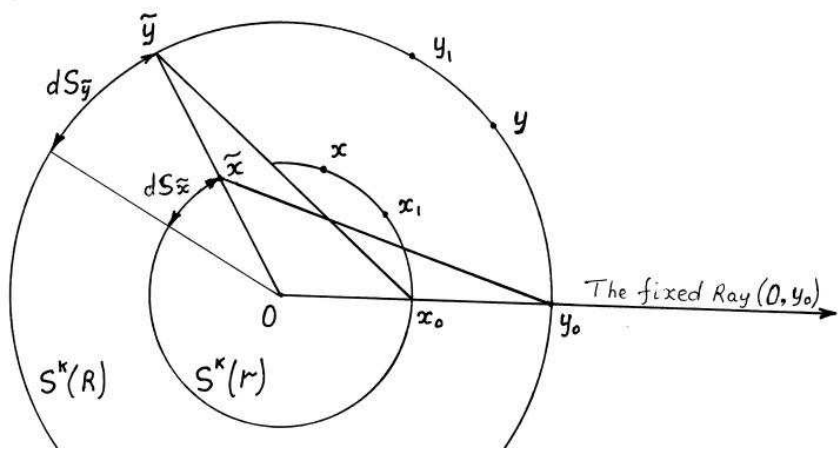,height=6cm}\\
    \caption{Sphere exchange rule.}\label{Sphere_Exchange}
\end{figure}

\smallskip

Let $dS_{\widetilde{x}}$ be the measure of some spherical infinitesimal neighborhood $U(\widetilde{x})$ around $\widetilde{x}$ and $dS_{\widetilde{y}}$ be the measure of the spherical infinitesimal neighborhood around $\widetilde{y}$ obtained as a dilated image of $U(\widetilde{x})$. This implies that
\begin{equation}\label{ExchaSphere-6}
    dS_{\widetilde{y}}=\left(\frac{R}{r}\right)^k dS_{\widetilde{x}}\,.
\end{equation}
Notice, now, that the function
\begin{equation}\label{ExchaSphere-7}
    F_1(y)=\int\limits_{S^k(r)} g(|x_1-y|)dS_{x_1}
\end{equation}
is invariant under isometries of $\mathbb{R}^{k+1}$ that fix the origin $O$. This is why $F_1(y)=F_1(y_0)$ for all $y,y_0\in S^k(R)$. Then, if we replace in $F_1(y_0)$ the parameter of integration $x_1$ to $\widetilde{x}$, we obtain the following formula
\begin{equation}\label{ExchaSphere-8}
    \int\limits_{S^k(r)} g(|x_1-y|)dS_{x_1}=\int\limits_{S^k(r)} g(|\widetilde{x}-y_0|)dS_{\widetilde{x}}\,.
\end{equation}
Recall that
\begin{equation}\label{ExchaSphere-10}
    |\widetilde{x}-y_0|=|x_0-\widetilde{y}|\quad\text{and}\quad
    dS_{\widetilde{y}}=\left(\frac{R}{r}\right)^k dS_{\widetilde{x}}\,.
\end{equation}
Therefore, by a change of variables and \eqref{ExchaSphere-10},
\begin{equation}\label{ExchaSphere-11}
\begin{split}
    & \int\limits_{S^k(r)} g(|\widetilde{x}-y_0|)dS_{\widetilde{x}}=\left(\frac{r}{R}\right)^k
    \int\limits_{\widetilde{x}\in S^k(r)} g(|x_0-\widetilde{y}|)\left(\frac{R}{r}\right)^k dS_{\widetilde{x}}
    \\& \quad\quad
    =\left(\frac{r}{R}\right)^k
    \int\limits_{\widetilde{x}(\widetilde{y})\in S^k(r)} g(|x_0-\widetilde{y}(\widetilde{x})|) dS_{\widetilde{y}(\widetilde{x})}
    \\& \quad\quad\quad\quad\quad\quad\quad\quad\quad\quad=
    \left(\frac{r}{R}\right)^k
    \int\limits_{\widetilde{y}\in S^k(R)} g(|x_0-\widetilde{y}|)dS_{\widetilde{y}}\,,
\end{split}
\end{equation}
where the last equality follows since
\begin{equation}\label{ExchaSphere-12}
    \widetilde{x}=\widetilde{x}(\widetilde{y})=\frac{r}{R}\cdot \widetilde{y}\in S^k(r)\quad\Leftrightarrow\quad
    \widetilde{y}=\widetilde{y}(\widetilde{x})=\frac{R}{r}\cdot \widetilde{x}\in S^k(R)\,.
\end{equation}
As above, the function
\begin{equation}\label{ExchaSphere-13}
    F_2(x_0)=\int\limits_{S^k(R)} g(|x_0-\widetilde{y}|)dS_{\widetilde{y}}
\end{equation}
is invariant under isometries of $\mathbb{R}^{k+1}$ that fix the origin. Thus,
\begin{equation}\label{ExchaSphere-14}
    \int\limits_{S^k(R)} g(|x_0-\widetilde{y}|)dS_{\widetilde{y}}=
    \int\limits_{S^k(R)} g(|x-\widetilde{y}|)dS_{\widetilde{y}} \quad\forall x,x_0\in S^k(r)\,.
\end{equation}
Finally, gathering all results from the chain \eqref{ExchaSphere-8}, \eqref{ExchaSphere-11}, \eqref{ExchaSphere-14} and changing of notation for the variable of integration, we have
\begin{equation}\label{ExchaSphere-15}
    \int\limits_{S^k(r)} g(|x_1-y|)dS_{x_1}=\left(\frac{r}{R}\right)^k
    \int\limits_{S^k(R)} g(|x-y_1|)dS_{y_1} \,,
\end{equation}
which completes the proof of \eqref{ExchaSphere-1} in Lemma \eqref{Basic_lemma}.
\end{proof}

\medskip

\begin{proof}[Proof of \eqref{ExchaSphere-2}.]

A similar argument is used to prove \eqref{ExchaSphere-2}.

The identity \eqref{ExchaSphere-2} holds, because the function $g\circ
\omega(x,y)=\widetilde{g}(|x-y|)$ is also integrable function of one variable $w=|x-y|\in[|R-r|, |R+r|]$ since \begin{equation}\label{ExchaSphere-16}
    \omega(x,y)=\frac{|R^2-r^2|}{|x-y|^2}
\end{equation}
is continuous as a function of $w=|x-y|$ if $R\neq r$. If $R=r$, observe that $\omega(x,y)\equiv0$ for all $y\neq x$, and then,
\begin{equation}
    \frac{1}{|S^k(r)|}\int\limits_{S^k(r)}g\circ\omega(x_1,y)dS_{x_1}=g(0)
    =\frac{1}{|S^k(R)|}\int\limits_{S^k(R)}g\circ\omega(x,y_1)dS_{y_1}\,.
\end{equation}
Therefore, \eqref{ExchaSphere-2} remains true for $R=r$ as well.
\end{proof}

\begin{proof}[Proof of \eqref{ExchaSphere-3}.]

Note, first, that $\widetilde{y}/|\widetilde{y}|=\widetilde{x}/|\widetilde{x}|$ and then, by \eqref{ExchaSphere-5.1},
\begin{equation}\label{ExchaSphere-17}
    g(|\widetilde{x}-y_0|)f\left(\frac{\widetilde{x}}{|\widetilde{x}|}\right)
    =g(|x_0-\widetilde{y}|)f\left(\frac{\widetilde{y}}{|\widetilde{y}|} \right)\,.
\end{equation}
Hence, we can repeat the same procedure of changing of variables as the one used in \eqref{ExchaSphere-11}. This yields
\begin{equation}\label{ExchaSphere-18}
    \int\limits_{S^k(r)} g(|\widetilde{x}-y_0|)f\left(\frac{\widetilde{x}}{|\widetilde{x}|}\right)dS_{\widetilde{x}}=
    \left(\frac{r}{R}\right)^k
    \int\limits_{S^k(R)} g(|x_0-\widetilde{y}|)f\left(\frac{\widetilde{y}}{|\widetilde{y}|}\right)dS_{\widetilde{y}} \,.
\end{equation}
Finally, by changing the notation for the variables of integration, where $\widetilde{x}$ is replaced by $x$ and $\widetilde{y}$ is replaced by $y$, we obtain \eqref{ExchaSphere-3}.
\end{proof}
Therefore, the proof of Lemma \eqref{Basic_lemma} is complete.
\end{proof}


\section {A useful property of $l$ and $q$.}

\smallskip

The next goal is to describe a useful feature of $l$ and $q$ defined above. Recall that if $O$ is the origin, $x, y\in\mathbb{R}^{k+1}$ and $R=|y|>|x|$, then
\begin{equation}\label{Feature-1}
    y^*\in S^k(|y|)\quad\text{such that}\quad y^*,x,y\quad\text{are collinear}\,;
\end{equation}

\begin{equation}\label{Feature-2}
    q(y)=|x-y|\,;\quad l(y)=|x-y^*|\,;\quad\psi=\angle Oxy\,.
\end{equation}
All of the notations are presented on the left Figure~\ref{l,q_exchange} below. Clearly, if $x$ and $R$ are fixed, $l$ and $q$ depend only on $\psi$ since both of the distances $|x-y|$ and $|x-y^*|$ depend only on $\psi, x, R$. Fix some $\psi\in(0,\pi)$. Then we observe that while there is the whole set of points
\begin{equation}\label{Feature-3}
    \overline{y}=\overline{y}(\psi)=\{y\in S^k(R)\, | \,\angle yxO=\psi\}\,,
\end{equation}
we need only one plane defined by $x,O$ and some arbitrary $y\in \overline{y}$ to demonstrate the desired relationship among $l,q$ and $\psi$. This is possible because all the values $l,q,\psi$ are invariant of $y\in \overline{y}$ and can be pictured on the plane passing through $x,O$ and some $y\in \overline{y}(\psi)$. The invariance mentioned implies that $q(\psi)$ and $l(\psi)$ can be defined as follows.
\begin{equation}\label{Feature-5}
    q(\psi)=q(\overline{y}(\psi))=|x-y|\quad\text{for every}\quad y\in \overline{y}\,;
\end{equation}
\begin{equation}\label{Feature-6}
    l(\psi)=l(\overline{y}(\psi))=|x-y^*|\,,
\end{equation}
where
\begin{equation}
   y^*\in S^k(|y|),\,\,\, y\in \overline{y} \quad\text{and}\quad y^*,x,y\quad\text{are collinear}.
\end{equation}
Fix some $y=y(\psi)\in \overline{y}$ and picture $l(\psi), q(\psi)$ on the plane defined by $O,x,y(\psi)$. (See the Figure \ref{l,q_exchange} below on the right).

\begin{figure}[!h]
    \centering
    \epsfig{figure=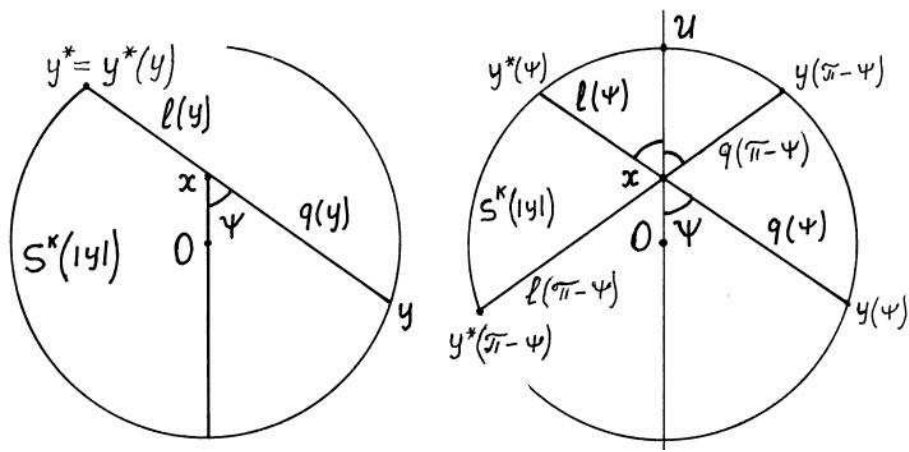,height=6cm}
    \caption{Exchanging $l$ and $q$ cumbersome.}\label{l,q_exchange}
\end{figure}

\begin{lemma}[$l, q$ - property]\label{Feature-Feature-lemma}

\begin{equation}\label{Feature-Feature-Lemma-formula}
    l(\psi)=q(\pi-\psi)\quad\text{and}\quad q(\psi)=l(\pi-\psi)\,.
\end{equation}

\end{lemma}

\begin{proof}[Proof of Lemma \ref{Feature-Feature-lemma}] Using the notation described above and the right picture from the Figure~\ref{l,q_exchange} above, we have the following sequence of implications:
\begin{equation}\label{Feature-10}
\begin{split}
    & \angle y(\psi)xO=\psi\quad \Leftrightarrow\quad \angle y^*(\psi)xU=\psi\quad\Leftrightarrow
    \\& \Leftrightarrow\quad \angle y^*(\psi)xO=\pi-\psi=\angle y(\pi-\psi)xO\quad\Rightarrow
    \\& \Rightarrow\quad|x-y^*(\psi)|=|x-y(\pi-\psi)| \quad\Leftrightarrow\quad l(\psi)=q(\pi-\psi)\,.
\end{split}
\end{equation}
The same argument shows that $q(\psi)=l(\pi-\psi)$. This completes the proof of Lemma~\ref{Feature-Feature-lemma}.
\end{proof}

\subsection {Change of variables.}

\smallskip

The following Lemma describes some important for integration exchanging rules. First let us introduce some notation necessary to state the Lemma.

\textbf{Notation:}
\begin{description}
  \item[$S^k(P;R)=S^k_P(R)$] is the $k$-dimensional sphere of radius $R$ centered \\ at $P\in\mathbb{R}^{k+1}$;
  \item[$S^k(R)=S^k(O;R)=S^k_O(R)$] is the $k$-dimensional sphere of radius $R$ centered at the origin $O$;
  \item[$x,y$] are two fixed points in $\mathbb{R}^{k+1}$, such that $r=|x|<|y|=R$;
  \item[$\Sigma=S^k(x;1)$] is the $k$-dimensional unit sphere centered at $x$;
  \item[$\psi=\angle Oxy$] and $\theta=\pi-\angle xOy$;
  \item[$\widetilde{y}=\Sigma\cap xy$].
\end{description}

\begin{figure}[!h]
    \centering
    \epsfig{figure=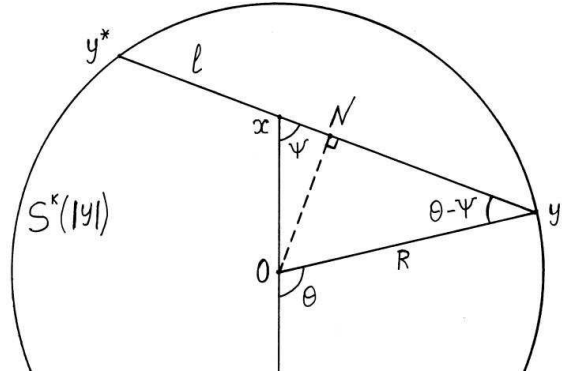,height=5cm}\\
    \caption{$\theta\leftrightarrow \psi$ Exchange}\label{Theta_Psi_Exchange}
\end{figure}

\begin{lemma}[The integration exchanging rules]\label{VarExchange-Lemma}

\begin{equation}\label{VarExchange-1}
  \text{\textbf{(A)}}\quad d\theta=\frac{2q}{l+q}d\psi\quad\text{and}\quad dS_y=\frac{2R}{l+q}q^k d\Sigma_{\widetilde{y}}\,,
\end{equation}
where $dS_y$ and $d\Sigma_{\widetilde{y}}$ are the volume elements of $S^k(R)$ and $\Sigma$ respectively.
\begin{equation}\label{VarExchange-2}
    \text{\textbf{(B)}}\quad\quad\quad\quad\quad\quad\quad  \int\limits_{\Sigma}f(l,q)d\Sigma=\int\limits_{\Sigma}f(q,l)d\Sigma
\end{equation}
for any integrable on $[0,\pi]$ complex-valued function $\widetilde{f}(\psi)=f(l(\psi), q(\psi))$.
\end{lemma}

\begin{proof}[Proof of \eqref{VarExchange-1}.]

Using the elementary geometry and the picture above, we observe that $\angle xyO=\theta-\psi$. The law of sines applied to the triangle $\triangle xOy$, gives
\begin{equation}\label{VarExchange-3}
|x|\, \sin(\psi)=R\, \sin(\theta-\psi)\,.
\end{equation}
Differentiation with respect to $\theta$ and $\psi$ yields
\begin{equation}\label{VarExchange-4}
    d\theta=\frac{|x|\, \cos(\psi)+R\, \cos(\theta-\psi)}
    {R\, \cos(\theta-\psi)}\, d\psi\,.
\end{equation}
Again, look at the picture above and observe that if $N$ is the orthogonal projection of the origin $O$ to the chord $yy^*$, then $N$ must be the midpoint for the chord $yy^*$. Therefore,
\begin{equation}\label{VarExchange-4-1}
    |Ny|=R\cos(\theta-\psi)=\frac{l+q}{2}\,,
\end{equation}
which is the denominator in \eqref{VarExchange-4}. Note also that
\begin{equation}\label{VarExchange-4-2}
    q=|x-y|=|xN|+|Ny|=|x|\cos\psi+R\cos(\theta-\psi)\,,
\end{equation}
which is precisely the numerator in \eqref{VarExchange-4}.
Therefore, combining \eqref{VarExchange-4}, \eqref{VarExchange-4-1} and \eqref{VarExchange-4-2}, we have
\begin{equation}\label{VarExchange-5}
    d\theta=\frac{2q}{q+l}\, d\psi\,,
\end{equation}
and then, the first formula in \eqref{VarExchange-1} is complete.

To prove the second formula in \eqref{VarExchange-1}, let us introduce some additional notation listed and pictured on Figure~\ref{dS_dSigma_Exchange} below.

\textbf{Additional notation:}
\begin{description}
  \item[$\sigma_k$] is the volume of a $k$-dimensional unit sphere;
  \item[$P_y$] and $P_{\widetilde{y}}$ are the orthogonal projections of $y$ and $\widetilde{y}$ respectively to line $Ox$;
  \item[$H(y)$] and $H(\widetilde{y})$ are the $k$-dimensional hyperplanes passing through $y$ and $\widetilde{y}$ respectively and orthogonal to $Ox$;
  \item[$S^{k-1}(P_y; |y P_y |)$]$=H(y)\cap S^k(|y|)$;
  \item[$S^{k-1}(P_{\widetilde{y}}; |\widetilde{y} P_{\widetilde{y}}|)$]$=H(\widetilde{y})\cap \Sigma$;
  \item[$dS_y^{k-1}(P_y; |yP_y|)$] is the volume element of $S^{k-1}(P_y;|yP_y|)$ at point $y$;
  \item[$dS_{\widetilde{y}}^{k-1}(P_{\widetilde{y}}; |\widetilde{y} P_{\widetilde{y}}|)$] is the volume element of $S^{k-1}(P_{\widetilde{y}}; |\widetilde{y} P_{\widetilde{y}}|)$ at point $\widetilde{y}$.
\end{description}

\begin{figure}[!h]
    \centering
    \epsfig{figure=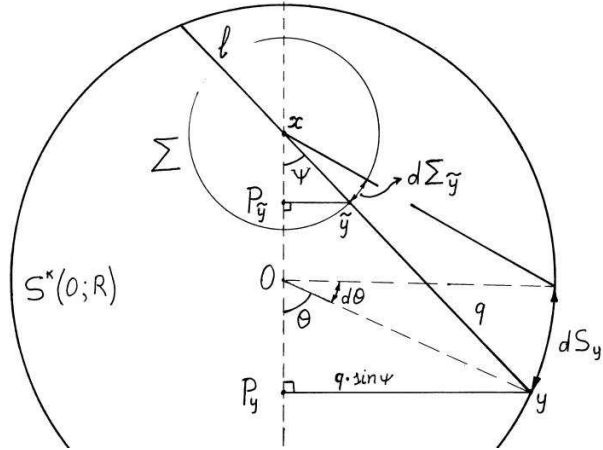,height=6cm}\\
    \caption{$dS\leftrightarrow d\Sigma$ Exchange}\label{dS_dSigma_Exchange}
\end{figure}

Note first that
\begin{equation}\label{VarExchange-6}
\begin{split}
    & dS_y=Rd\theta\cdot dS_y^{k-1}(P_y;|yP_y|)\quad\text{and}
    \\& d\Sigma_{\widetilde{y}}=
    d\psi\cdot dS_{\widetilde{y}}^{k-1}(P_{\widetilde{y}};|\widetilde{y}P_{\widetilde{y}}|) \,.
\end{split}
\end{equation}
On the other hand,
\begin{equation}\label{VarExchange-7}
    dS_y^{k-1}(P_y;|yP_y|)=
    q^{k-1}\cdot dS_{\widetilde{y}}^{k-1}(P_{\widetilde{y}};|\widetilde{y}P_{\widetilde{y}}|)\,.
\end{equation}
Therefore, combining \eqref{VarExchange-6}, \eqref{VarExchange-7} and the expression for $d\theta$ from \eqref{VarExchange-1}, we have the following sequence of equalities.
\begin{equation}\label{VarExchange-8}
\begin{split}
    dS_y
    & =Rd\theta\cdot dS_y^{k-1}(P_y;|yP_y|)=Rd\theta\cdot q^{k-1}\cdot dS_{\widetilde{y}}^{k-1}(P_{\widetilde{y}};|\widetilde{y}P_{\widetilde{y}}|)
    \\& =\frac{2R}{l+q}q^{k-1}d\psi\cdot dS_{\widetilde{y}}^{k-1}(P_{\widetilde{y}};|\widetilde{y}P_{\widetilde{y}}|)
    =\frac{2R}{l+q}q^{k-1}d\Sigma_{\widetilde{y}}\,,
\end{split}
\end{equation}
which completes the proof of \eqref{VarExchange-1}.
\end{proof}

\begin{proof}[Proof of \eqref{VarExchange-2}.]

Note that the function $f(l,q)=f(l(\psi), q(\psi))$ depends only on angle $\psi$ pictured above. Therefore, if we introduce
\begin{equation}\label{VarExchange-9}
    \Sigma^{k-1}(\psi)=\{\widetilde{y}\in\Sigma\mid\angle \widetilde{y}xO=\psi\}\,,
\end{equation}
we may write
\begin{equation}\label{VarExchange-10}
\begin{split}
    \int\limits_{\widetilde{y}\in\Sigma}f(l,q)d\Sigma_{\widetilde{y}}
    & =\int\limits_0^\pi d\psi\int\limits_{\widetilde{y}\in\Sigma^{k-1}(\psi)}f(l,q)d\Sigma_{\widetilde{y}}^{k-1}(\psi)
    \\& =\int\limits_0^\pi f(l,q)(\sin\psi)^{k-1}\sigma_{k-1}d\psi\,,
\end{split}
\end{equation}
since $f(l,q)$ remains constant while $\widetilde{y}\in\Sigma^{k-1}(\psi)$ and $|\Sigma^{k-1}(\psi)|=\sigma_{k-1}(\sin\psi)^{k-1}$ is the volume of $\Sigma^{k-1}(\psi)$.
 Using \eqref{VarExchange-10}, then  \eqref{Feature-Feature-Lemma-formula} from Lemma \eqref{Feature-Feature-lemma}, p.~\pageref{Feature-Feature-lemma}, and the following change of variables $\widetilde{\psi}=\pi-\psi$, we have the following sequence of equalities.
\begin{equation}\label{VarExchange-11}
\begin{split}
    & \int\limits_{\Sigma}f(l,q)d\Sigma=\int\limits_{\psi=0}^{\psi=\pi} f(l(\psi),q(\psi))\sigma_{k-1}(\sin\psi)^{k-1}d\psi
    \\& =\int\limits_0^\pi f(q(\pi-\psi), l(\pi-\psi))\sigma_{k-1}(\sin\psi)^{k-1}d\psi
    \\& =-\int\limits_\pi^0 f(q(\widetilde{\psi}), l(\widetilde{\psi}))\sigma_{k-1}(\sin\widetilde{\psi})^{k-1}d\widetilde{\psi}
    \\& =\int\limits_{\widetilde{\psi}=0}^{\widetilde{\psi}=\pi} f(q(\widetilde{\psi}), l(\widetilde{\psi})) \sigma_{k-1}(\sin\widetilde{\psi})^{k-1}d\widetilde{\psi}=\int\limits_\Sigma f(q,l)d\Sigma\,,
\end{split}
\end{equation}
which completes the proof of \eqref{VarExchange-2}, and the proof of Lemma \eqref{VarExchange-Lemma}.
\end{proof}

\subsection {The Differentiation of $l/q$ and $l+q$ \\ with respect to $\psi$.}

\begin{lemma}\label{l/q-differentiation-lemma}

\begin{equation}\label{Differention-1}
    \frac{d\omega}{d\psi}=\frac{d}{d\psi}\left(\frac{l}{q}\right)=\frac{l}{q}\cdot\frac{4r\sin\psi}{l+q}\,,
\end{equation}
where $q,l, \psi$ were defined in \eqref{Feature-2} and $\theta=\pi-\angle xOy$.
\end{lemma}

\begin{proof}[Proof of Lemma \ref{l/q-differentiation-lemma}] Recall that for $R=|y|>|x|=r$,
\begin{equation}\label{Differention-2}
    \omega(x,y)=\frac{R^2-r^2}{R^2+r^2+2Rr\cos\theta}=\frac{l}{q}\,,
\end{equation}
and then, since $R^2-r^2=lq$, $R^2+r^2+2Rr\cos\theta=q^2$ and ${R\sin\theta=q\sin\psi}$, the direct computation yields
\begin{equation}\label{Differention-3}
    \frac{d\omega}{d\theta}=\frac{2rl\sin\psi}{q^2}\,.
\end{equation}
Therefore, using \eqref{Differention-3} and the expression for $d\theta$ from \eqref{VarExchange-1}, p.~\pageref{VarExchange-1}, we have
\begin{equation}\label{Differention-4}
    \frac{d\omega}{d\psi}=\frac{d\omega}{d\theta}\frac{d\theta}{d\psi}
    =\frac{2rl\sin\psi}{q^2}\cdot\frac{2q}{l+q}
    =\frac{l}{q}\cdot\frac{4r\sin\psi}{l+q}\,,
\end{equation}
which is precisely what was stated in Lemma \eqref{l/q-differentiation-lemma}.
\end{proof}

\begin{corollary}
\begin{equation}\label{Substitution-for-d-psi}
    d\psi=\frac{l+q}{4r\sin\psi}d\ln\omega\,.
\end{equation}
\end{corollary}

\begin{proof}
Formula \eqref{Substitution-for-d-psi} is the direct consequence of~\eqref{Differention-1}.
\end{proof}

\begin{lemma}\label{(l+q)-Differentiation-Lemma}
\begin{equation}\label{(l+q)-Differentiation-formula}
    \frac{d(l+q)}{d\psi}=-\frac{2|x|^2\sin(2\psi)}{l+q}\,.
\end{equation}
\end{lemma}

\begin{proof}
 Using Figure~\ref{Theta_Psi_Exchange}, p.~\pageref{Theta_Psi_Exchange} and the Pythagorean theorem we can observe that
 $$\left(\frac{l+q}{2}\right)^2=\rho^2-|x|^2\sin^2\psi\,,$$
 which leads to~\eqref{(l+q)-Differentiation-formula}. This completes the proof of Lemma~\ref{(l+q)-Differentiation-Lemma}.
\end{proof}

\subsection{Some useful inequalities involving $l$ and $q$.}

In this subsection we obtain some inequalities used in the Chapter related to Hyperbolic Geometry, see the proof of Proposition~\ref{Limit-for-alpha-(0,k)}, p.~\pageref{Proof-of-Limit-for-alpha-(0,k)} and proof of Proposition~\ref{Limit-for-alpha-outside-[0,k]}, p.~\pageref{Proof-of-Limit-for-alpha-outside-[0,k]}. So, let us introduce the notation we are going to use in the Hyperbolic Geometry topics. Let $\rho$ be the radius of a circle (sphere) centered at the origin $O$. Points $A$ and $u$ be arbitrary points from the boundary of $S^k(\rho)$ and $m$ be an arbitrary point from segment $OA$.
Angle $\psi=\angle Omu$ and $\theta=\pi-\angle mOu$. $\eta=|Om|$ and $u^*$ is defined as the second intersection of line $um$ and $S^k(\rho)$. Then $q=|um|$ and $l=|u^*m|$. All of these notations will be pictured on Figure~\ref{MInimal Chord-picture} below.

\begin{lemma}\label{Minimum-for-(l+q)}
    Let $\psi_0\in(0,\pi/2)$, $D=(0,\rho]\times(0,\psi_0]\quad\text{and}\quad (\eta,\psi)\in D$. Then
    $\min\limits_{D}(l+q)=2\rho\cos\psi_0$.
\end{lemma}

\begin{figure}[!h]
    \centering
    \epsfig{figure=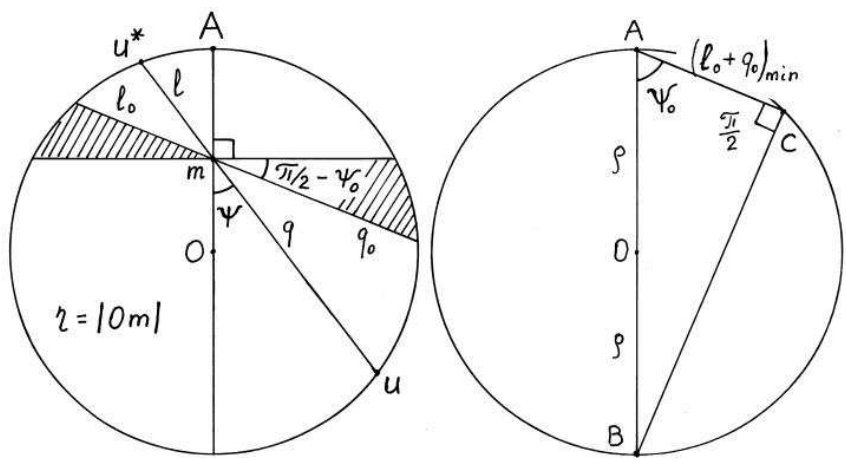,height=6cm}\\
    \caption{The minimal chord $l+q$.}\label{MInimal Chord-picture}
\end{figure}

\begin{proof}[Proof of Lemma \ref{Minimum-for-(l+q)}]
    Note that the length of every chord in a circle depends only on the distance between the chord and the center of the circle. The bigger distance, the smaller chord.  Observe also that the distance between the center $O$ and the chord $u^*u$ is $\eta\sin\psi$.
Therefore, the bigger $\eta$ or the bigger $\psi$, the smaller $|u^*u|=l+q$. Thus, if we set
\begin{equation}
    q_0=q_0(m)=q(m,\psi_0)\quad\text{and}\quad l_0=l_0(m)=l(m,\psi_0)
\end{equation}
and observe that $\angle ACB$ on Figure~\ref{MInimal Chord-picture} is equal to $\pi/2$, we can write the following sequence of equalities.
\begin{equation}
\begin{split}
    & \min\limits_{D}(l+q)=\min\limits_{\eta\in(0,\rho]}\min\limits_{\psi\in(0,\psi_0]}(l+q)=\min\limits_{\eta\in(0,\rho]}(l_0+q_0)
    \\& =l_0(A)+q_0(A)=q_0(A)=2\rho\cos\psi_0 \,,
\end{split}
\end{equation}
which completes the proof of Lemma~\ref{Minimum-for-(l+q)}.
\end{proof}


The following Lemma will be used in the proof of Proposition~\ref{Limit-for-alpha-outside-[0,k]}, p.~\pageref{Proof-of-Limit-for-alpha-outside-[0,k]}.

\begin{lemma}\label{l-q-estimation-Lemma} For every $(|m|,\psi)\in[0,\rho]\times[2\pi/3,\pi]$,
\begin{equation}\label{l-q-estimation-formulae}
   l\geq\rho\quad\text{and}\quad q\leq2(\rho-\eta)\,.
\end{equation}
\end{lemma}

\begin{proof}[Proof of Lemma~\ref{l-q-estimation-Lemma}.]
In the proof of Lemma~\ref{l-q-estimation-Lemma} we are going to use Figure~\ref{l-q-Estimation-Picture} below.

\begin{figure}[!h]
    \centering
    \epsfig{figure=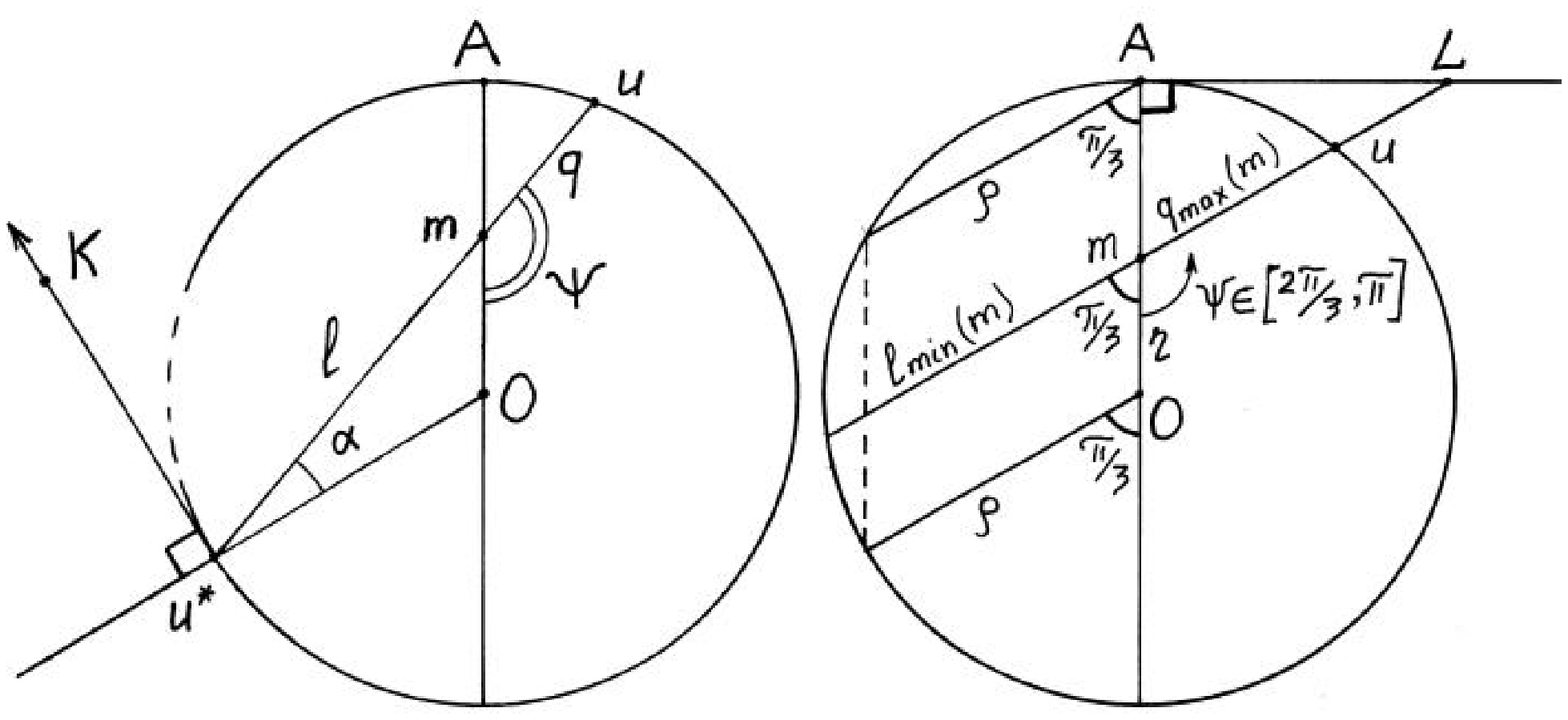,height=6cm}\\
    \caption{Estimation for $l$ and $q$.}\label{l-q-Estimation-Picture}
\end{figure}

Let $\alpha=\angle mu^*O$ and $u^*K$ defines the ray perpendicular to $u^*O$ and directed to the upper-half plane. Fix $\eta=|m|\in(0,\rho]$ and observe that for every $\psi\in(0,\pi)$, angle $\alpha>0$. Therefore, $\angle Ku^*m=\pi/2-\alpha$ must be acute, which implies that $l=|u^*O|$ is strictly decreasing as long as $\psi$ is decreasing. Thus, for every fixed $\eta=|m|\in(0,\rho]$, $l$ attains its minimum at $\psi=2\pi/3$. By the same argument, for every fixed $\eta\in(0,\rho)$, $q=|mu|$ attains its maximum at $\psi=2\pi/3$.

Using the right part of Figure~\ref{l-q-Estimation-Picture} we can observe that for every $m\in OA$,
\begin{equation}\label{l-q-estimation-formulae-1}
    l(m, 2\pi/3)\geq\rho\quad\text{and}\quad q(m,2\pi/3)\leq  |mL|\,.
\end{equation}
Combining \eqref{l-q-estimation-formulae-1} with the previous observation saying that $l(\psi)$ is increasing and $q(\psi)$ is decreasing for $\psi\in[0,\pi]$ and for every fixed $\eta=|m|\in[0,\rho]$, we have
\begin{equation}
    \min\limits_{\psi\in[2\pi/3, \pi]} l(m,\psi)\geq l_{\min}(m)=l(m, 2\pi/3)\geq\rho
\end{equation}
as well as
\begin{equation}
    \max\limits_{\psi\in[2\pi/3, \pi]} q(m,\psi)\leq q_{\max}(m)=q(m, 2\pi/3)\leq|mL|=2(\rho-\eta)\,,
\end{equation}
which completes the proof of Lemma~\ref{l-q-estimation-Lemma}.
\end{proof}


\chapter {Main Theorem.}

The main goal of this chapter is to describe the constraints for complex numbers $\alpha,\beta$ and for a fixed point $x\in\mathbb{R}^{k+1}\setminus\{S^k(R)\}$ under which the following equivalence holds.
\begin{equation}\label{Chapter-goal}
    \alpha+\beta=k\,\,\,\text{or}\,\,\,\alpha=\beta\quad\Longleftrightarrow\quad\int\limits_{S^k(R)}\omega^{\alpha}dS_y=
    \int\limits_{S^k(R)}\omega^{\beta}dS_y\,,
\end{equation}
where $\omega=\omega(x,y)$ is defined below in \eqref{Sphe-Pro-Statement-1}.

\begin{theorem}[Main Theorem]\label{Basic-theorem}

Let $S^k$,
$k\in\mathbb{N}$, be a $k$-dimensional sphere of radius $R$
centered at the origin $O$, $x, y\in\mathbb{R}^{k+1}$ such that $r=|x|\neq|y|=R$ and
\begin{equation}\label{Sphe-Pro-Statement-1}
    \omega(x,y)=\frac{|\,|x|^2-|y|^2|}{|x-y|^2}\,.
\end{equation}
Until further notice we assume that $R$ and a point $x\notin S^k$ are fixed.
Then, the following statements hold.

\begin{description}

  \item[(A)]\label{Sphe-pro-direct-statement} If $\alpha, \beta\in\mathbb{C}$ and $\alpha+\beta=k$, then
\begin{equation}\label{Sphe-Pro-Statement-2}
    \int\limits_{S^k}\omega^\alpha dS_y=
    \int\limits_{S^k}\omega^\beta dS_y\,.
\end{equation}

  \item[(B)]\label{sphe-pro-inverse-real-statement} If $\alpha, \beta$ are real, then
\begin{equation}\label{Sphe-Pro-Statement-3}
     \int\limits_{S^k}\omega^\alpha dS_y=
    \int\limits_{S^k}\omega^\beta dS_y\quad \text{implies}\quad \alpha+\beta=k\,\,\,
    \text{or}\,\,\,\alpha=\beta  \,.
\end{equation}

  \item[(C)]\label{sphe-pro-liuville-statement} For every $\beta\in \mathbb{C}$ there are infinitely many numbers $\alpha\in\mathbb{C}$ such that
\begin{equation}\label{Sphe-Pro-Statement-4}
    \int\limits_{S^k}\omega^\alpha dS_y=
    \int\limits_{S^k}\omega^\beta dS_y\,.
\end{equation}

    \item[(D)]\label{sphe-pro-inverse-complex-statement} Suppose that for the fixed point $x\notin S^k$
\begin{equation}\label{Sphe-Pro-Statement-5-1}
    \max\{|\Im(\alpha)|, |\Im(\beta)|\}\leq \left.\frac\pi2 \right/ \ln\frac{R+r}{|R-r|}\,.
\end{equation}
Then
\begin{equation}\label{Sphe-Pro-Statement-6}
\int\limits_{S^k} \omega^\alpha dS_y =
\int\limits_{S^k} \omega^\beta dS_y \quad\text{implies}\quad
\alpha+\beta=k\quad\text{or}\quad\alpha=\beta \,.
\end{equation}

\end{description}
\end{theorem}

\begin{remark}
    Statement~(C) shows that the implication introduced in \eqref{Sphe-Pro-Statement-3} fails if we just let $\alpha$ and $\beta$ be complex. However, a certain additional restriction formulated in Statement~(D) for the complex numbers $\alpha, \beta$ allows us to obtain precisely the same implication as in \eqref{Sphe-Pro-Statement-3}.
\end{remark}

\begin{remark}
    Observe that Statement~(B) is a special case of Statement~(D). Indeed, if $\alpha, \beta$ are real, then \eqref{Sphe-Pro-Statement-5-1} holds for every $r\neq R$.
\end{remark}

\begin{corollary}\label{imaginary-part-vanishing-corollary} If $\alpha=k/2+ib$, then
\begin{equation}\label{imaginary-part-vanishing}
    \int\limits_{S^k}\omega^\alpha dS_y=
    \int\limits_{S^k}\omega^{k/2}\cos(b\ln\omega) dS_y\,.
\end{equation}
Moreover, for every natural number $m=0,1,2,3,...$
\begin{equation}\label{logarithm-integral}
\begin{split}
    & \int\limits_{S^k}\omega^{k/2}(\ln\omega)^{2m+1}\cos(b\ln\omega) dS_y=0\,,
    \\& \int\limits_{S^k}\omega^{k/2}(\ln\omega)^{2m}\sin(b\ln\omega) dS_y=0
\end{split}
\end{equation}
and for every $\alpha\in\mathbb{C}$
\begin{equation}\label{Taylor-Decomposition}
    F(\alpha)=\int\limits_{S^k}\omega^{\alpha}dS_y=\sum\limits_{m=0}^{\infty}\frac{(\alpha-k/2)^{2m}}{(2m)!}
    \int\limits_{S^k}\omega^{k/2}(\ln\omega)^{2m}dS_y\,,
\end{equation}
which is the Taylor decomposition of $F(\alpha)$ at $\alpha=k/2$.
\end{corollary}

\begin{corollary}\label{sphe-pro-distance-basic-statement}
If $\alpha+\beta=2k$ and $\alpha,\beta\in\mathbb{C}$, then
\begin{equation}\label{Sphe-Pro-Statement-9.0}
\int\limits_{S^k} \frac{dS_y}{|x-y|^\alpha}
=|R^2-|x|^2|^{(\beta-\alpha)/2}\, \int\limits_{S^k}
\frac{dS_y}{|x-y|^\beta}\,.
\end{equation}
\end{corollary}

\begin{remark} It is clear that formula \eqref{Sphe-Pro-Statement-9.0} can be obtained from Statement~(A) if we rewrite formula \eqref{Sphe-Pro-Statement-2} in terms of distance between $x$ and $y$. In particular, for $\alpha=k-1$ and for every $x\notin S^k$, formula~\eqref{Sphe-Pro-Statement-9.0} yields
\begin{equation}\label{Sphe-Pro-Statement-9}
    \int\limits_{S^k} \frac{dS_y}{|x-y|^{k-1}}
    = \int\limits_{S^k}
    \frac{|R^2-|x|^2|}{|x-y|^{k+1}}\,dS_y\,.
\end{equation}
Note that for $k>1$ the integrand on the left is the Newtonian
potential in $\mathbb{R}^{k+1}$ and the integrand on the right is
exactly the Poisson kernel in $\mathbb{R}^{k+1}$.
\end{remark}

\begin{corollary}\label{One-dimension-sphe-pro} If $k=1$, then for every $p\in\mathbb{C},a,b
\in\mathbb{R}$ and $a>b>0$,
\begin{equation}\label{Sphe-Pro-Statement-10}
    \int\limits_0^{2\pi} \frac{d\theta}{(a-b\,\sin(\theta))^{p}}=
    (a^2-b^2)^{1/2-p}\, \int\limits_0^{2\pi}
    (a-b\,\sin(\theta))^{p-1}d\theta\,.
\end{equation}
\end{corollary}

\begin{corollary}\label{Basic-Th-Integrable-Case-general-Item}
If $k=2$, then
\begin{equation}\label{Basic-Th-Integrable-Case-1}
    \int\limits_{S^2(R)}\omega^{1+ib}dS_y=\frac{2\pi R}{r}
    \cdot\frac{|R^2-r^2|}{b}\cdot\sin\left(b\ln\frac{R+r}{|R-r|}\right)\,.
\end{equation}
In particular, if $b=0$,
\begin{equation}\label{Basic-Th-Integrable-Case-2}
    \int\limits_{S^2(R)}\omega\, dS_y=\frac{2\pi R}{r}\cdot|R^2-r^2|\cdot\ln\frac{R+r}{|R-r|}\,.
\end{equation}
\end{corollary}

Let us start the proof of Main Theorem, p.~\pageref{Basic-theorem}.

\begin{proof}[\textbf{Proof of Statement (A)}] This proof consists of two steps described in Case 1 and Case 2 below. We are going to use the notation introduced in the previous chapter.

  \textbf{Case 1 ($x$ is inside the sphere $S^k(|y|)$, i.e., $|x|<R$).}  Let $\alpha, \beta\in\mathbf{C}$ be such that $\beta=k-\alpha$. Using formulae \eqref{VarExchange-1} and \eqref{VarExchange-2}, p.~\pageref{VarExchange-1}, we have the following sequence of equalities.
\begin{equation}\label{A-SphePro-Proof-1}
\begin{split}
    & \int\limits_{S^k}
    \omega^\alpha\,dS_y
    =\int\limits_{S^k}
    \left(\frac{l}{q}\right)^\alpha\,dS_y
     =\int\limits_{\Sigma}
    \left(\frac{l}{q}\right)^\alpha\, \frac{2R}{l+q}\, q^k\,
    d\Sigma_{\widetilde{y}}
    \\& =\int\limits_\Sigma f(l,q)d\Sigma_{\widetilde{y}}
    =\int\limits_\Sigma f(q,l)d\Sigma_{\widetilde{y}}
    \\& =\int\limits_{\Sigma}
    \left(\frac{q}{l}\right)^\alpha\, \frac{2R}{l+q}\, l^k\,
    d\Sigma_{\widetilde{y}}
    =\int\limits_{\Sigma}
    \left(\frac{l}{q}\right)^{-\alpha}\, \frac{l^k}{q^k}\,
    \frac{2R}{l+q}\, q^k\, d\Sigma_{\widetilde{y}}
    \\& =\int\limits_{S^k}
    \left(\frac{l}{q}\right)^{k-\alpha}\, dS_y
    =\int\limits_{S^k}
    \omega^{\beta}\,dS_y\,,
\end{split}
\end{equation}
where
\begin{equation}\label{A-SphePro-Proof-2}
    f(l,q)=\left(\frac{l}{q}\right)^\alpha\, \frac{2R}{l+q}\, q^k
\end{equation}
and the fourth equality in \eqref{A-SphePro-Proof-1} is obtained from \eqref{VarExchange-2}. Thus, the proof of statement (A) is complete for the case $|x|<R$. To finish the proof we have to consider the case when $|x|>R$.

\begin{remark} Besides the Geometric Interpretation, for some reason, Hermann Schwarz introduced angle $\theta^*=\pi-\angle xOy^*$, see \cite{Ahlfors}, p.~168. We shall see in the Appendix, Theorem~\ref{Schwartz-Angle-Theorem}, p.~\pageref{Schwartz-Angle-Theorem} that this angle $\theta^*$ could also be used to obtain an elementary proof of Statement~(A) of the Main Theorem. Unfortunately,
the direct substitution of $d\theta^*$ instead of $d\theta$ does not compute certain important integrals related to Electrostatics and to Hyperbolic Geometry, while using $d\psi$ instead of $d\theta$ does. For this reason the angle $\psi$ was chosen in the proof presented above.
\end{remark}

\begin{remark} This statement obtained in Case 1 can also be interpreted in terms of eigenvalues and eigenfunctions of the Hyperbolic Laplacian. We shall see this interpretation in Proposition~\ref{Statement-A'-Proposition}, p.~\pageref{Statement-A'-Proposition}.
\end{remark}
\textbf{Case 2 ($x$ is outside of the sphere $S^k(|y|)$, i.e., $|x|>R$).}\label{x-is-outside-proof-Main-Thm} This case can be reduced to Case 1 by Sphere Exchange Rule introduced in Lemma~\ref{Basic_lemma}, formula \eqref{ExchaSphere-2}, p.~\pageref{ExchaSphere-2}. Recall that $y\in S^k=S^k(R)$, where $R=|y|<|x|=r$. Let $x\in
S^k(r)$, which is the $k$-dimensional sphere of radius $r$
centered at the origin $O$. Then, using formula \eqref{A-SphePro-Proof-1} and Lemma~\ref{Basic_lemma}, formula \eqref{ExchaSphere-2}, we have
\begin{equation}\label{A-SphePro-Proof-3}
\begin{split}
    & \int\limits_{S^k} \omega^{\alpha}(x,y_1) d\,S_{y_1}=
    \left(\frac{|y|}{|x|}\right)^k \int\limits_{S^k(r)}
    \omega^{\alpha}(x_1,y) d\,S_{x_1}
    \\& =\left(\frac{|y|}{|x|}\right)^k \int\limits_{S^k(r)}
    \omega^{\beta}(x_1,y) d\,S_{x_1} = \int\limits_{S^k}
    \omega^{\beta}(x,y_1) d\,S_{y_1}\,.
\end{split}
\end{equation}
The above derivation completes the proof of (A) of Theorem~\ref{Basic-theorem}.
\end{proof}

\begin{proof}[\textbf{Proof of Corollary \ref{imaginary-part-vanishing-corollary}, p.~\pageref{imaginary-part-vanishing-corollary}.}]

According to Statement (A),
\begin{equation}\label{NYU-1}
    \int\limits_{S^k}\omega^{k/2+ib}dS_y=
    \int\limits_{S^k}\omega^{k/2-ib}dS_y\,.
\end{equation}
Comparing the imaginary parts of this identity  gives
\begin{equation}\label{imaginary-part-only}
\int\limits_{S^k}\omega^{k/2}\sin(b\ln\omega)dS_y=
    -\int\limits_{S^k}\omega^{k/2}\sin(b\ln\omega)dS_y\,,
\end{equation}
which implies that imaginary part is identically zero, and then, the proof of \eqref{imaginary-part-vanishing} is complete. The repeated differentiation with respect to $b$ of the imaginary part \eqref{imaginary-part-only} leads directly to \eqref{logarithm-integral}.

To proof the Taylor decomposition presented in \eqref{Taylor-Decomposition}, notice that $F(\alpha)$ is entire function and then for every $\alpha\in\mathbb{C}$,
\begin{equation}\label{Taylor-Deco-Proof-1}
    F(\alpha)=F(k/2)+F'(k/2)(\alpha-k/2)+\frac{F''(k/2)}{2!}(\alpha-k/2)^2+\cdots\,.
\end{equation}
The direct computation together with the first formula from \eqref{logarithm-integral}, p.~\pageref{logarithm-integral} yields
\begin{equation}\label{Taylor-Deco-Proof-2}
    F^{(2m+1)}(k/2)=\int\limits_{S^k}\omega^{k/2}(\ln\omega)^{2m+1}dS\equiv 0
\end{equation}
for $m=0,1,2,\ldots$ and
\begin{equation}\label{Taylor-Deco-Proof-3}
    F^{(2m)}(k/2)=\int\limits_{S^k}\omega^{k/2}(\ln\omega)^{2m}dS>0
\end{equation}
for $m=0,1,2,\cdots$. This completes the proof of the Corollary \ref{imaginary-part-vanishing-corollary}.
\end{proof}

\begin{proof}[\textbf{Proof of Statement (B), p.~\pageref{sphe-pro-inverse-real-statement}.}]

Fix $\alpha\in\mathbb{R}$. We have to show that the equation
\begin{equation}\label{2-r}
    F(\alpha)=\int\limits_{S^k} \omega^{\alpha}d\,S_y=\int\limits_{S^k}
    \omega^{\lambda}d\,S_y=F(\lambda)
\end{equation}
has only two solutions with respect to the variable $\lambda$,
namely, $\lambda=\alpha$ and $\lambda=k-\alpha$. First note that,
since $\lambda+(k-\lambda)=k$, we have $F(\lambda)=F(k-\lambda)$
for every $\lambda\in \mathbb{R}$. Therefore, the function
$F(\lambda)$ is symmetric with respect to the point $\lambda=k/2$.
One can easily check that the Leibnitz rule of differentiation
under the integral sign, see \cite{Sokolnikoff}, (p.~121) is applicable to
$F(\lambda)$. Hence,

\begin{equation}\label{3-r}
    \frac{dF(\lambda)}{d\lambda}=\int\limits_{S^k}
    \omega^{\lambda}\, (\ln\omega)\,d\,S_y\,.
\end{equation}
We apply the same rule one more time to obtain
\begin{equation}\label{4-r}
    \frac{d^2F(\lambda)}{(d\lambda)^2}=\int\limits_{S^k}
    \omega^{\lambda}\, (\ln\omega)^2\,d\,S_y >0\quad
    \text{for every}\,\,\,\lambda\in \mathbb{R}\,.
\end{equation}
Therefore, the graph of our function $\xi=F(\lambda)$, plotted in
the $(\lambda,\xi)$-plane, is convex and symmetric with respect to
the line $\lambda=k/2$. Therefore, $\lambda=k/2$ is the
minimum point for $F(\lambda)$. Thus, for every $\nu >
F(k/2)$, the equation $F(\lambda)=\nu$ has only two solutions
$\lambda=\alpha$ and $\lambda=k-\alpha$. This completes the proof of Statement~(B) of the Main Theorem, p.~\pageref{Basic-theorem}.
\end{proof}

\begin{proof}[\textbf{Proof of Statement (C), p.~\pageref{sphe-pro-liuville-statement}.}]

Let $F(\alpha)=\int\limits_{S^k} \omega^\alpha dS_y$, where $\alpha=u+iv$. Then, the statement (C) can be obtained as a consequence of the following Picard's Great Theorem.

\begin{theorem}[Picard's Great Theorem]\label{Picard-Great-theorem}
 Let $c\in\mathbf{C}$ be an isolated essential singularity of $f$. Then, in every neighborhood of $c$, $f$ assumes every complex number as a value with at most one exception infinitely many times, see \cite{Reinhold}, (p.240).
\end{theorem}

Indeed, the direct computation shows that $F$ is an entire function, which means that $\infty$ is an isolated singularity of $F$.
Note that the Taylor decomposition presented in \eqref{Taylor-Decomposition}, p.~\pageref{Taylor-Decomposition} shows that $\infty$ is the essential singularity for $F(\alpha)$, since all even Taylor coefficients are positive.

It is clear also, we can not meet an exception mentioned in Picard's theorem above, since we are looking for complex numbers $\lambda$, satisfying $F(\lambda)=F(\alpha)$ for some $\alpha$. This means that the value $F(\alpha)$ is already assumed by $F$, and then, by Picard's Great Theorem, $F$ must attain this value infinitely many times in every neighborhood of an isolated essentially singular point, i.e., at every neighborhood of $\infty$, in our case. This completes the proof of Statement~(C).
\end{proof}

\begin{proof}[\textbf{Proof of Statement (D), p.~\pageref{sphe-pro-inverse-complex-statement}.}]\label{Level-Curve-Story-Proof}

Denote
\begin{equation}\label{E-basic-proof-1}
    \Upsilon=\Upsilon(x)=\min\left\{\frac{R}{|x|},\frac{|x|}{R}\right\}.
\end{equation}
Recall that $R$ and $x\notin S^k(R)$ are fixed. Then,
\begin{equation}\label{E-basic-proof-2}
    \omega(x,y)=\omega(\theta)=\frac{1-\Upsilon^2}{1+\Upsilon^2+2\Upsilon\cos\theta}\,,
\end{equation}
where $\theta=\pi-\angle xOy$, $\theta\in[0,\pi]$ and it is clear that $\Upsilon<1$ for every possible $x\notin S^k(R)$. Observe also that $\omega(\theta)$ is increasing while $\theta\in[0,\pi]$. Therefore,
\begin{equation}\label{E-basic-proof-3}
    \frac{1-\Upsilon}{1+\Upsilon}\leq\omega(\theta)\leq\frac{1+\Upsilon}{1-\Upsilon}.
\end{equation}
Let $p\geq 0$ be an arbitrary non-negative number such that
\begin{equation}\label{introduction-of-p}
    \max\{|\Im(\alpha)|, |\Im(\beta)|\}\leq p \leq \left.\frac\pi2\right/\ln\frac{R+r}{|R-r|}=\left.\frac\pi2\right/\ln\frac{1+\Upsilon}{1-\Upsilon}\,.
\end{equation}
The direct computations show that the second inequality in~\eqref{introduction-of-p}, in itself, implies
\begin{equation}\label{E-basic-proof-4}
    e^{-\pi/2p}\leq\frac{1-\Upsilon}{1+\Upsilon}\quad\text{and}\quad
    e^{\pi/2p}\geq\frac{1+\Upsilon}{1-\Upsilon}
\end{equation}
Therefore, combining \eqref{E-basic-proof-3} and \eqref{E-basic-proof-4}, we have
\begin{equation}\label{E-basic-proof-5}
    e^{-\pi/2p}\leq\frac{1-\Upsilon}{1+\Upsilon}\leq\omega(\theta)\leq
    \frac{1+\Upsilon}{1-\Upsilon}\leq e^{\pi/2p}\,,
\end{equation}
or, equivalently,
\begin{equation}\label{E-basic-proof-6}
    -\frac{\pi}{2}\leq p\ln\omega\leq\frac{\pi}{2}\quad\text{or}\quad
    p\ln\omega\in\left[-\frac{\pi}{2},\frac{\pi}{2}\right]\,.
\end{equation}
If $\alpha=\xi+i\zeta$ and $|\zeta|\leq p$, then, clearly, \eqref{E-basic-proof-6} implies
\begin{equation}\label{E-basic-proof-7}
    -\frac{\pi}{2}\leq \zeta\ln\omega\leq\frac{\pi}{2}\quad\text{or}\quad
    \zeta\ln\omega\in\left[-\frac{\pi}{2},\frac{\pi}{2}\right]\,.
\end{equation}
Hence,
\begin{equation}\label{cosine-positivenes}
\cos(\zeta\ln\omega(\theta))\geq 0 \quad\text{for every}\quad \zeta\in[-p,p]
\end{equation}
and for every $\theta\in[0,\pi]$. Fix a number $p>0$ satisfying \eqref{cosine-positivenes} or, equivalently~\eqref{Sphe-Pro-Statement-5-1}, p.~\pageref{Sphe-Pro-Statement-5-1} and introduce the set
\begin{equation}\label{Strip-definition}
    \Xi=\{\xi+i\zeta\mid-p\leq\zeta\leq p\}\,.
\end{equation}
Consider the function
\begin{equation}\label{Patch-6}
    F(\alpha)=\int\limits_{S^k} \omega^\xi\cos(\zeta\ln\omega) dS_y +
    i\int\limits_{S^k} \omega^\xi\sin(\zeta\ln\omega) dS_y \,.
\end{equation}
Note now that Statement~(E) will be proven if we show that for any $\beta=a+ib\in\Xi$, the equation $F(\alpha)=F(\beta)$ has only two solutions in $\Xi$, i.e., $\alpha=\beta$ and $\alpha=k-\beta$. To accomplish this goal let us split our argument into the following four steps.
\begin{description}
  \item[Step 1.] Fix any $\beta=a+ib\in\Xi$ and set the equation $F(\alpha)=F(\beta)$ as the following system of two equations
      \begin{equation}\label{system_for_real_Imaginary}
        \Re F(\xi,\zeta)=\Re F(\beta)\quad\text{and}\quad \Im F(\xi,\zeta)=\Im F(\beta).
      \end{equation}
           Then, we establish a symmetry property for $\Re F(\xi,\zeta)$ and $\Im F(\xi,\zeta)$. We shall see that $\Re F(\xi,\zeta)$ is symmetric with respect to the two lines $\xi=k/2$ and $\zeta=0$, while $\Im F(\xi,\zeta)$ is skew-symmetric with respect to the same lines. This step is described on p.~\pageref{Step_curve_1}.
  \item[Step 2.] Here we describe all solutions inside $\Xi$ for the equation $\Re F(\xi,\zeta)=\Re F(\beta)$. Because of the symmetry mentioned in Step 1 it is enough to describe the set of solutions in the first quadrant $[k/2,\infty)\times[0,p]$. The description will be presented in Proposition~\ref{Level-curve-description}, page~\pageref{Level-curve-description}.


  \item[Step 3.] Here we shall study the behavior of $\Im F(\xi,\zeta)$ along any of the $\mathcal{C}^1$ curves from Step 2. Proposition~\ref{uniquiness-total-in-first-quadrant}, p.~\pageref{uniquiness-total-in-first-quadrant} shows that $\Im F(\xi,\zeta)$ is a strictly increasing function along any of the level curves inside the first quadrant $[k/2,\infty)\times[0,p]$, which implies that the system \eqref{system_for_real_Imaginary} has the unique solution in the first quadrant for any $\beta\in[k/2,\infty)\times[0,p]$, i.e., the solution is $\beta$ itself. This step is carried out on p.~\pageref{Step_curve_3}.
  \item[Step 4.] Here we study the behavior of $\Im F(\xi,\zeta)$ in the whole strip $\Xi$. We shall see that for every $\beta\in[k/2,\infty)\times[0,p]$ there exist only two solutions of the system \eqref{system_for_real_Imaginary}, i.e., $\alpha_1=\xi+i\zeta=\beta$ and $\alpha_2=k-\beta$. The same argument will show that for every $\beta\in\Xi$ there are at most two possible solutions of system~\eqref{system_for_real_Imaginary}, i.e., $\alpha_1=\beta$ and $\alpha_2=k-\beta$. This step is carried out on p.~\pageref{Step_curve_4}.
\end{description}
Now let us follow this plan described in the Steps 1-4 above.

\underline{\textbf{Step 1.}}\label{Step_curve_1} It is convenient to introduce the notation for the real and the imaginary parts of $F(\alpha)=F(\xi+i\zeta)=F(\xi,\zeta)$. Since $R=|y|$ and $x\notin S^k(R)$ are fixed, we can denote
\begin{equation}\label{Definition-W}
    W(\alpha)=W(\xi+i\zeta)=W(\xi,\zeta)=
    \int\limits_{S^k} \omega^\xi\cos(\zeta\ln\omega) dS_y
\end{equation}
and similarly,
\begin{equation}\label{Definition-I}
    I(\alpha)=I(\xi+i\zeta)=I(\xi,\zeta)=
    \int\limits_{S^k} \omega^\xi\sin(\zeta\ln\omega) dS_y \,.
\end{equation}
Fix some $\beta=a+ib\in\Xi$ and rewrite system \eqref{system_for_real_Imaginary} in the new notation.
\begin{equation}\label{system_for_real_Ima_new_notation}
    W(\xi,\zeta)=W(a,b)\quad\text{and}\quad I(\xi,\zeta)=I(a,b)\,,
\end{equation}
which is going to be solved. As we saw in Statement (A), $F(\alpha)=F(k-\alpha)$ for every $\alpha=\xi+i\zeta\in \mathbb{C}$ and then, the direct computation yields
\begin{equation}\label{Symmetry-Reations}
\begin{split}
    & W(\xi+i\zeta)=W(k-\xi,-\zeta)=W(\xi,-\zeta)=W(k-\xi,\zeta) \quad\text{and}
    \\& I(\xi+i\zeta)=I(k-\xi,-\zeta)=-I(\xi, -\zeta)=-I(k-\xi,\zeta)\,,
\end{split}
\end{equation}
which pictured on the diagram below, see Figure~\ref{Re-Symmetry}.

\begin{figure}[!h]
    \centering
    \epsfig{figure=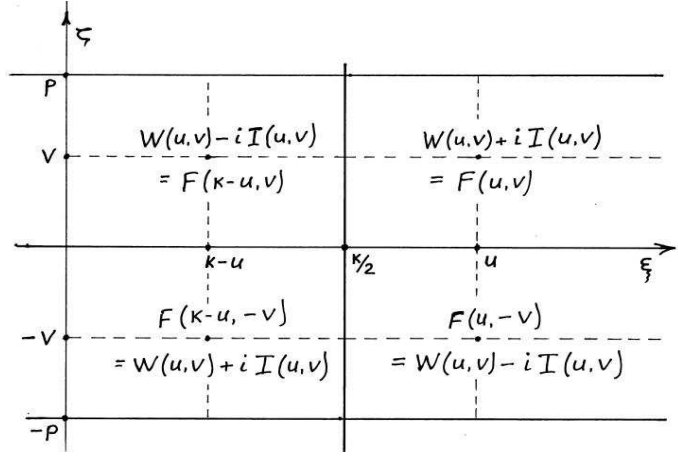,height=6cm,width=9cm}\\
    \caption{The symmetry of $\Re$.}\label{Re-Symmetry}
\end{figure}

The symmetry relations presented in \eqref{Symmetry-Reations} imply that $W(u+iv)$ is symmetric with respect to the two perpendicular lines $\xi=k/2$ and $\zeta=0$ in the plane $(\xi,\zeta)$. Similarly, $I(\xi,\zeta)$ is symmetric with respect to point $(k/2,0)$ and skew-symmetric with respect to the lines $\zeta=0$ and $\xi=k/2$.

\underline{\textbf{Step 2.}}\label{Step_curve_2} Because of the symmetry, without loss of generality, we may assume that $\beta=a+ib\in\{[k/2,\infty)\times[0,p]\}$, which is the upper-right quadrant. Now we are ready to solve the equation $F(\xi+i\zeta)=F(\beta)$. First we need to describe all solutions for the real parts of this equation, i.e., $\Re F(\xi,\zeta)=\Re F(\beta)$.
Let us define two quadrants $Q$, $\overline{Q}$ and the level set $S(a,b)$ by
\begin{equation}\label{Definition-Q-S(a,b)}
\begin{split}
    & Q=(k/2,\infty)\times(0,p)\,,\quad \overline{Q}=[k/2, \infty)\times[0,p],
    \\& S(a,b)=\{(\xi,\zeta)\in\overline{Q}\mid W(\xi,\zeta)=W(a,b)\}\,.
\end{split}
\end{equation}
It is clear that $\overline{Q}$ is the closure of $Q$.
The following proposition describes all possible types of such level sets.

\begin{proposition}[Level curves description]\label{Level-curve-description} \
    \begin{enumerate}
      \item For every fixed $\beta=a+ib\in\overline{Q}$ the set of solutions $S(a,b)$ for the equation
        \begin{equation}
            W(\xi,\zeta)=\int\limits_{S^k}\omega^\xi\cos(\zeta\ln\omega)dS_y=W(a,b)
        \end{equation}
        is a continuous curve $\gamma_{(a,b)}(t)=\gamma(t)=(t,v(t))\subseteq{\overline{Q}}$.
      \item The function $v(t)$ is $\mathcal{C}^{\infty}$ and $v'(t)>0$ for every point $t$ such that $(t,v(t))\subseteq{Q}$.
      \item No two of these curves have common points in $\overline{Q}$.
    \end{enumerate}
\end{proposition}

\begin{remark}
    The more detailed description of the level curves will be given in Appendix by Theorem~\ref{Level-curve-descri-Appendix}, p.~\pageref{Level-curve-descri-Appendix}, where we shall see that if a level curve is started at the corner $C=(k/2,0)$, then it bisects the corner and if a level curve is started at the lower or at the left edge of $\overline{Q}$, then it must be perpendicular to the edge, see Figure~\ref{Level-Curves-fig}, p.~\pageref{Level-Curves-fig}.
\end{remark}

\begin{proof}[\textbf{Proof of the Proposition \ref{Level-curve-description}}]

First, we need to analyze the behavior of partial derivatives of $W(\xi,\zeta)$ in $\overline{Q}$.
\begin{claim}\label{Partial-dW-d-xi-claim}
    \begin{equation}\label{Partial-dW-d-xi}
    \begin{split}
        & \frac{\partial W(\xi,\zeta)}{\partial\xi}>0\quad\text{for every}\quad(\xi,\zeta)\in
        \left(\frac{k}{2},\infty\right)\times[0,p]\,;
        \\&  \frac{\partial W(\xi,\zeta)}{\partial\xi}=0\quad\text{if}\quad\xi=
        \frac{k}{2}\quad\text{and}\quad\zeta\in[0,p]\,.
    \end{split}
    \end{equation}
\end{claim}

\begin{proof}[Proof of the claim \ref{Partial-dW-d-xi-claim}]
    Note that $W(\xi,\zeta)$ is $\mathcal{C}^{\infty}$ function of both variables and, according to \eqref{Symmetry-Reations}, $W(\xi,\zeta)$ is symmetric with respect to $\xi=k/2$ as a function of $\xi$ for every $\zeta$. This implies that
    \begin{equation}\label{horizontal-derivative-zero}
        \frac{\partial W(\xi,\zeta)}{\partial\xi}=0\quad\text{for}\quad\xi=
        \frac{k}{2}\quad\text{and for every}\quad\zeta\,.
    \end{equation}
    Note also that according to \eqref{cosine-positivenes}, p.~\pageref{cosine-positivenes},
    \begin{equation}\label{horizontal-second-derivative}
        \frac{\partial^2 W(\xi,\zeta)}{\partial\xi^2}=
        \int\limits_{S^k}\omega^{\xi}(\ln\omega)^2\cos(\zeta\ln\omega)dS_y>0
    \end{equation}
    for every $(\xi,\zeta)\in\Xi$ defined in \eqref{Strip-definition}. It is clear that \eqref{horizontal-derivative-zero} and \eqref{horizontal-second-derivative} imply \eqref{Partial-dW-d-xi}. This completes the proof of Claim~\ref{Partial-dW-d-xi-claim}.
    \end{proof}

    \begin{claim}\label{partial-vertical-claim}
      \begin{equation}\label{partial-vertical}
      \begin{split}
        & \frac{\partial W(\xi,\zeta)}{\partial\zeta}<0\quad\text{for every}\quad(\xi,\zeta)\in
        \left[\frac{k}{2},\infty\right)\times(0,p]\,;
        \\&  \frac{\partial W(\xi,\zeta)}{\partial\zeta}=0\quad\text{if}\quad\zeta=0
        \quad\text{and}\quad\xi\in[k/2, \infty]\,.
      \end{split}
      \end{equation}
    \end{claim}
    \begin{proof}[Proof of the Claim \ref{partial-vertical-claim}]
        Here we are going to use similar symmetry and second derivative arguments used in the proof of previous Claim. Recall that $W(\xi,\zeta)$ is $\mathcal{C}^{\infty}$ function of both variables and, according to \eqref{Symmetry-Reations}, $W(\xi,\zeta)$ is symmetric with respect to $\zeta=0$ as a function of $\zeta$ for every $\xi$. This implies that
    \begin{equation}\label{vertical-derivative-zero}
        \frac{\partial W(\xi,\zeta)}{\partial\zeta}=0\quad\text{for}\quad\zeta=0
        \quad\text{and for every}\quad\xi\,.
    \end{equation}
    Note also that according to \eqref{cosine-positivenes},
    \begin{equation}\label{vertical-second-derivative}
        \frac{\partial^2 W(\xi,\zeta)}{\partial\zeta^2}=
        -\int\limits_{S^k}\omega^{\xi}(\ln\omega)^2\cos(\zeta\ln\omega)dS_y<0
    \end{equation}
    for every $(\xi,\zeta)\in\Xi$ defined in \eqref{Strip-definition}, p.~\pageref{Strip-definition}. It is clear that \eqref{vertical-derivative-zero} and \eqref{vertical-second-derivative} imply \eqref{partial-vertical}. This completes the proof of Claim~\ref{partial-vertical-claim}.
    \end{proof}

    For the next step we need the following notation for the boundary parts of $\overline{Q}$. With every $\beta=a+ib\in\overline{Q}$ we associate two curves $\Gamma_0(a,b)$ and $\Gamma_1(a,b)$ as shown on the Figure~\ref{Special-boundary-curves} below and defined as follows.
\begin{equation}
\begin{split}
     \Gamma_0(a,b)= & \{(\xi,\zeta)\mid\xi=k/2\,\,\text{and}\,\,\zeta\in[0,b]\}\cup
    \\& \{(\xi,\zeta)\mid\xi\in[k/2,a]\,\,\text{and}\,\,\zeta=0\};
\end{split}
\end{equation}
\begin{equation}
    \Gamma_1(a,b)=\{(\xi,\zeta)\mid \xi\in[a,\infty)\,\,\text{and}\,\,\zeta=p\}\,.
\end{equation}

\begin{figure}[!h]
    \centering
    \epsfig{figure=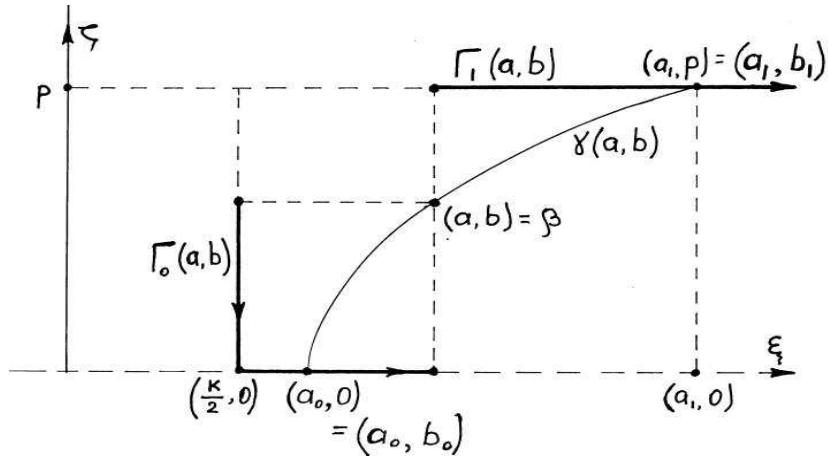,height=6cm,width=11cm}
  \caption{Special boundary curves.}\label{Special-boundary-curves}
\end{figure}

Note then, according to \eqref{Partial-dW-d-xi} and \eqref{partial-vertical}, the value of $W(\xi,\zeta)$ is a continuous and strictly increasing function as the point $(\xi, \zeta)$ runs along $\Gamma_0(a,b)$ from $(k/2,b)$ to $(a,0)$ and
\begin{equation}
    W(k/2,b)\leq W(\xi,\zeta)\leq W(a,0)\quad\text{for every}\,\,(\xi,\zeta)\in\Gamma_0(a,b)\,.
\end{equation}
On the other hand, $W(k/2,b)\leq W(a,b)\leq W(a,0)$.
Therefore, there exists the unique point $(a_0,b_0)\in\Gamma_0$ such that $W(a_0,b_0)=W(a,b)$.

Note also that by \eqref{partial-vertical}, $W(\xi,\zeta)$ is continuous and strictly decreasing along the vertical segment connecting $(a,b)$ and $(a,p)$. Thus, $W(a,p)\leq W(a,b)$. Then, when the point $(\xi,\zeta)$ runs along $\Gamma_1$ from $(a,p)$ to $\infty$, the value of $W(\xi,\zeta)$ is continuously and strictly increasing to $\infty$, according to~\eqref{Partial-dW-d-xi} and \eqref{horizontal-second-derivative}. Therefore, there exists the unique point $(a_1,p)\in\Gamma_1$ such that $W(a_1,p)=W(a,b)$. Clearly, $a_0\leq a\leq a_1$.

Now we are ready to describe the level set $S(a,b)$ for any fixed $(a,b)\in \overline{Q}$.

\begin{claim}\label{Level-set-boundary}
    \begin{equation}\label{rectangle-for-level-set}
        S(a,b)\subseteq[a_0,a_1]\times[0,p]\,.
    \end{equation}
\end{claim}

\begin{proof}[Proof of the Claim \ref{Level-set-boundary}]

Using \eqref{Partial-dW-d-xi} and \eqref{partial-vertical} we can observe that the value of $W(a_0,0)$ is the strict maximum of $W(\xi,\zeta)$ for $(\xi,\zeta)\in [k/2,a_0]\times[0,p]$ as well as the value of $W(a,p)$ is the strict minimum for $(\xi,\zeta)\in [a_1,\infty]\times[0,p]$. This completes the proof of Claim~\ref{Level-set-boundary}.
\end{proof}

\begin{claim}\label{function-presentation-claim} $S(a,b)$ can be interpreted as the graph of a function $v(t)$, i.e.,
\begin{equation}\label{function-presentation}
    S(a,b)=\{(t,v(t))\mid t\in[a_0,a_1]\}\,.
\end{equation}
\end{claim}

\begin{proof} Fix any $t\in[a_0,a_1]$. Then, for the vertical segment connecting two points $(t,0)$ and $(t,p)$ we can observe that
\begin{equation}
    W(t,0)\geq W(a,b)\geq W(t,p)
\end{equation}
since, according to \eqref{Partial-dW-d-xi}, the value of $W(\xi,\zeta)$ is a strictly increasing function along every horizontal line in $\overline{Q}$. Note also that the value of $W(t,\zeta)$, according to \eqref{partial-vertical}, is strictly decreasing for $\zeta\in[0,p]$. Hence, there exists the unique value of $\zeta=v(t)\in[0,p]$, such that $W(t,v(t))=W(a,b)$, which completes the proof of Claim \ref{function-presentation-claim}.
\end{proof}

To complete the proof of Proposition~\ref{Level-curve-description}, from p.~\pageref{Level-curve-description}, we need to study the functional properties of the function $v(t)$. First, let us observe that $v(t)>0$ for every $t\in(a_0,a_1)$, since the value of $W(\xi,\zeta)$, according to \eqref{Partial-dW-d-xi}, is a strictly increasing function along the lower edge of $\overline{Q}$. Therefore, by the Implicit Function Theorem, for every $t\in(a_0,a_1)$ there exists some neighborhood $N_n(t)=(t-\epsilon_n, t+\epsilon_n)$, where $v(t)$ is continuously differentiable $n$ times for any natural $n$ and then, according to \eqref{Partial-dW-d-xi} and \eqref{partial-vertical}, we have
      \begin{equation}\label{derivative-v-by-Implicit-FT}
        \frac{dv(t)}{dt}=\frac{W_\xi(t,v(t))}{-W_\zeta(t,v(t))}>0\quad\text{for every}\,\,t\in(a_0,a_1)\,.
      \end{equation}
      A possible curve $\gamma(t)=(t,v(t))$ starting with the horizontal edge of $\overline{Q}$ is sketched on Figure \ref{Special-boundary-curves}, page~\pageref{Special-boundary-curves}.
Note also that $v(t)$ is continuous in $[a_0,a_1]$, since $v(t)$ is monotone in $(a_0,a_1)$ and $W(\xi,\zeta)$ is continuous in $\overline{Q}$. Now the proof of the Proposition \ref{Level-curve-description} from p.~\pageref{Level-curve-description} is complete.
\end{proof}

\underline{\textbf{Step 3.}}\label{Step_curve_3} The goal of this step is to obtain the following proposition.

\begin{proposition}\label{uniquiness-total-in-first-quadrant}
   If $(a,b)\in\overline{Q}$, then the equation $F(\xi,\zeta)=F(a,b)$ has the unique solution $(\xi,\zeta)\in\overline{Q}$, i.e., $(\xi,\zeta)=(a,b)$.
\end{proposition}

       \begin{proof}[Proof of the Corollary \ref{uniquiness-total-in-first-quadrant}.]
        It is clear that a solution $(\xi,\zeta)$ may appear only on the level curve $S(a,b)=\gamma(t)$ described above. For a point $(\xi,\zeta)\in S(a,b)$ to be a solution we need $I(\xi,\zeta)=I(t,v(t))=I(a,b)$. So, let us study the behavior of $I(\xi,\zeta)$ along $\gamma(t)$. By the direct computation we can observe that
        \begin{equation}\label{Cauchy-Riemann}
            \frac{\partial W}{\partial \xi}=\frac{\partial I}{\partial\zeta}\quad\text{and}\quad
            \frac{\partial W}{\partial \zeta}=-\frac{\partial I}{\partial\xi}\,,
        \end{equation}
        which together with \eqref{derivative-v-by-Implicit-FT} leads to
        \begin{equation}
            \frac{d I(t,v(t))}{dt}=
            \left( 1+[v'(t)]^2\right)\left[-\frac{\partial W(t,v(t))}{\partial v(t)}\right]>0
        \end{equation}
        for every $t\in(a_0,a_1)$, where the last inequality holds because of \eqref{partial-vertical} from Claim~\ref{partial-vertical-claim}, p.~\pageref{partial-vertical-claim}. Therefore, the value of $I(t,v(t))$ is a strictly increasing function on $[a_0, a_1]$ and then, the function $I(t,v(t))$ assumes each of its values for $t\in[a_0,a_1]$ only once. Hence, the equation $I(\xi,\zeta)=I(a,b)$ must have only one solution in $S(a,b)$. This completes the proof of the Proposition \ref{uniquiness-total-in-first-quadrant}.
      \end{proof}

\underline{\textbf{Step 4.}}\label{Step_curve_4} The next and the last stage in the proof of Statement (D) is to analyze the behavior of $I(\xi,\zeta)=\Im F(\alpha)$ in $\Xi=\{(\xi,\zeta)\mid\zeta\in[-p,p]\}.$

\begin{proposition}\label{Signum-Behavior-of-I-propo}
    \begin{equation}\label{Signum-Behavior-of-I-1}
        I(\xi,\zeta)>0\quad\text{for}\quad(\xi,\zeta)\in \{(0,\infty)\times(0,p]\}\cup\{(-\infty,0)\times(0,-p]\};
    \end{equation}

    \begin{equation}\label{Signum-Behavior-of-I-2}
        I(\xi,\zeta)<0\quad\text{for}\quad(\xi,\zeta)\in \{(-\infty,0)\times(0,p]\}\cup\{(0,\infty)\times(0,-p]\}.
    \end{equation}
    The signum of $I$ behavior is sketched on Figure~\ref{Signum-Behavior-of-I} below.

  \begin{figure}[!h]
    \centering
    \epsfig{figure=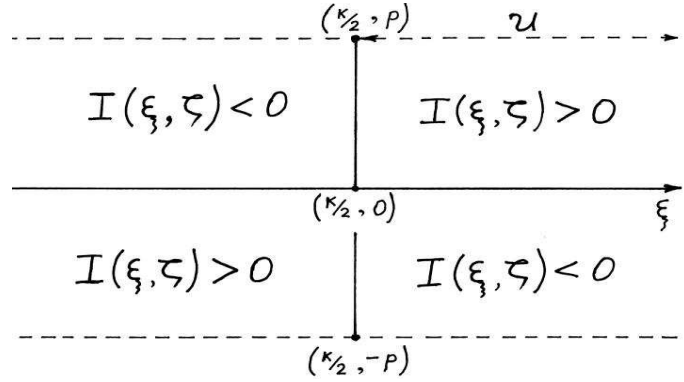,height=5cm,width=9cm}\\
    \caption{Signum behavior of $I(\xi,\zeta)$.}\label{Signum-Behavior-of-I}
\end{figure}
\end{proposition}

\begin{proof}[Proof of Proposition \ref{Signum-Behavior-of-I-propo}.]
    Using the skew-symmetry of $I(\xi,\zeta)$ introduced in \eqref{Symmetry-Reations}, p.~\pageref{Symmetry-Reations}, we can observe that
\begin{equation}\label{cross-imaginary-vanishing}
    I(\xi,\zeta)=\int\limits_{S^k}\omega^\xi\sin(\zeta\ln\omega)dS_y=0\quad\text{for}\quad
    \xi=0\,\,\text{or}\,\,\zeta=k/2\,.
\end{equation}
On the other hand, we can observe that the condition $\zeta\ln\omega\in[-\pi/2,\pi/2]$ yields
\begin{equation}\label{signum-relation-for-sinus}
    \text{sign}(\ln\omega\cdot\sin(\zeta\ln\omega))=\text{sign}\zeta
\end{equation}
for every value of $\omega$, except $\omega=1$ and then, for every point $(\xi,\zeta)\in\Xi$, one of following relations depending on the signum of $\zeta$ holds
\begin{equation}\label{Cross-difinition}
\begin{split}
    & \left.\frac{dI(\xi,\zeta)}{d\xi}\right|_{\zeta=0}=
    \left[\int\limits_{S^k}\omega^\xi(\ln\omega)\sin(\zeta\ln\omega)dS_y \right]_{\zeta=0}=0;
    \\& \left.\frac{dI(\xi,\zeta)}{d\xi}\right|_{\zeta>0}>0 \quad\text{or}\quad\left.\frac{dI(\xi,\zeta)}{d\xi}\right|_{\zeta<0}<0,
\end{split}
\end{equation}
which to together with \eqref{cross-imaginary-vanishing} implies \eqref{Signum-Behavior-of-I-1} and \eqref{Signum-Behavior-of-I-2}. This completes the proof of Proposition~\ref{Signum-Behavior-of-I-propo}.
\end{proof}


Note also that $I(\xi,\zeta)$, as well as $W(\xi,\zeta)$ are symmetric with respect to $(k/2,0)$. Therefore, the uniqueness obtained in Proposition~\ref{uniquiness-total-in-first-quadrant}, p.~\pageref{uniquiness-total-in-first-quadrant} together with the symmetry of $W$ and the skew-symmetry of $I$ with respect to the lines $\xi=k/2$ and $\zeta=0$ implies that the equation $F(\xi,\zeta)=F(a,b)$ has only two solutions $(\xi,\zeta)=(a,b)$ and $(\xi,\zeta)=(k-a,-b)$ for every $(a,b)\in\overline{Q}$.

It is clear that a similar symmetry argument shows that for every $(a,b)\in\Xi$, the equation $F(\xi,\zeta)=F(a,b)$ also must have only two solutions $(\xi,\zeta)=(a,b)$ and $(\xi,\zeta)=(k-a,-b)$. This completes the proof of Statement~(D) of Main Theorem as well as completes the proof of Main Theorem~\ref{Basic-theorem}, p~\pageref{Basic-theorem}.
\end{proof}

\begin{proof}[\textbf{Proof of Corollary \ref{One-dimension-sphe-pro}, p.~\pageref{One-dimension-sphe-pro}.}]

Note that the condition: $a>b>0$ is sufficient to find $R$ and $r$
such that $R>r>0$ and the equalities $R^2+r^2=a\, ,\,\, 2R\, r=b$
hold. Note also that if we integrate a $2\pi$-periodic function
over the period from 0 to $2\pi$, we can replace $\sin(\theta)$ by
$\cos(\theta)$ and vice versa without changing the result of the
integration. Thus, the function under the integral can be reduced to the
Poisson kernel, which yields

\begin{equation}\label{22}
\begin{split}
& \int\limits_{0}^{2\pi} \frac{d\theta}{(a-b\,\sin(\theta))^p}
=\int\limits_{0}^{2\pi} \frac{d\theta}{(a+b\,\cos(\theta))^p}
\\& =(R^2-r^2)^{-p}\,\int\limits_{0}^{2\pi}
\left(\frac{R^2-r^2}{R^2+r^2+2R\,r\,\cos(\theta)}\right)^p\,d\theta
\\& =(a^2-b^2)^{-p/2}\,\int\limits_0^{2\pi}
\omega^p\, d\theta=
\\& =(a^2-b^2)^{-p/2}\int\limits_0^{2\pi}
\omega^{1-p}\, d\theta
\\& =(a^2-b^2)^{1/2-p}\, \int\limits_0^{2\pi}
(a-b\, \sin(\theta))^{p-1}\, d\theta\,,
\end{split}
\end{equation}
where the fourth equality was obtained by \eqref{Sphe-Pro-Statement-2}. This completes the proof of Corollary \ref{One-dimension-sphe-pro}.
\end{proof}

\begin{proof}[\textbf{Proof of Corollary \ref{Basic-Th-Integrable-Case-general-Item} of the Main Theorem, p.~\pageref{Basic-Th-Integrable-Case-general-Item}.}]

Let us organize the proof as the chain of identities, which will be explained at each step.

\begin{description}
  \item[Step 1.] Using the Corollary \ref{imaginary-part-vanishing-corollary}, from p.~\pageref{imaginary-part-vanishing-corollary}, we have
    \begin{equation}\label{Step-1}
        \int\limits_{S^k(R)}\omega^{1+ib}dS_y=\int\limits_{S^k(R)}\omega\cos(b\ln\omega)dS_y.
    \end{equation}
  \item[Step 2.] Recall that according to \eqref{VarExchange-1} from the Lemma \ref{VarExchange-Lemma}, p.~\pageref{VarExchange-Lemma},
     \begin{equation}\label{Step-2-1}
        dS_y=\frac{2R}{l+q}q^2d\Sigma_{\widetilde{y}}\,,
    \end{equation}
    where all notations are pictured on the Figure \ref{Reference-Picture} below that can be used for a reference. Thus, the last integral in \eqref{Step-1} can be written as
    \begin{equation}\label{Step-2-2}
        \int\limits_{S^k(R)}\omega\cos(b\ln\omega)\frac{2R}{l+q}q^2d\Sigma_{\widetilde{y}}.
    \end{equation}

    \begin{figure}[!h]
    \centering
    \epsfig{figure=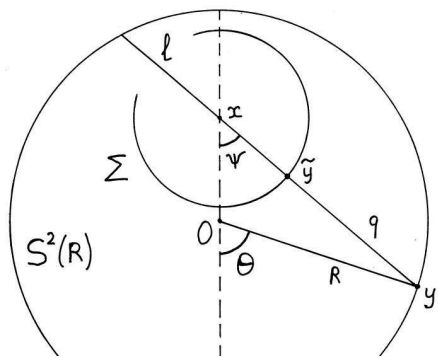,height=4.8cm}\\
  \caption{Reference picture.}\label{Reference-Picture}
    \end{figure}

  \item[Step 3.] Recall that according to \eqref{GeoInte-4}, \eqref{GeoInte-5} and \eqref{GeoInte-6} obtained in the proof of Theorem \ref{Geometric-Interpretation-theorem}, p.~\pageref{Geometric-Interpretation-theorem},
    \begin{equation}\label{Step-3-1}
        \omega(x,y)=\frac lq\quad\text{and}\quad lq=|R^2-r^2|.
    \end{equation}
    Note also that the integrand in \eqref{Step-2-2} depends only on angle $\psi$ pictured on Figure \ref{Reference-Picture} above. Therefore, the integral in \eqref{Step-2-2} becomes
    \begin{equation}\label{Step-3-2}
        4\pi R|R^2-r^2|\int\limits_{0}^{\pi}\frac{\sin\psi}{l+q}\cos(b\ln\omega)d\psi.
    \end{equation}
   \item[Step 4.] Observe that the last integrand is symmetric with respect to $\psi=\pi/2$. Indeed, $\sin(\pi/2+\tau)=\sin(\pi/2-\tau)$ and according to~\eqref{Feature-Feature-Lemma-formula}, p.~\pageref{Feature-Feature-Lemma-formula},
\begin{equation}\label{Three_D_Dirichlet-29}
\begin{split}
    & (l+q)\circ\left(\frac{\pi}{2}-\tau\right)=(l+q)\circ\left(\frac{\pi}{2}+\tau\right)
    \\& \text{and}\quad\ln\frac{l}{q}\circ\left(\frac{\pi}{2}-\tau\right)
    =-\ln\frac{l}{q}\circ\left(\frac{\pi}{2}+\tau\right)\,.
\end{split}
\end{equation}
where the symbol $\circ$ denotes composition. Hence, the integral in \eqref{Step-3-2} can be written as
    \begin{equation}\label{Step-4-2}
        8\pi R|R^2-r^2|\int\limits_{0}^{\pi/2}\frac{\sin\psi}{l+q}\cos\left(b\ln\frac{l}{q}\right)d\psi.
    \end{equation}





   \item[Step 5.] The substitution for $d\psi$ presented in \eqref{Substitution-for-d-psi}, p.~\pageref{Substitution-for-d-psi}
   allows to rewrite the integral in \eqref{Step-4-2} as follows
    \begin{equation}\label{Step-5-2}
        \frac{8\pi R|R^2-r^2|}{4r}
        \int\limits_{\psi=0}^{\psi=\pi/2}\cos(b\ln\omega)d(\ln\omega).
    \end{equation}
   \item[Step 6.] It is clear that the last integral can be computed directly, since $(\sin u)'=\cos u$ and then the integral in \eqref{Step-5-2} yields
    \begin{equation}\label{Step-6-1}
        \left.\frac{2\pi R|R^2-r^2|}{br}\sin(b\ln\omega)\right|_{\psi=0}^{\psi=\pi/2}.
    \end{equation}
   \item[Step 7.] Note now that
    \begin{equation}\label{Step-7-1}
        \omega(\pi/2)=\frac{l(\pi/2)}{q(\pi/2)}=1\,,\quad\text{while}\quad\omega(0)=\frac{|R-r|}{R+r}.
    \end{equation}
    Therefore, the direct computation in \eqref{Step-6-1} gives
    \begin{equation}\label{Step-7-2}
        \frac{2\pi R|R^2-r^2|}{br}\sin\left(b\ln\frac{R+r}{|R-r|}\right)\,,
    \end{equation}
    which completes the proof of \eqref{Basic-Th-Integrable-Case-1}, p.~\pageref{Basic-Th-Integrable-Case-1} of the Corollary.
\end{description}
The partial case stated in \eqref{Basic-Th-Integrable-Case-2}, p.~\pageref{Basic-Th-Integrable-Case-2} can be obtained by taking limit at both parts of \eqref{Basic-Th-Integrable-Case-1} as $b\rightarrow0$. This limit exists since the integrand in \eqref{Basic-Th-Integrable-Case-1}, p.~\pageref{Basic-Th-Integrable-Case-1} or equivalently, in \eqref{Step-1}, p.~\pageref{Step-1} converges uniformly as $b\rightarrow0$. This completes the proof of Corollary~\ref{Basic-Th-Integrable-Case-general-Item}, p.~\pageref{Basic-Th-Integrable-Case-general-Item}.
\end{proof}

\chapter{Analysis in Euclidean Geometry.}

In this chapter we derive a new approach to some classical problems arising in electrostatics.

\section {A sufficient condition for a function depending only on the distance to be harmonic.}

For a reference for all variables used in the following theorem a reader may look at the Figure \ref{Converse_Mean_Value} below.

\begin{theorem}\label{Sufficient-Dista-Condi-Thm}
Let $x, y$ be points in $\mathbb{R}^{k+1}$ and $v\in S^k(|y|)$ be integration variable. Let $g(\xi)$ be a
continuous function of one variable $\xi\in
\mathbb{R}$ defined for $\xi>0$. Assume that
\begin{equation}\label{2-dd}
    \int\limits_{S^k(|y|)} g(|x-v|)dS_v=M(|y|)
\end{equation}
is independent of $x$ inside $S^k(|y|)$, a sphere of
radius $|y|$ centered at the origin. In other words, the integral
depends only on $|y|$ and does not depend on $x$ inside the
$S^k(|y|)$ for all $x, y\in \mathbb{R}^{k+1}$ such that
$|x|<|y|$. Then $g(|x|)$ is harmonic with respect to $x$ for every $x\neq0$.
\end{theorem}

\begin{proof} Observe first that the direct computation yields
\begin{equation}\label{3-dd}
    \bigtriangleup_x g(|x-y|)=\bigtriangleup_y
    g(|x-y|)=\bigtriangleup_{(x-y)}
    g(|x-y|)=\widetilde{g}(|x-y|)\,,
\end{equation}
where $\widetilde{g}(|x-y|)$ is again a function depending only
on the distance between $x$ and $y$. Note also that
$M(|y|)=\sigma_k|y|^k\cdot g(|y|)$, since by the condition, the integral in
\eqref{2-dd} does not depend on $x$ for $|x|<|y|$ and, therefore,
the value of the integral is equal to the value at $x=0$.
Thus, we have
\begin{equation}\label{4-dd}
    \frac{1}{\sigma_k |y|^k}\int\limits_{S^k(|y|)}
    g(|x-v|) dS_v = g(|y|)\,,
\end{equation}
where $\sigma_k$ is the area of $k$-dimensional unit sphere.

We are going to use the converse of mean
value property, see~\cite{Sheldon},~p.16. We must only show that the condition
\eqref{2-dd} implies the mean value property for $g(|x-y|)$ with
respect to $x$ for every $y$ such that $|y|>|x|$. Indeed, look at
the following picture.

\begin{figure}[!h]
    \centering
    \epsfig{figure=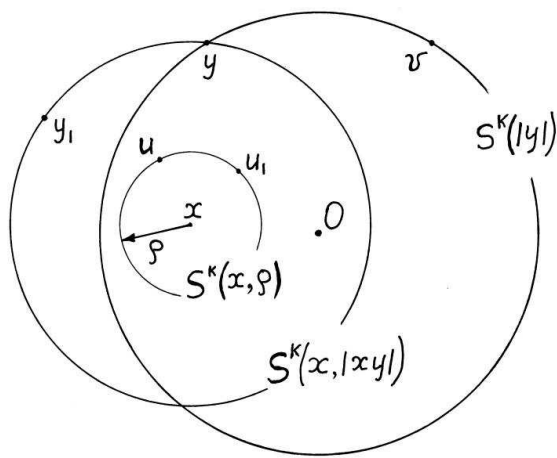,height=6cm}\\
  \caption{Sufficient condition theorem.}\label{Converse_Mean_Value}
\end{figure}

Let $x,y\in\mathbb{R}^{k+1}$ such that $|y|>|x|$ and let $\rho$ be a positive number such that $\rho<|y|-|x|$. Let $S^k(|y|)$ be the sphere of radius $|y|$ centered at the
origin $O$ and $S^k(x,\rho)$ be the sphere of radius $\rho$ centered at the point $x$. Fix $u\in S^k(x,\rho)$ and let $u_1\in S^k(x,\rho)$, $y_1\in S^k(x,|xy|)$ be the parameters of integration.
Now we are ready to check the mean value property for $g(|x-y|)$ with
respect to $x$. Using Lemma~\ref{Basic_lemma}, p.~\pageref{Basic_lemma} and \eqref{4-dd}, we have
\begin{equation}\label{1-ee}
\begin{split}
    & \frac{1}{|S^k(x,\rho)|}
    \int\limits_{S^k(x,\rho)}
    g(|y-u_1|) dS_{u_1}
    \\& =\frac{1}{|S^k(x,|xy|)|}
    \int\limits_{S^k(x,|xy|)}
    g(|u-y_1|) dS_{y_1}=g(|y-x|)\,,
\end{split}
\end{equation}
and therefore, $g$ satisfies the mean value property. Therefore,
$g(|x-y|)$ is harmonic with respect to $x$ for every $x$ inside
$S^k(|y|)$ and for every $y\in S^k(|y|)$. Using
\eqref{3-dd}, we can easily see that $g(|x-y|)$ is harmonic for
all $x,y\in\mathbb{R}^{k+1}$ such that $x\neq y$, which completes the proof of Theorem \ref{Sufficient-Dista-Condi-Thm}.
\end{proof}
 \smallskip

\begin{remark} The statement of Theorem~\ref{Sufficient-Dista-Condi-Thm}, p.~\pageref{Sufficient-Dista-Condi-Thm} can not be
generalized for every function of two variables $g(x,y)$, i.e.,
the condition given by
\begin{equation}\label{2-ee}
    \int\limits_{S^k(|y|)} g(x,y) d\,S_y=M(|y|)
\end{equation}
for all $x$ and $y$ such that $|x|<|y|$ does not imply that
$g(x,y)$ is harmonic in $x$ even for all $x$ and $y$
such that $|x|<|y|$. For example, the function
\begin{equation}\label{9-dd}
    \omega^k(x,y)=\left(\frac{|x|^2-|y|^2}{|x-y|^2}\right)^k\quad\text{with}\quad k>1\,,
\end{equation}
according to \eqref{Sphe-Pro-Statement-2}, p.~\pageref{Sphe-Pro-Statement-2} satisfies \eqref{2-ee}, but if $|x|<|y|$, then
 $(\triangle_x \omega^k)(x,y)=0$ for every $y$ only at $x=0$. This fact is proven in the following
Proposition.
\end{remark}

\medskip




\begin{proposition}\label{Couner-Example-Hermo-Proposition}
Let $x,y\in\mathbb{R}^{k+1}$ such that $|x|< |y|$ and $k\geq
1$. Then $\omega^k$ is harmonic in $x$ only if $k=1$. If $k>1$,
then $(\triangle_x \omega^k)(x,y)=0$ for every admissible $y$ only at $x=0$.
\end{proposition}


\begin{proof} At first, let us consequently obtain the following
list of trivial identities.



\begin{equation}\label{1-13-dd}
    \frac{\partial \omega}{\partial
    x_1}=\frac{2x_1}{|x-y|^2}-\frac{2(x_1-y_1)}{|x-y|^2}\omega\,,
\end{equation}

\begin{equation}\label{2-14-dd}
    \frac{\partial\,^2 \omega}{(\partial
    x_1)^2}=\frac{2}{|x-y|^2}-\frac{8x_1(x_1-y_1)}{|x-y|^4}+
    \frac{8(x_1-y_1)^2\omega}{|x-y|^4}-\frac{2\omega}{|x-y|^2}\,,
\end{equation}

\begin{equation}\label{3-15-dd}
    (x,x-y)=\frac{|x|^2-|y|^2+|x-y|^2}{2}=\frac{1}{2}(\omega
    +1)\cdot |x-y|^2\,.
\end{equation}
Using \eqref{2-14-dd}, \eqref{3-15-dd} and \eqref{1-13-dd},
\eqref{3-15-dd} respectively, we obtain
\begin{equation}\label{4-16-dd}
    \bigtriangleup_x \omega=\frac{2(k-1)(1-\omega)}{|x-y|^2}\quad
    \text{and}\quad |\nabla_x\omega|^2=\frac{4|y|^2}{|x-y|^4}\,.
\end{equation}
Thus, $\omega^k$ is harmonic for $k=1$ and $x,y\in\mathbb{R}^2$.

To prove the statement of the Proposition, note that
\begin{equation}\label{5-18-dd}
    \bigtriangleup_x \omega^k=k(k-1)\omega^{k-2}|\nabla_x
    \omega|^2+k\omega^{k-1}\bigtriangleup_x \omega
\end{equation}
and the direct computation shows that for $k>1$
\begin{equation}\label{6-19-dd}
    \bigtriangleup_x \omega^k=0\quad \text{if and only if}\quad
     |x-y|^2=\frac{(|x|^2-|y|^2)^2}{|x|^2+|y|^2}\,,
\end{equation}
where the last equality holds for every admissible $y\in\mathbb{R}^{k+1}$
only if $x=0$. Therefore, if $|y|>|x|$, then $(\triangle_x \omega^k)(x,y)=0$ for every $y$ only at $x=0$, which completes the proof of Proposition~\ref{Couner-Example-Hermo-Proposition}.
\end{proof}


\bigskip

\section {Algebraic computation of the electrostatic potential created by a charge with evenly distributed density.}

In this section we are going to present a new method relying on geometric interpretation of integrands to compute the following important integrals arising in electrostatics:

\begin{equation}\label{Two-Integrals}
    \int\limits_{S^k}\frac{dS_y}{|x-y|^{k-1}}\quad\text{and}\quad
    \int\limits_{S^k}\frac{|x|^2-|y|^2}{|x-y|^{k+1}}dS_y\,.
\end{equation}

The first integral appears when we compute the potential created by a sphere with an evenly distributes charge and the second integral plays an important role in Dirichlet Problem.

Note that the potential created by an evenly distributed charge in a sphere is well known. It can be computed, for example, using Gauss Integral Theorem saying that "the flux outward across the surface bounding a region is equal to $-4\pi$ times the total mass (charge) in the region, provided the bounding surface meets no mass (charge)". The computation process is described in \cite{Kellog}, pp.~40-53.

To compute the second integral in \eqref{Two-Integrals} it is convenient to use the integral form of Gauss Integral Theorem saying that

\begin{equation}\label{Gauss-Sokes-Formula}
    \int\limits_V \triangle fdV=\int\limits_{\partial V} \nabla f\cdot d\sigma\,,
\end{equation}
which is the special case of Stokes' formula, see~\cite{Sheldon}, (p.~27, exercises 17, 18). This method is applicable to both integrals in \eqref{Two-Integrals} and it is summarized in the proof of  Theorem~\ref{Stokes_Formula_Computation_Theorem} below on page~\pageref{Stokes_Formula_Computation_Theorem}.

A new and a shorter way to compute the second integral in \eqref{Two-Integrals} relies on formula \eqref{Sphe-Pro-Statement-9}, p.~\pageref{Sphe-Pro-Statement-9} saying that both integrals in \eqref{Two-Integrals} are identically equal for every point $x\notin S^k$.

From now on our goal is to present a new algebraic way based on geometric interpretation of integrands to compute both integrals in \eqref{Two-Integrals}. We introduce a few definitions that will be used later.

\begin{definition} A point charge in
$\mathbb{R}^{k+1}$ as a pair $(Q,x)$, where $Q\in \mathbb{R}$, $x\in\mathbb{R}^{k+1}$. $Q$ is
called charge measure or, simply, charge and $x\in
\mathbb{R}^{k+1}$ is called the location of the charge.
\end{definition}

\begin{definition}\label{Electro_Potential_by_Point_Charge_Def} The electrostatic potential in
$\mathbb{R}^{k+1}$ created by a point charge $Q$ located at point $x$ is the
function of one variable $y\in \mathbb{R}^{k+1}\backslash \{x\}$,
\begin{equation}\label{23}
    W_k^Q (x,y)=\frac{Q}{|x-y|^{k-1}}\,.
\end{equation}
\end{definition}
\smallskip

We shall see that for every $y\neq x$, $W^Q_k(x,y)$ is a harmonic
function with respect to the variables $x$, $y$ and $W^Q_k(x,y)$
depends only on the distance between $x$ and $y$. Definition \ref{Electro_Potential_by_Point_Charge_Def}
is a natural generalization of the electrostatic potential in
$\mathbb{R}^3$ created by a point charge, see~\cite{Freedman} (p.~738).

\begin{definition}\label{Electro_Potential_by_Sphere_Def} The electrostatic potential
created by a charge with density $\rho(y)$ on the k-dimensional sphere $S^k\in
\mathbb{R}^{k+1}$ is the function
\begin{equation}\label{24}
    F_k(x)=\int\limits_{S^k} \frac{\rho
    (y)}{|x-y|^{k-1}}\,d\,S_y\, ,
\end{equation}
where $y\in S^k$ and $\rho (y)$ is a continuous function
called the density distribution.
\end{definition}

Note that this definition is a natural generalization of the
potential created by a distributed charge on a surface in
$\mathbb{R}^3$, see \cite{Freedman}, p.~738. Note that although the function
$F_k(x)$ is well defined for every $x\in \mathbb{R}^{k+1}$,
we consider it only for $x \notin S^k$.
\smallskip

Consider the potential created by an evenly distributed charge, $\rho(y)\equiv\text{const}$, on the sphere $S^k \in
\mathbb{R}^{k+1}$. Without loss of generality, we may assume that
$\rho (y)\equiv 1$. Thus, we have

\begin{proposition} Let $\rho(y)\equiv 1$. Then
\begin{equation}\label{25-g}
    \int\limits_{S^k}\frac{d\,S_y}{|x-y|^{k-1}}=
    F_k(x)=\frac{|S^k|}{R^{k-1}}\,\, , \quad \text{for} \quad |x|<R
\end{equation}
and
\begin{equation}\label{26}
    \int\limits_{S^k}\frac{d\,S_y}{|x-y|^{k-1}}=
    F_k(x)=\frac{|S^k|}{|x|^{k-1}}\,\, , \quad \text{for} \quad
    |x|>R
\end{equation}
\end{proposition}

\begin{proof}[Proof of \eqref{25-g}.] Let $x$ be inside a sphere
$S^k$.

Using the change of variables formulae \eqref{VarExchange-1}, \eqref{VarExchange-2} from Lemma~\ref{VarExchange-Lemma}, p.~\pageref{VarExchange-Lemma} and elementary algebra, we have

\begin{equation}\label{25}
\begin{split}
& \int\limits_{S^k}\frac{dS_y}{|x-y|^{k-1}}=
\int\limits_{\Sigma}\frac{1}{q^{k-1}}\, \frac{2R}{l+q}\,
q^k\, d\Sigma
\\& =R\, \int\limits_{\Sigma}\frac{2\,q}{l+q}\, d\Sigma=
R\,\int\limits_{\Sigma}\frac{2\,l}{l+q}\,d\Sigma
\\&
=R\,\int\limits_{\Sigma}\left(2-\frac{2\,q}{l+q}\right)\,d\Sigma\,.
\end{split}
\end{equation}
Hence
\begin{equation}\label{27}
F(x)=R\, \int\limits_{\Sigma}\frac{2\,q}{l+q}\,
d\Sigma=
R\,\int\limits_{\Sigma}\left(2-\frac{2\,q}{l+q}\right)\,d\Sigma\,.
\end{equation}
Therefore
\begin{equation}\label{26-a}
F(x)=\int\limits_{S^k}\frac{dS_y}{|x-y|^{k-1}}=R\,|\sigma_k|=
\frac{|S^k|}{R^{k-1}}\,,
\end{equation}
where $|\sigma_k|$ is the area of a $k$-dimensional unit sphere
and $|S^k|$ is the area of a $k$-dimensional sphere of
radius $R$. Thus, the electrostatic potential inside a sphere
created by a charge, evenly distributed on the sphere, is
constant, which is, clearly, a well known fact.
\end{proof}

\begin{proof}[Proof of \eqref{26}.] Let $x$ be outside of the sphere
$S^k$. Note that identity \eqref{26} can be obtained as a
simple consequence of Lemma~\ref{Basic_lemma} from p.~\pageref{Basic_lemma} and formula~\eqref{25-g}. Indeed, let
$r=|x|>|y|=R$ and let $g(|x-y|)=1/|x-y|^{k-1}$. Then, using
\eqref{25-g} and \eqref{ExchaSphere-1} from p.~\pageref{ExchaSphere-1}, we have
\begin{equation}\label{1-m}
    |x|^k \int\limits_{S^k(R)} g(|x-y_1|) d\,S_{y_1}=R\,^k
    \int\limits_{S^k(r)} g(|x_1-y|)
    d\,S_{x_1}=R\,^k\cdot\frac{|S^k(r)|}{|x|^{k-1}}\,,
\end{equation}
which yields
\begin{equation}\label{2-m}
    \int\limits_{S^k(R)} g(|x-y|) d\,S_y=\frac{R\,^k\cdot
    |\sigma_k|}{|x|^{k-1}}=\frac{|S^k(R)|}{|x|^{k-1}}\,,
\end{equation}
where, as before, $S^k(r)$ and
$S^k(R)=S^k$ are two concentric $k$-dimensional
spheres of radii $r$ and $R$ respectively centered at the origin $O$. Thus, the potential at point $x$ is equal to the potential created by the point charge of measure $|S^k|$ and located at $O$, which also coincide with the well known fact from electrostatics.
\end{proof}

\begin{remark}
Identity \eqref{25-g} implies that the function $1/|x-y|^{k-1}$ is
harmonic. This follows from Theorem~\ref{Sufficient-Dista-Condi-Thm}, p.~\pageref{Sufficient-Dista-Condi-Thm} above.
\end{remark}

\begin{remark} As a bonus, using the identity \eqref{Sphe-Pro-Statement-9}
from p.~\pageref{Sphe-Pro-Statement-9} and formulae \eqref{25-g}, \eqref{26}, we obtain

\begin{enumerate}
  \item For $x$ inside a sphere,
    \begin{equation}\label{28-a} \int\limits_{S^k}
        \frac{R^2-|x|^2}{|x-y|^{k+1}}\,dS_y=R\,|\sigma_k|\,.
    \end{equation}
  \item For $x$ outside of a sphere
    \begin{equation}\label{29}
        \int\limits_{S^k}
        \frac{|x|^2-R^2}{|x-y|^{k+1}}\,dS_y=\left(\frac{R}{|x|}\right)^{k-1}
        R\,|\sigma_k|\,.
    \end{equation}
\end{enumerate}
\end{remark}

\begin{remark}\label{Direct_Poisson_compute_Geometrically} The identity \eqref{28-a} can be obtained
directly. Using formulae \eqref{GeoInte-5}, \eqref{GeoInte-6} from p.~\pageref{GeoInte-5} and Figure~\ref{Geometric_Interpretation} from p.~\pageref{Geometric_Interpretation}, observe that
\begin{equation}\label{2-x}
    \frac{R^2-|x|^2}{|x-y|^{k+1}}=\frac{l\cdot
    q}{q^{k+1}}=\frac{l}{q^{k}}
\end{equation}
and repeat the process shown in \eqref{25}.
\end{remark}

\begin{theorem}\label{Stokes_Formula_Computation_Theorem}
Let function $\Phi(x,y)=\tau(|x|,|y|)\cdot g(|x-y|)$ be harmonic
in $x$ for every $x\in\mathbb{R}^{k+1}$ such that
$|x|\neq |y|$. Let $|x|=r$, $|y|=R$. Define
$\widetilde{g}(x)=\int\limits_{S^k(R)} g(|x-y|)
d\,S_y$. Then

\begin{description}
  \item[for every $x$ inside $S^k(R)$]
    \begin{equation}\label{Stokes_Inside}
        \widetilde{g}(x)
        =\int\limits_{S^k(R)} g(|x-y|) d\,S_y =
        \frac{\tau(0,R)\cdot g(R)\cdot
        |S^k(R)|}{\tau(|x|,R)}
    \end{equation}
  \item[and for $x$ outside of $S^k(R)$]
    \begin{equation}\label{Stokes_Outside}
        \widetilde{g}(x)
        =\int\limits_{S^k(R)} g(|x-y|)
        d\,S_y=\frac{\tau(0,|x|)\cdot g(|x|)\cdot
        |S^k(R)|}{\tau(R,|x|)}\,,
    \end{equation}
    where $|S^k(R)|$, as before, is the volume of
    $S^k(R)$.
\end{description}
\end{theorem}

\begin{proof}[\textbf{Proof of Theorem \ref{Stokes_Formula_Computation_Theorem}.}] Note, at first, that the function
\begin{equation}\label{Stokes-1}
    \widetilde{\Phi}(x)=\int\limits_{S^k(R)}
    \Phi(x,y)d\,S_y=\int\limits_{S^k(R)}
    \tau(|x|,|y|)\cdot g(|x-y|) d\,S_y
\end{equation}
is harmonic and depends only on $|x|$. Then $\nabla
\widetilde{\Phi}(\overline{x})$ is orthogonal to $S^k(r)$
for every $\overline{x}\in S^k(r)$ and, moreover,
$|\widetilde{\Phi}(\overline{x})|\equiv$ const for every
$\overline{x}\in S^k(r)$. By Stokes' formula, for any
$x\in S^k(r)$, we have
\begin{equation}\label{Stokes-2}
    0=\int\limits_{S^k(r)}
    \bigtriangleup\widetilde{\Phi}(\overline{x})
    d\,S_{\overline{x}}=\int\limits_{S^k(r)}
    (\nabla\widetilde{\Phi}(\overline{x}),n_{\overline{x}})    d\,S_{\overline{x}}=\pm
   |\nabla\widetilde{\Phi}(x)|\cdot |S^k(r)|\,,
\end{equation}
where $n_{\overline{x}}$ is a unit vector orthogonal to
$S^k(r)$ at point $\overline{x}$. Therefore,
$\widetilde{\Phi}(x)$=const=$\widetilde{\Phi}(0)$ for every $x$
inside $S^k(R)$. Thus,
\begin{equation}\label{Stokes-3}
   \widetilde{\Phi}(x)=\widetilde{\Phi}(0)=
    \int\limits_{S^k(R)} \tau(0,R)\cdot g(R) d\,S_y=
    \tau(0,R)\cdot g(R)\cdot|S^k(R)|\,,
\end{equation}
which yields \eqref{Stokes_Inside}.

The identity \eqref{Stokes_Outside} can easily be obtained from \eqref{Stokes_Inside}
by Lemma~\ref{Basic_lemma}, page~\pageref{Basic_lemma}. Indeed, for $|x|>|y|$, we have
\begin{equation}\label{Stokes-4}
\begin{split}
    & \int\limits_{S^k(R)} g(|x-y_1|)
    d\,S_{y_1}=\left(\frac{R}{r}\right)^k
    \int\limits_{S^k(r)} g(|y-x_1|)
    d\,S_{x_1}
    \\& =\left(\frac{R}{r}\right)^k\,\,\frac{\tau(0,r)\cdot g(r)\cdot
    |S^k(r)|}{\tau(R,|x|)}=\frac{\tau(0,|x|)\cdot g(|x|)\cdot
    |S^k(R)|}{\tau(R,|x|)}\,,
\end{split}
\end{equation}
which completes the proof of Theorem~\ref{Stokes_Formula_Computation_Theorem}.
\end{proof}

\section {Dirichlet Problem for a ball or its exterior.}

Let $f(y)$ be a continuous function defined on the $k$-dimensional
sphere $S^k$ of radius R centered at the origin $O$. The
Dirichlet problem is to find a harmonic function $U(x)$ inside
(outside) the sphere $S^k$ such that $\lim_{x\rightarrow
x_0} U(x)=f(x_0)$ for every $x_0\in S^k$.

It is known that the solution for this problem can be written as

\begin{equation}\label{Standard-Dir-Solution}
    U(x)=
    \frac{1}{R\,|\sigma_k|}\int\limits_{S^k}
    \frac{|R\,^2-|x|^2|}{|x-y|^{k+1}}f(y)\,dS_y\,,
\end{equation}
where the ratio in this integrand is called Poisson Kernel, see~\cite{Kellog}, p.~237 or~\cite{Sheldon}, p.12. The standard way to deduce formula~\eqref{Standard-Dir-Solution} relies on Green's formula and Green's function. In particular, the Poisson kernel is obtained as the normal derivative of Green's function, see~\cite{Kellog}, pp.~240-243.

On the other hand, as we have seen, Poisson kernel can be obtained just by formula \eqref{Sphe-Pro-Statement-9}, p.~\pageref{Sphe-Pro-Statement-9}. Thus, we have an opportunity to show that formula~\eqref{Standard-Dir-Solution} can be deduced in a shorter way by using formula~\eqref{Sphe-Pro-Statement-9} mentioned above.

\begin{proof}[\textbf{Solution.}]

Let us solve, first, the simplest Dirichlet problem corresponding
to $f(y)\equiv 1$. It turns out that the solution is $U(x)\equiv
1$. Note, also, that the solution can be described by means of
electrostatic potential function defined in \eqref{23}, since this
function is harmonic. According to \eqref{24} and \eqref{25-g}, we can distribute
charge on $S^k$ in such a way as to make unit potential
inside the sphere. It is not hard to see that the density of this
distribution should be $\rho=1/R\,|\sigma_k|$. Therefore, a
solution for the simplest Dirichlet problem can be presented in
the following form

\begin{equation}\label{1-y}
    U(x)\equiv 1=\frac{1}{R\,|\sigma_k|}\int\limits_{S^k}
    \frac{d\,S_y}{|x-y|^{k-1}}=
    \frac{1}{R\,|\sigma_k|}\int\limits_{S^k}
    \frac{R\,^2-|x|^2}{|x-y|^{k+1}}f(y)\,dS_y\,,
\end{equation}
where $f(y)\equiv1$ and the last equality was obtained by using \eqref{Sphe-Pro-Statement-9}, p.~\pageref{Sphe-Pro-Statement-9}, or the
last integral can be computed directly by using the argument
described in Remark \ref{Direct_Poisson_compute_Geometrically} above.


To show that the last expression in \eqref{1-y} gives a solution for a Dirichlet problem with arbitrary boundary values represented by an arbitrary continuous function $f(y)$, we need to note that the
integrand
    \begin{equation}\label{2-y}
    P_k(x,y)=\frac{1}{R\,|\sigma_k|}\cdot
    \frac{R\,^2-|x|^2}{|x-y|^{k+1}}\,,
\end{equation}
obtained in \eqref{1-y}, is harmonic with respect to $x$, which
can be easily checked by the direct computation, and
\begin{equation}\label{3-y}
    \lim_{x\rightarrow x_0} \int\limits_{S^k} P_k(x,y)
    \cdot f(y) d\,S_y =f(x_0)
\end{equation}
for every $x_0\in S^k$ and for every continuous function
$f$ defined on $S^k$. The last property \eqref{3-y} is
proved carefully for $k=1$ in \cite{Ahlfors}, (p.~168). The proof for
arbitrary $k$ can be obtained in the same way.
Note also that both of these properties hold for
$|x|<R$, as well as, for $|x|>R$. Therefore, the function
\begin{equation}\label{4-y}
    W(x)=\frac{1}{R\,|\sigma_k|}\int\limits_{S^k}
    \frac{|R\,^2-|x|^2|}{|x-y|^{k+1}}\cdot f(y) d\,S_y
\end{equation}
gives a solution for a Dirichlet problem with arbitrary boundary values,
i.e. $W(x)$ is
harmonic inside (outside) the $S^k$ and
$\lim_{x\rightarrow x_0} W(x)=f(x_0)$ for every continuous
function $f$ defined on $S^k$.
\end{proof}

\section {Inequalities.}

Recall that the following trivial inequalities
\begin{equation}\label{Trivial-Inequalities}
 ||x|-|y||\leq|x-y|\leq|x|+|y|
\end{equation}
imply that for $y\in S^k(R)$ and for $x\notin S^k(R)$, we have
\begin{equation}\label{1-s}
    C=\frac{|S|^k}{|R+|x|\,|^\alpha}\leq \int\limits_{S^k}
    \frac{d\,S_y}{|x-y|^\alpha}\leq
    \frac{|S|^k}{|R-|x|\,|^\alpha}=D\,,
\end{equation}
where equalities hold only for $|x|=0$.

The goal of this section is to show that for $\alpha>k\geq 1$ these lower
and upper bounds can be essentially improved if we apply first the integral identity~\eqref{Sphe-Pro-Statement-9.0}, from p.~\pageref{Sphe-Pro-Statement-9.0} to the middle integral in \eqref{1-s} and only then, use \eqref{Trivial-Inequalities}. The result obtained by this procedure is stated in the following theorem.

 \begin{theorem}\label{Inequality_theorem} Let $\alpha$ be real, $x\notin S^k$ and $y\in
 S^k$. Then
 \begin{equation}\label{37}
 \frac{|S^k|}{|R-s|x||^k
    |R+s|x||^{\alpha-k}}\leq
    \int\limits_{S^k}
    \frac{d\,S_y}{|x-y|^\alpha}\leq \frac{|S^k|}
    {|R+s|x||^k |R-s|x||^{\alpha-k}}\,,
\end{equation}
where $s=\text{sign}(\alpha-2k)$ for $\alpha \neq 2k$ and $s=1$
for $\alpha=2k$. In each inequality, the equality holds if and only if
$\alpha=2k$ or $x=0$.
\end{theorem}

We will show in the following remarks and examples that for $\alpha>k\geq 1$ these lower
and upper bounds in \eqref{37} are finer then the obvious bounds $C$
and $D$ introduced in \eqref{1-s}.

\begin{remark} For $\alpha>k$ and $x\neq 0$, we have
\begin{equation}\label{38}
C<\frac{|S^k|}{|R+s\,|x|\,|^k\cdot
|R-s\,|x|\,|^{\alpha-k}} \leq
\frac{|S^k|}{|R-s\,|x|\,|^k\cdot
|R+s\,|x|\,|^{\alpha-k}}<D
\end{equation}
where the equality holds only for $\alpha=2k$.
\end{remark}

\begin{remark} For $k=1$, and $a>b>0$, we can write
\begin{equation}\label{39}
\begin{split}
   A= \frac{2\pi}{(a-s\,b)^{1/2}\cdot (a+s\,b)^{p-1/2}} & \leq
    \int\limits_0^{2\pi}\frac{d\theta}{(a-b\,\sin(\theta))^p} \\& \leq
    \frac{2\pi}{(a+s\,b)^{1/2}\cdot (a-s\,b)^{p-1/2}}=B\,,
\end{split}
\end{equation}
where $s=\text{sign}(p-1)$ for $p\neq 1$ and $s=1$ for $p=1$. A
and B are just notation for these upper and lower bounds. The
equality holds if and only if $p=1$ or $b=0$.

For $p>1/2$, we have
\begin{equation}\label{40}
\frac{2\pi}{(a+b)^p}<A\leq B < \frac{2\pi}{(a-b)^p}\,,
\end{equation}
where the equality holds only for $p=1$.
\end{remark}

\textbf{Example 1.} If $a>b>0$, then
\begin{equation}\label{41}
    \int\limits_0^{2\pi}\frac{d\theta}{(a-b\,\sin\theta)^{3/2}} \leq
    \frac{2\pi}{(a-b)\sqrt{a+b}} < \frac{2\pi}{(a-b)^{3/2}}\,.
\end{equation}
\medskip

\textbf{Example 2.}
\begin{equation}\label{42}
    \int\limits_{S^2} \frac{dS_y}{|x-y|^5} \leq
    \frac{4\pi R^2}{||x|+R|^2\cdot ||x|-R|^3} < \frac{4\pi
    R^2}{||x|-R|^5}\,,
\end{equation}
where $x\notin S^2$ and $y\in S^2$.\medskip

\begin{proof}[\textbf{Proof of theorem \ref{Inequality_theorem}.}] Using the basic identity
\eqref{Sphe-Pro-Statement-2}, p.~\pageref{Sphe-Pro-Statement-2}, we have

\begin{enumerate}
  \item \textbf{For $\alpha\leq 2k$,}
    \begin{equation}\label{1-p}
    \begin{split}
        & \int\limits_{S^k}
        \frac{d\,S_y}{|x-y|^{\alpha}}=|R\,^2-|x|^2|^{k-\alpha}\int\limits_{S^k}
        \frac{d\,S_y}{|x-y|^{2k-\alpha}}\\& \leq
        |S^k| \frac{|R-|x||^{k-\alpha}\cdot
        |R+|x||^{k-\alpha}}{|R-|x||^{2k-\alpha}}=\frac{|S^k|}{|R-|x||^k\cdot
        |R+|x||^{\alpha-k}} \\& =\frac{|S^k|}{|R+s|x||^k\cdot
        |R-s|x||^{\alpha-k}}\,,
    \end{split}
    \end{equation}
    where $s=\text{sign} (\alpha-2k)$ for $\alpha\neq 2k$.
  \item \textbf{For $\alpha\geq 2k$,}
    \begin{equation}\label{2-p}
    \begin{split}
        & \int\limits_{S^k}
        \frac{d\,S_y}{|x-y|^{\alpha}}=|R\,^2-|x|^2|^{k-\alpha}\int\limits_{S^k}
        |x-y|^{\alpha-2k} d\,S_y \\& \leq
        |R-|x||^{k-\alpha}\cdot |R+|x||^{k-\alpha}\cdot
        |R+|x||^{\alpha-2k}\cdot |S^k| \\& =\frac{|S^k|}{|R+|x||^k\cdot
        |R-|x||^{\alpha-k}} =\frac{|S^k|}{|R+s|x||^k\cdot
        |R-s|x||^{\alpha-k}}
    \end{split}
    \end{equation}
    where $s=\text{sign}(\alpha-2k)$ for $\alpha \neq 2k$.
\end{enumerate}
It is clear that for $\alpha=2k$ we have to set $s=1$
or $s=-1$. This, according to \eqref{Sphe-Pro-Statement-9.0}, p.~\pageref{Sphe-Pro-Statement-9.0}, gives us the same equality in both cases. Thus, combining
the results obtained in \eqref{1-p} and \eqref{2-p}, we get the
upper bound in \eqref{37}. Similarly, we obtain the lower
bound given in \eqref{37}. Using the same method we can
prove the inequalities from \eqref{39}.
\end{proof}

\begin{proof}[\textbf{Proof of \eqref{38}}.] This proof can be obtained just by the direct computation applied separately for the following intervals for $\alpha$.
\begin{equation}\label{4-p}
     k<\alpha<2k\,, \quad \alpha\geq2k\,, \quad \alpha=k\,, \quad \alpha<k\,.
\end{equation}
These direct computations complete the proof of \eqref{38}.
\end{proof}

\chapter{Analysis in Hyperbolic Geometry.}

In this chapter we consider some applications of the geometric interpretation introduced in \eqref{GeoInte-4} to hyperbolic geometry.

\section {The relationship of the Main Theorem to Hyperbolic Geometry.}

\subsection {Notations and some geometric preliminaries.}

For the future reference, let us introduce the following notation. Let $H_\rho^n$ and $B^n_\rho$ be the half-space model and the ball model, respectively, of a hyperbolic $n-$dimensional space with a constant sectional curvature $\kappa=1/{\rho^2}<0$. Recall that the half-space model $H_\rho^n$ consists of the open half-space of points $(x_1, ..., x_{n-1}, t)$ in $R^n$ for all $t>0$ and the metric is given by $(\rho/t)|ds|$, where $|ds|$ is the Euclidean distance element. The ball model $B^n_\rho$ consists of the open unit ball $|X|^2+T^2<\rho^2, (X,T)=(X_1, ..., X_{n-1}, T)$ in $R^n$, and the metric for this model is given by $2\rho^2|ds|/(\rho^2-|X|^2-T^2)$.

Recall also that in a hyperbolic space a Laplacian can be represent as
\begin{equation}\label{Uniq-6}
\begin{split}
    \triangle
    & =\frac{1}{\rho^2}\left[ t^2\left( \frac{\partial^2}{\partial x_1^2}+...+\frac{\partial^2}{\partial x_{n-1}^2} + \frac{\partial^2}{\partial t^2} \right)-(n-2)t\frac{\partial}{\partial t} \right]
    \\& =\frac{\partial^2}{\partial r^2}+
    \frac{n-1}{\rho}\coth\left(\frac{r}{\rho}\right)\frac{\partial}{\partial r}+\triangle_{S(0;r)}\,,
\end{split}
\end{equation}
where the first expression is the hyperbolic Laplacian in $H^n_\rho$ expressed by using Euclidean rectangular coordinates and the second expression represents  the hyperbolic Laplacian in $B^n_\rho$ expressed in the geodesic hyperbolic polar coordinates, where $\triangle_{S(0;r)}$ is the Laplacian on the geodesic sphere of a hyperbolic radius $r$ about the origin.

The next step is to show that $\omega^\alpha$ as an eigenfunction of the Hyperbolic Laplacian in $B^n_{\rho}$.

\begin{proposition}\label{eigenfunction-omega-proposition}
Let $u$ be any point of $S(\rho)$ and $m=(X,T)\in B(\rho)$, where $S(O;\rho)$ and $B(\rho)$ are the Euclidean sphere and ball, respectively, both of the same radius $\rho$ centered at the origin $O$. Let $k=n-1$ and
\begin{equation}\label{Uniq-11}
    \omega=\omega(u,m)=\frac{\rho^2-\eta^2}{|u-m|^2}=\frac{|u|^2-|m|^2}{|u-m|^2}\,.
\end{equation}
Then
\begin{equation}\label{Omega-as-Eigenfunction}
    \triangle_m \omega^\alpha+\frac{\alpha k-\alpha^2}{\rho^2}\, \omega^\alpha=0\,.
\end{equation}
\end{proposition}

\begin{proof}[Proof of the Proposition.]
Notice that
\begin{equation}\label{Uniq-12}
    \triangle t^\alpha+\frac{\alpha k-\alpha^2}{\rho^2} t^\alpha=0
\end{equation}
and the relationship between $H^n_\rho$ and $B^n_\rho$ is given by Cayley Transform \\
 $K:B^n_\rho\rightarrow H^n_\rho$ expressed by the following formulae
\begin{equation}\label{Uniq-13}
\begin{split}
& x=\frac{2\rho X}{|X|^2+(T-\rho)^2}
\\& t=\frac{\rho^2-|X|^2-T^2}{|X|^2+(T-\rho)^2}=\frac{\rho^2-\eta^2}{|u-m|^2}=\omega(u,m)\,,
\end{split}
\end{equation}
where $u=(0,\rho)\in\partial B(O;\rho)=S(O;\rho)$. Since Cayley transform is an isometry, an orthogonal system of geodesics defining the Laplacian in $B^n_\rho$ is mapped isometrically to an orthogonal system of geodesics in $H^n_\rho$ and then,
\begin{equation}\label{Uniq-14}
    \triangle t^\alpha=\triangle_m\omega^\alpha(u,m)\quad \text{and}\quad t^\alpha=\omega^\alpha(u,m)\,.
\end{equation}
Therefore, equation \eqref{Omega-as-Eigenfunction} is precisely equation \eqref{Uniq-12} written in $B^n_\rho$, which completes the proof of Proposition~\ref{eigenfunction-omega-proposition}.
\end{proof}



\subsection {Explicit representation of radial eigenfunctions.}

In this subsection we shall see that every radial eigenfunction in $B^{k+1}_{\rho}$ depending only on the distance from the origin has an explicit integral representation.

\begin{definition}\label{sharp-definition} Recall that $B^{k+1}_\rho$ is the ball model of the hyperbolic space~$H^{k+1}_\rho$ with the sectional curvature $\kappa=-1/\rho^2$ and let $B^{k+1}(O,\rho)$ be the Euclidean ball of radius $\rho$ centered at the origin and represents the ball model $B^{k+1}_\rho$. Suppose that $f$ is a function on $B^{k+1}_\rho$. We define its radialization about the origin $O$, written $f^\sharp_O (m)$, by setting
\begin{equation}\label{Radialization-Definition}
    f^\sharp_O(m)=\frac{1}{|S^k(|m|)|}\int\limits_{S^k(|m|)} f(m_1) dS_{m_1}\,,
\end{equation}
where the integration is considered with respect to the measure on $S^k(|m|)$ induced by the Euclidean metric of $\mathbb{R}^n\supset B^{k+1}(O,\rho)$; $|m|$ is the Euclidean distance between the origin $O$ and a point $m\in B^{k+1}_\rho$; $|S^k(|m|)|$ is the Euclidean volume of $S^k(|m|)$.
\end{definition}

\begin{lemma}\label{Laplacian-Commutes-lemma}
\begin{equation}\label{Laplacian-Commutes}
    \triangle_m f^\sharp_O(m)=\frac{1}{|S^k(|m|)|}\int\limits_{S^k(|m|)}\triangle_{m_1} f(m_1) dS_{m_1}\,.
\end{equation}
\end{lemma}

\begin{remark}
     A proof of \eqref{Laplacian-Commutes} involving substantial amount of elementary computations will be given in the Appendix, see Lemma~\ref{Laplacian-Commutes-lemma-Appendix}, p.~\pageref{Laplacian-Commutes-lemma-Appendix}. Here we recall a geometrical proof relying on Haar measure which yields~\eqref{Laplacian-Commutes} directly.
\end{remark}

\begin{proof}[Proof of Lemma~\ref{Laplacian-Commutes-lemma}]
Let us introduce two measures on $S^k$. One is the Lebesque measure of the sphere, i.e. the Riemannian volume of the flat metric of the ambient space restricted to $S^k$ and the other one is the Haar measure of $\text{SO}(k+1,R)$ pushed forward to the sphere by the map $T\rightarrow Tv_0$, where $v_0$ is a fixed vector on the sphere and $T\in\text{SO}(k+1,R)$. These two measures agree up to constant multiple factor so that these two averaging procedures must be the same. It is clear that Laplacian commutes with averaging over $\text{SO}(k+1,R)$, since any individual isometry of $\mathbb{R}^{k+1}$ commutes with Laplacian. Therefore, Laplacian must commute with radialization defined in \eqref{Radialization-Definition}. This completes the proof of Lemma~\ref{Laplacian-Commutes-lemma}.
\end{proof}



The next step is to obtain the explicit representation for radial eigenfunctions.

\begin{definition} Let $u\in S^k(\rho)$ be a fixed point and $|m_1|=|m|=\eta<\rho=|u|$. Then let us define
\begin{equation}\label{Radialization-of-omega}
    V_\alpha(\eta)=(\omega_O^\alpha)^\sharp(m)=
    \frac{1}{|S(\eta)|}\int\limits_{S(\eta)}\omega^\alpha(u,m_1)dS_{m_1}\,,
\end{equation}
where $\alpha$ is a complex number. Thus, $V_\alpha(\eta)$ is the radialization of $\omega^\alpha(u,m)$ about the origin.
\end{definition}




\begin{theorem}\label{Radial-func-repre-Thoerem} Let $r$ be the hyperbolic distance between the origin $O$ and $S(\eta)$. Then, the following function
\begin{equation}\label{Radial-func-presentation}
    \varphi_\mu(r)= V_\alpha(\eta(r))= V_\alpha\left(\rho\tanh\left(\frac{r}{2\rho}\right)\right)\,,
\end{equation}
where $\mu=(\alpha k-\alpha^2)/\rho^2$, is the unique radial eigenfunction assuming the value $1$ at the origin and corresponding to an eigenvalue $\mu$, i.e.,

\begin{equation}\label{Equation-for-Radi-Eigenfunction}
    \triangle\varphi_\mu(r)+\mu\varphi_\mu(r)=0.
\end{equation}
\end{theorem}

\begin{proof}
    Recall that $\eta=|m|$ is the Euclidean distance between $m=(X,T)\in B(\rho)$ and the origin, while $r=r(m)=r(\eta)$ is the hyperbolic distance between the origin and $m$. Therefore, the relationship between $r$ and $\eta$ is
\begin{equation}\label{Eucli-Hype-Coordinates-relationship}
    r=\rho\ln\frac{\rho+\eta}{\rho-\eta}\quad\text{or}\quad\eta=\rho\tanh\left(\frac{r}{2\rho}\right)\,,
\end{equation}
which justifies the last expression in \eqref{Radial-func-presentation}.

Recall also that according to \eqref{Omega-as-Eigenfunction}, $\omega^\alpha$ is the eigenfunction of the hyperbolic Laplacian with the eigenvalue $(\alpha k-\alpha^2)/\rho^2$. Therefore, according to Lemma \ref{Laplacian-Commutes-lemma}, p.~\pageref{Laplacian-Commutes-lemma}, the radialization of $\omega^\alpha$ defined in \eqref{Radialization-of-omega} is also an eigenfunction with the same eigenvalue.

The last step is to show that the eigenfunction $\varphi_\mu(r)$ assuming the value $1$ at the origin is unique. We obtain the uniqueness by the procedure described in \cite{Olver}, pp.~148-153 or in \cite{Chavel}, p.~272 resulting that the general radial eigenfunction is on the form $c_1f_1(r)+c_2f_2(r)$, where $f_1(r)$ is an entire function of $r$ with $f_1(0)=1$, and $f_2(r)$ is defined for all $r>0$, with a singularity at $r=0$. Therefore, for a radial eigenfunction to have the value $1$ at the origin, we must set $c_2=0$ and then, we can see that the radial eigenfunction is defined uniquely by its value at the origin and by the eigenvalue. Observing that $\varphi_\mu(0)=V_{\alpha}(0)=1$ completes the proof of Theorem~\ref{Radial-func-repre-Thoerem}.
\end{proof}

\subsection {The Hyperbolic Geometry interpretation of Statement (A) of Theorem 2.}

Here we give another proof of (A) of Theorem~\ref{Basic-theorem}, p.~\pageref{Basic-theorem} using uniqueness of radial eigenfunctions Theorem~\ref{Radial-func-repre-Thoerem}. Recall Statement (A) of Theorem~\ref{Basic-theorem} from p.~\pageref{Basic-theorem}.

\begin{proposition}[Statement (A).]\label{Statement-A'-Proposition} Let $\alpha, \beta$ be complex numbers such that $\alpha+\beta=k$; let $\rho, \eta$ be two positive numbers such that $\rho\neq\eta$; let $m\in S^k(\eta), u\in S^k(\rho)$. Then
\begin{equation}\label{Uniq-27}
   \int\limits_{S^k(\eta)}\omega^\alpha(m,u)dS_m=
   \int\limits_{S^k(\eta)}\omega^\beta(m,u)dS_m
\end{equation}
\end{proposition}

\begin{proof}[Proof of Proposition \ref{Statement-A'-Proposition}.] Let us observe that the equivalent form of the identity in \eqref{Uniq-27} can be written as $V_\alpha(|m|)=V_\beta(|m|)$, where $V$ was defined in \eqref{Radialization-of-omega}. Notice also that $V_\alpha(|m|)$ as well as $V_\beta(|m|)$, according to \eqref{Laplacian-Commutes}, are the radial eigenfunctions of the Hyperbolic Laplacian. Clearly that the condition $\alpha+\beta=k$ implies
\begin{equation}\label{Uniq-28}
   \frac{\alpha(\alpha-k)}{\rho^2}=\frac{\beta(\beta-k)}{\rho^2}\,,
\end{equation}
which means that $V_\alpha(|m|)$ and $V_\beta(|m|)$ have the same eigenvalue. In addition, the direct computation shows that $V_\alpha(0)=V_\beta(0)=1$. According to Theorem~\ref{Radial-func-repre-Thoerem}, p.~\pageref{Radial-func-repre-Thoerem}, there exists exactly one radial eigenfunction with such properties. Therefore, $V_\alpha(|m|)=V_\beta(|m|)$ for all $|m|=\eta<\rho$, which is equivalent to \eqref{Uniq-27} for all $\eta<\rho$. To see that \eqref{Uniq-27} remains true for $\eta>\rho$, we can apply Lemma~\ref{Basic_lemma} from p.~\pageref{Basic_lemma}. This completes the proof of Proposition~\ref{Statement-A'-Proposition}.
\end{proof}

\section {One Radius Theorem.}

In this section we shall see that if $\varphi_\mu(r)=\varphi_\nu(r)$ for $r=0$ and for some $r=r_0$ chosen sufficiently close to the origin, then $\mu=\nu$ and $\varphi_\mu(r)=\varphi_\nu(r)$ for all $r\in[0,\infty)$. In other words, we shall introduce a condition for an eigenvalue under which the eigenfunction is uniquely determined just by its value at $\emph{one}$ point sufficiently close to the origin. Moreover, we shall see that with an additional condition for the eigenvalue, a radial eigenfunction is uniquely determined by its value at an arbitrary point.

\begin{proposition}\label{Parabola-Strip-equivalence-Lemma} Let $\alpha\in\mathbb{C}$ and $\mu=\Phi(\alpha)$, where $\Phi:\mathbb{C}\rightarrow \mathbb{C}$ is given by $\Phi(\alpha)=-\kappa(\alpha k-\alpha^2)$, where $\kappa<0, k>0$ are constants. Then
\begin{equation}\label{One_Rad_1}
   \Im(\alpha)\in[-p,p]\iff \mu\in\{\text{Shaded area inside the parabola bellow}\}
\end{equation}

\begin{figure}[!h]
    \centering
    \epsfig{figure=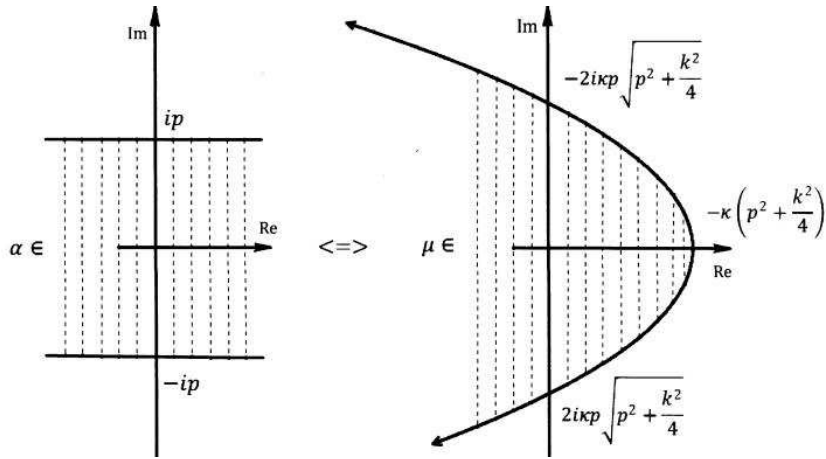,height=6cm}\\
  \caption{Parabola and Strip equivalence.}\label{Parabola-Strip-Equiv}
\end{figure}

\end{proposition}

\begin{proof}[Proof of Proposition \ref{Parabola-Strip-equivalence-Lemma}.]  Let $\alpha=a+ib; a_1=a-k/2$ and $t=a_1+ib$. Then $\alpha=a+ib=k/2+t$. The direct computation shows that $\Phi(k/2+t)=\Phi(k/2-t)$, which means that the image of the following strip
\begin{equation}\label{One_Rad_2}
   \Xi(-p,p)=\{\alpha\mid\Im(\alpha)\in[-p,p]\}
\end{equation}
must be the same as the image of the upper half of the strip above. I.e.,
\begin{equation}\label{One_Rad_3}
   \Phi(\Xi(-p,p))=\Phi(\Xi(0,p))\,,
\end{equation}
where
\begin{equation}\label{One_Rad_4}
   \Xi(0,p)=\{\alpha\mid\Im(\alpha)\in[0,p]\}\,.
\end{equation}
To figure out what is the image of $\Xi(0,p)$, it is enough to look at the image of a horizontal line from $\Xi(0,p)$. Let us choose the boundary line: $\alpha=\alpha(a_1)=k/2+ip+a_1$, where $p$ is fixed and $a_1$ serves as a parameter. Then,
\begin{equation}\label{One_Rad_5}
\begin{split}
   & \Phi(\alpha(a_1))=-\kappa\left(\frac{k^2}{4}+p^2 \right)+\kappa a_1^2+2i\kappa p a_1
   \\& =-\kappa\left(\frac{k^2}{4}+p^2\right)+\frac{1}{4\kappa p^2}(\Im\Phi(\alpha))^2+i\Im\Phi(\alpha)\,,
\end{split}
\end{equation}
where, clearly, the imaginary part depends on $a_1$ linearly, while the real part depends quadratically. Therefore, the horizontal line chosen is mapped to a parabola. Any parabola is uniquely determined by three arbitrary points on it. Thus, to determine this parabola, we find all points of intersections with the coordinate axis. If $a_1=0$, then
\begin{equation}\label{One_Rad_6}
   \Phi(\alpha(0))=-\kappa\left(\frac{k^2}{4}+p^2 \right)
\end{equation}
gives the point of intersection the parabola with the horizontal coordinate axis. If $a_1=\pm\sqrt{p^2+k^2/4}$, then
\begin{equation}\label{One_Rad_7}
   \Phi(\alpha(\pm\sqrt{p^2+k^2/4}))=\pm2i\kappa p\sqrt{p^2+k^2/4}\,,
\end{equation}
what is denoted on Figure \ref{Parabola-Strip-Equiv} above. Therefore, both of these two parallel lines $\alpha(a_1)=\pm ip+k/2+a_1$ are mapped to the parabola described above. Notice also that this parabola is symmetric with respect to the horizontal axis. In addition, the smaller the parameter $p$, the narrower the parabola and its tip is closer to the origin.
So, the proof of Proposition~\ref{Parabola-Strip-equivalence-Lemma} is complete.
\end{proof}

\begin{remark}
Note that if the parameter $p$ is zero, then the parabola is folded up to the half of the real line $(-\infty, -\kappa k^2/4]$, which is the image of the real line.
\end{remark}

\subsection {One Radius Theorem for radial eigenfunctions.}

\begin{theorem}[One Radius Theorem]\label{One-radius-Theorem}

Let $\mu\neq\nu$ and $\varphi_\nu(r)$, $\varphi_\mu(r)$ be two radial eigenfunctions for the hyperbolic Laplacian given by
\begin{equation}\label{One_Rad_8}
    \vartriangle \varphi(r)=\varphi^{''}(r)+\frac{k}{\rho}\coth\left(\frac{r}{\rho}\right)\varphi^{'}(r) \,,\,\text{where}\,\,0\leq r<\infty
\end{equation}
in a ball model of a hyperbolic $(k+1)-$ dimensional space of a constant sectional curvature $\kappa=-1/\rho^2<0$. Let $\varphi_\nu(0)=\varphi_\mu(0)=1$.
Then $\varphi_\nu(r)\neq\varphi_\mu(r)$ for every $r\in(0, \pi\rho/(2p)]$, where $p=\max\{|\Im(\alpha)|, |\Im(\beta)|\}$ and $\alpha, \beta$ are complex numbers related to $\mu, \nu$ by the following quadratic equations
\begin{equation}\label{One_Rad_9}
    \mu=-\kappa(\alpha k-\alpha^2) \quad \text{and} \quad \nu=-\kappa(\beta k-\beta^2)\,.
\end{equation}
\end{theorem}

\begin{remark} Recall that the first quadratic equation above together with \eqref{One_Rad_5} implies that
\begin{equation}\label{One_Rad_10}
    \Im(\alpha)\in [-p, p] \iff
    \Re\mu\leq-\kappa\left(p^2+\frac{k^2}{4}\right)+\frac{1}{4\kappa p^2}(\Im\mu)^2\,,
\end{equation}
i.e., $\mu$ belongs to the inner part of the parabola depending on parameter $p$ and pictured on Figure~\ref{Parabola-Strip-Equiv}, p.~\pageref{Parabola-Strip-Equiv}.

 If $p=0$, we arrive at the case when $\alpha, \beta$ are real or, equivalently, $\nu, \mu\leq -\kappa k^2/4$, which, according to the One Radius Theorem, means that the condition $\mu\neq\nu$ implies that $\varphi_\nu(r)\neq\varphi_\mu(r)$ for every $r\in (0, \pi\rho/(2p)]_{p=0}=(0,\infty)$.
\end{remark}

\begin{proof}[Proof of Theorem~\ref{One-radius-Theorem}.]

Let us assume that $\varphi_\mu(r)=\varphi_\nu(r)$ for some $r\in(0, \pi\rho/(2p)]$. Recall that according to Theorem~\ref{Radial-func-repre-Thoerem} from p.~\pageref{Radial-func-repre-Thoerem}, a radial eigenfunction in $B^{k+1}_\rho$ assuming the value~$1$ at the origin is uniquely defined by its eigenvalue and this eigenfunction can be expressed by using the Euclidean coordinates as
\begin{equation}\label{One_Rad_11}
   \varphi_\mu(r)=V_\alpha(\eta(r))=\frac{1}{|S(\eta)|}\int\limits_{S(\eta)}\omega^\alpha(u,m)dS_m\,,
\end{equation}
where $\mu=-\kappa(\alpha k-\alpha^2);\,\, |m|=\eta;\,\, r=\rho\ln[(\rho+\eta)/(\rho-\eta)]$ and $|u|=\rho$. Let us set $x=u\notin S^k(\eta)$ and $y=m\in S^k(\eta)$ in Statement (D) of Theorem~\ref{Basic-theorem}, p.~\pageref{Basic-theorem}.
Using \eqref{One_Rad_11} we can observe that
\begin{equation}\label{One_Rad_12}
   \varphi_\nu(r)=\varphi_\mu(r)\quad\text{for some}\,\,r=\rho\ln\frac{\rho+\eta}{\rho-\eta}\in(0, \pi\rho/2p]
\end{equation}
is equivalent to say that there exists $\eta=\rho\tanh(r/2\rho)$ such that
\begin{equation}\label{One_Rad_13}
\begin{split}
   & \int\limits_{S(\eta)}\omega^\alpha(u,m)dS_m=\int\limits_{S(\eta)}\omega^\beta(u,m)dS_m\quad\text{and}
   \\& p=\max\{|\Im(\alpha)|,|\Im(\beta)|\}\leq\left.\frac\pi2\right/\ln\frac{\rho+\eta}{\rho-\eta}\,.
\end{split}
\end{equation}
It is clear that \eqref{One_Rad_13}, according to Statement~(D) of Main Theorem~\ref{Basic-theorem}, p.~\pageref{Basic-theorem} yields $\alpha+\beta=k$ or $\alpha=\beta$, which implies that $\mu=\nu$. The last equality contradicts to the assumption of the theorem, which completes the proof of Theorem~\ref{One-radius-Theorem}.
\end{proof}




\begin{corollary} Using \eqref{One_Rad_10}, we observe that
\begin{equation}\label{One_Rad_14}
   \Omega(p)=
   \left\{\varphi_\mu(r)\mid\varphi_\mu(0)=1\,\,\text{and}\,\,\,\Re\mu\leq-\kappa\left(p^2+\frac{k^2}{4}\right)+\frac{(\Im\mu)^2}{4\kappa p^2} \right\}
\end{equation}
is the set of all radial eigenfunctions such that $\varphi_\mu(0)=1$ and with eigenvalues within the shaded area inside the parabola depending on $p$. The parabola was pictured on Figure~\ref{Parabola-Strip-Equiv}, p.~\pageref{Parabola-Strip-Equiv}. Then the value of $\varphi_\mu(r_0)$ at any point $r_0\in(0,\pi\rho/(2p)]$ determines the radial eigenfunction $\varphi_\mu(r)\in\Omega(p)$ uniquely. Meanwhile, according to Statement (C) of Theorem~\ref{Basic-theorem} form p.~\pageref{Basic-theorem}, there are infinitely many eigenfunctions $\varphi_\nu(r)\notin\Omega(p)$ such that $\varphi_\nu(r_0)=\varphi_\mu(r_0)$.
\end{corollary}

\begin{corollary} If $p=0$ or, equivalently, if $\Omega(0)$ is defined as
\begin{equation}\label{One_Rad_15}
   \Omega(0)=
   \left\{\varphi_\mu(r)\mid\varphi_\mu(0)=1\quad\text{and}\quad\mu\leq\frac{-\kappa k^2}{4}\right\}\,,
\end{equation}
then the value $\varphi_\mu(r_0)$ defines $\varphi_\mu(r)\in\Omega(0)$ uniquely for any $r_0\in(0,\infty)$. And again, by Statement (C) of Theorem~\ref{Basic-theorem} form p.~\pageref{Basic-theorem}, there are infinitely many radial eigenfunctions $\varphi_\nu(r)\notin\Omega(0)$ such that $\varphi_\nu(r_0)=\varphi_\mu(r_0)$.
\end{corollary}

\section {Radial eigenfunctions vanishing \\ at some finite point.}

In this section we describe the radial eigenfunctions corresponding to real eigenvalues and vanishing at some finite radius $r$. We obtain also all radial eigenfunctions together with their eigenvalues for the Dirichlet Eigenvalue Problem in a hyperbolic 3-dimensional disk.

\subsection {Representations of a radial eigenfunction \\ vanishing  at some finite point.}

We shall see here the explicit representation of a radial eigenfunction corresponding to a real eigenvalue and vanishing at a finite radius.

\begin{theorem}\label{Radial-Vanishing-EF-Theorem} Let $\lambda$ is real. Then $\varphi_\lambda(r)$ is a radial eigenfunction vanishing at some $r<\infty$ if and only if $\lambda>-\kappa k^2/4$ or equivalently, if
\begin{equation}\label{Real-Lambda-and-alpha-b-relation}
   \lambda=\frac{\alpha k-\alpha^2}{\rho^2}\quad\text{with}\quad \alpha=\frac k2+ib\quad\text{and}\quad b\neq0\,.
\end{equation}
Every vanishing radial eigenfunction changes sign infinitely many times, vanishes at infinity and for every
fixed $\lambda>-\kappa k^2/4$ there exist three positive constants $T_1, T_2$ and $\Lambda_0$ such that
$T_1<|r_2-r_1|<T_2$ for every two successive zeros $r_1, r_2$ of $\varphi_\lambda(r)$ satisfying $r_2>r_1>\Lambda_0$.
Moreover,
\begin{equation}\label{Radial-Vanishing-EF}
    \varphi_\lambda(r(\eta))=\frac{1}{|S(\eta)|}\int\limits_{S(\eta)}\omega^{k/2\pm ib}(u,m)dS_m\,,
\end{equation}
where $b=\sqrt{\frac{\lambda}{-\kappa}-\frac{k^2}{4}}>0$
or, equivalently,
\begin{equation}\label{Radial-Vanishing-EF-cos}
    \varphi_\lambda(r(\eta))=\frac{1}{|S(\eta)|}\int\limits_{S(\eta)}\omega^{k/2}\cos\left( \sqrt{\frac{\lambda}{-\kappa}-\frac{k^2}{4}}\ln\omega \right) dS_m\,,
\end{equation}
where $\rho=|u|$, $\eta=|m|$, $r(\eta)=\rho\ln\left[(\rho+\eta)/(\rho-\eta)\right]$, $\theta=\pi-\widehat{uOm}$ and
\begin{equation}\label{Three_D_Dirichlet-2}
    \omega=\omega(u,m)=\frac{\rho^2-\eta^2}{|u-m|^2}=
    \frac{\rho^2-\eta^2}{\rho^2+\eta^2+2\eta\rho\cos\theta}=\omega(\eta,\theta)\,.
\end{equation}
\end{theorem}

\begin{proof} Let us assume first that $\varphi_\lambda(r)=0$ for some finite $r$. Recall that according to Theorem~\ref{Radial-func-repre-Thoerem} from p.~\pageref{Radial-func-repre-Thoerem}, the unique radial eigenfunction with an eigenvalue $\lambda\in\mathbb{C}$ is
\begin{equation}\label{Three_D_Dirichlet-3}
    \varphi_\lambda(r)=V_\alpha(\eta(r))=\frac{1}{|S(\eta)|}\int\limits_{S(\eta)}\omega^\alpha dS_m\,,
\end{equation}
where
\begin{equation}\label{Three_D_Dirichlet-4}
    |m|=\eta(r)=\rho\tanh\left(\frac{r}{2\rho}\right)\quad\text{and}\quad\lambda=
    \frac{\alpha k-\alpha^2}{\rho^2}=-\kappa(\alpha k-\alpha^2)\,.
\end{equation}
Let $\alpha=a+ib$. Then
\begin{equation}\label{Three_D_Dirichlet-5}
    \lambda=-\kappa\left( (a+ib)k-(a+ib)^2\right)=-\kappa\left[ak+b^2-a^2+ib(k-2a)\right]\,.
\end{equation}
Therefore, $\lambda$ is real if and only if
\begin{equation}\label{Three_D_Dirichlet-6}
    \alpha\in\gimel
    =\{(a+ib)|\,\, a=k/2\,\, \text{or}\,\, b=0\}
\end{equation}
or equivalently, if
\begin{equation}\label{Iff-conditions-for-Lambda-real}
    \alpha=\frac k2\pm s\quad\text{or}\quad\alpha=\frac k2\pm ib\,,\quad\text{where}\,\,s,b\in\mathbb{R}\,.
\end{equation}
It is clear that $\gimel$ is mapped to the real line $\lambda$ as it is pictured on Figure~\ref{Cross-Image} below.

\begin{figure}[!h]
    \centering
    \epsfig{figure=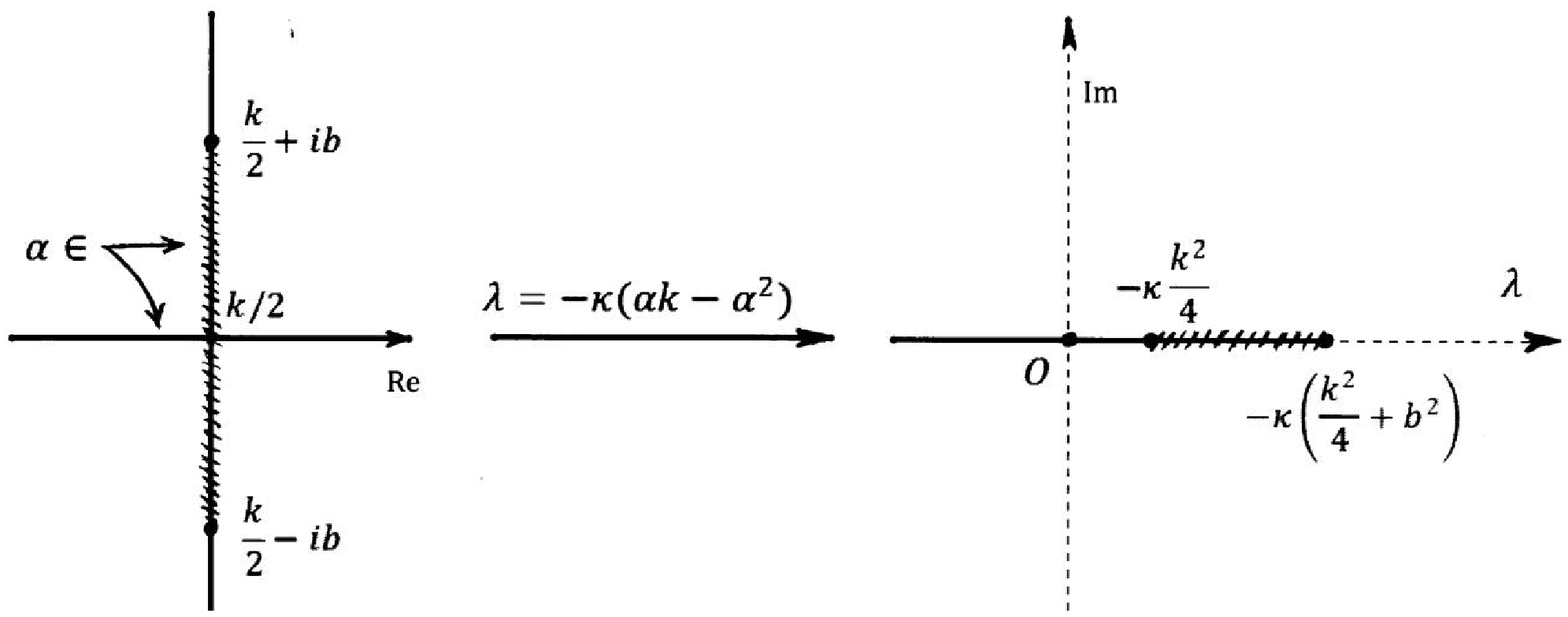,height=5cm}\\
  \caption{The image of the cross $\gimel$.}\label{Cross-Image}
\end{figure}

Using \eqref{Three_D_Dirichlet-5} and \eqref{Three_D_Dirichlet-6} we observe that $\alpha$ is real if and only if
\begin{equation}\label{Three_D_Dirichlet-7}
    \lambda\in\mathbb{R}\quad\text{and}\quad\lambda\leq-\kappa k^2/4\,.
\end{equation}
Therefore, for every $\lambda$ satisfying \eqref{Three_D_Dirichlet-7} the radial eigenfunction does not vanish for any finite $r$ because
\begin{equation}\label{Three_D_Dirichlet-8}
    \varphi_\lambda(r)=V_\alpha(\eta(r))=\frac{1}{|S(\eta)|}\int\limits_{S(\eta)}\omega^\alpha dS_m>0
\end{equation}
for every $\eta\in[0,\rho)$ and for every real $\alpha$. As we shall see later, in Theorem~\ref{Infinity-Behavior-radial-Theorem}, p.~\pageref{Infinity-Behavior-radial-Theorem}, $\lim\limits_{r\rightarrow\infty}\varphi_\lambda(r)=0$ for every $\lambda\in(0,\infty)$, but up to now, $\delta$ is supposed to be finite. Thus, to allow the radial eigenfunction to vanish at a finite $r=\delta$, $\lambda$ must be strictly greater than $-\kappa k^2/4$. Again, using \eqref{Three_D_Dirichlet-5} and \eqref{Three_D_Dirichlet-6} we can observe that
\begin{equation}\label{Three_D_Dirichlet-9}
    \lambda>-\kappa \frac{k^2}{4}\quad
\end{equation}
which yields $\eqref{Radial-Vanishing-EF}$. Note also that $\lambda=-\kappa\left(\alpha k-\alpha^2\right)$ and $\alpha=k/2+ib$ imply together that $b=\pm\sqrt{-\lambda/\kappa-k^2/4}$. Formula~\eqref{imaginary-part-vanishing} of Corollary~\ref{imaginary-part-vanishing-corollary}, p.~\pageref{imaginary-part-vanishing-corollary}, applied to \eqref{Radial-Vanishing-EF} leads directly to~\eqref{Radial-Vanishing-EF-cos}. Therefore, for a vanishing radial eigenfunction all the formulae \eqref{Real-Lambda-and-alpha-b-relation}, \eqref{Radial-Vanishing-EF}, \eqref{Radial-Vanishing-EF-cos}, p.~\pageref{Radial-Vanishing-EF-Theorem} as well as $\lambda>-\kappa k^2/4$ must hold.

The next step is to show that $\varphi_\lambda(r)$ has infinitely many zeros as long as $\lambda>-\kappa k^2/4$. Recall that $\varphi_\lambda(r)$ is a solution for the following equation
\begin{equation}\label{radial-laplac-Recall}
    Y''+\frac k\rho\coth\left(\frac r\rho\right)Y'+\lambda Y=0.
\end{equation}
The direct computation shows that the substitution $Y=U(\sinh(r/\rho))^{-k/2}$ suggested in~\cite{Simmons}, p.~119,   reduces~\eqref{radial-laplac-Recall} to
$U''+QU=0$, where
\begin{equation}\label{Expression-q-after-substitution}
    Q=Q(r)=\lambda-\frac{k^2}{4\rho^2}-\frac{k(k-2)}{4\rho^2\sinh^2(r/\rho)}
\end{equation}
It is clear that
\begin{equation}\label{limit-for-q(r)}
    \lim\limits_{r\rightarrow\infty}Q(r)=\lambda+\frac{\kappa k^2}{4}\,,\quad\text{where}\quad \kappa=-\frac{1}{\rho^2}\,.
\end{equation}
Therefore, for $\lambda>-\kappa k^2/4$ there exist positive numbers $\tau$ and $\Lambda$ such that
\begin{equation}\label{Inequalities-for-q-on-large-r}
    Q(r)>\tau>0\quad\text{for every}\quad r>\Lambda.
\end{equation}
Let us introduce $Z(r)$ as a solution of the following elementary equation $Z''+\tau Z=0$. According to the Sturm comparison theorem in \cite{Simmons}, p.~122, condition \eqref{Inequalities-for-q-on-large-r} implies that $U(r)$ vanishes at least once between any two successive zeros of $Z(r)$. It clear that $Z(r)$ has infinitely many zeros. Therefore, $U(r)$ as well as $\varphi_\lambda(r)$ must also have infinitely many zeros as $r\rightarrow\infty$. In addition, as a consequence of Theorem~\ref{Infinity-Behavior-radial-Theorem}, p.~\pageref{Infinity-Behavior-radial-Theorem}, we shall see that $\varphi_\lambda(r)$ vanishes at $\infty$.

To see that there exist constants $T_1, T_2, \Lambda_0$ let us recall the following Lemma from~\cite{Simmons}, p.~134.

\begin{lemma}[Simmons, \cite{Simmons}, p. 134, (Lemma 2)]\label{Simmons-zero-distance-estim-Lemma}
    Let $Q(r)$ be a positive continuous function that satisfies the inequalities
    $$0<M_1^2<Q(r)<M_2^2$$
    on a closed interval $[A,B]$. If $Y(r)$ is a nontrivial solution of $Y''+Q(r)Y=0$ on this interval, and if $r_1$ and $r_2$ are successive zeros of $Y(r)$, then
    $$\pi/M_2<r_2-r_1<\pi/M_1\,.$$
    Furthermore, if $Y(r)$ vanishes at $A$ and $B$, and at $n-1$ points in the open interval $(A,B)$, then
    $$M_1(B-A)/\pi<n<M_2(B-A)/\pi\,.$$
\end{lemma}

Since $Q(r)\rightarrow\lambda+\kappa k^2/4>0$ as $r\rightarrow\infty$, then the existence of $M_1, M_2$ for large $r>\Lambda_0$ follows and the conditions of Lemma~\ref{Simmons-zero-distance-estim-Lemma} are satisfied. Therefore, if we denote $T_1=\pi/M_2$ and $T_2=\pi/M_1$, we have
\begin{equation}\label{zero-distance-estimation-simmons}
    T_1=\pi/M_2<r_2-r_1<\pi/M_1=T_2
\end{equation}
for every $[A,B]\subset[\Lambda_0,\infty]$. Observe now that the estimation~\eqref{zero-distance-estimation-simmons} is independent on $A,B>\Lambda_0$, since $M_1, M_2$ can be fixed for all $r_i>\Lambda_0\,\,(i=1,2)$ and then \eqref{zero-distance-estimation-simmons} remains true for all $r_1, r_2>\Lambda_0$. This completes the proof of Theorem~\ref{Radial-func-repre-Thoerem}, p.~\pageref{Radial-func-repre-Thoerem}.
\end{proof}

\subsection {Geometric representations of a radial \\ eigenfunction vanishing at some finite point.}

The formulae presented in Theorem \ref{Radial-Vanishing-EF-Theorem}, p.~\pageref{Radial-Vanishing-EF-Theorem}, lead us to the following geometric presentations of radial eigenfunctions. All notations used in the following lemma are pictured on Figure~\ref{Euclidean-Notations-in-HD} below.

\begin{lemma}\label{Lemma-532} Geometric representation of radial eigenfunctions can be described as follows.

\begin{description}
  \item[\textbf{(A)}] A radial eigenfunction can be expressed as
\begin{equation}\label{Three_D_Dirichlet-15}
   \varphi_\lambda(r(\eta))=\frac{1}{|S(\rho)|}
   \int\limits_{S(\rho)}\omega^\alpha dS_u=
   \frac{1}{|S(\rho)|}
   \int\limits_{\Sigma}\left(\frac{l}{q}\right)^\alpha\frac{2\rho}{l+q}q^k d\Sigma_{\widetilde{u}}\,,
\end{equation}
where $\alpha$ is chosen in such a way that the eigenvalue $\lambda=(\alpha k-\alpha^2)/\rho^2$.

  \item[\textbf{(B)}] If a radial eigenfunction has a real eigenvalue and vanishes at some finite radius $r_0<\infty$ or, equivalently, by Theorem~\ref{Radial-Vanishing-EF-Theorem}, p.~\pageref{Radial-Vanishing-EF-Theorem}, if $\alpha=k/2+ib$ with $b\neq0$, then
\begin{equation}\label{Three_D_Dirichlet-16}
    \varphi_\lambda(r(\eta))
    =\frac{1}{|S(\rho)|}\int\limits_{S(\rho)}\left(\frac{l}{q}\right)^{k/2}
    \cos\left(b\ln\frac{l}{q}\right)dS_{u} \,,
\end{equation}
where $u\in S^k(\rho)$ is the parameter of integration,
\begin{equation}\label{Three_D_Dirichlet-17}
    \frac{l}{q}=\frac{|m-u|}{|m-u^*|}=\omega(u,m)\quad\text{and}\quad
    b=\pm\sqrt{\frac{\lambda}{-\kappa}-\frac{k^2}{4}}\,.
\end{equation}

\item[\textbf{(C)}] The integral given in \eqref{Three_D_Dirichlet-16} may be transformed to the following form
\begin{equation}\label{Three_D_Dirichlet-18}
    \varphi_\lambda(r(\eta))
    =\frac{4\rho(\rho^2-\eta^2)^{k/2}\sigma_{k-1}}{|S(\rho)|}\int\limits_0^{\pi/2}
    \frac{\sin^{k-1}\psi}{l+q}\cos\left(b\ln\frac{l}{q}\right)d\psi\,,
\end{equation}
which also represent a vanishing radial eigenfunction, with $b$ defined in~\eqref{Three_D_Dirichlet-17}.
\end{description}
\end{lemma}

\begin{remark}
We shall see that the last integral formula can be simplified using integration by parts for $k>1$. In particular, for $k=2$ this integral can be computed explicitly.
\end{remark}

\begin{proof}[Proof of Lemma \ref{Lemma-532}, Statement (A)] Notice that the restrictions $u\in S(\rho)$ and $m\in S(\eta)$ make $\omega^\alpha(u,m)$ a function depending only on the distance between $u$ and $m$, since
\begin{equation}\label{Three_D_Dirichlet-19}
    \omega^\alpha(u,m)=(\rho^2-\eta^2)^\alpha\cdot\frac{1}{|u-m|^{2\alpha}}=g(|u-m|)\,.
\end{equation}
Therefore, we can apply the Lemma~\ref{Basic_lemma}, p.~\pageref{Basic_lemma} to $\omega^\alpha(u,m)$. Let $u\in S(\rho)$ and $m_1\in S(\eta)$ pictured below serve as the parameters of integration, while $u_1\in S^k(\rho)$ is a fixed point. Let $\Sigma$ be the unit sphere centered at point $m$ and $u^*$ is defined as the intersection of $\Sigma$ and segment $mu$; $\psi=\angle Omu$.

\begin{figure}[!h]
    \centering
    \epsfig{figure=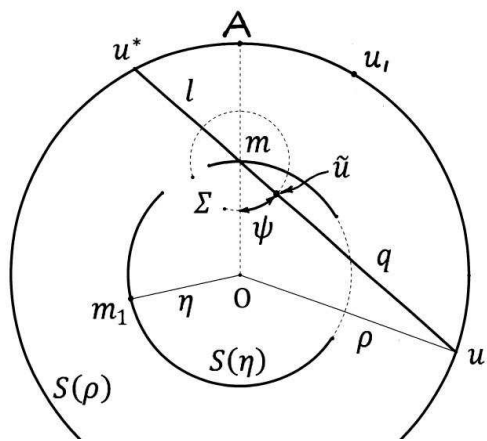,height=6cm}\\
  \caption{Euclidean variables in the hyperbolic disc.}\label{Euclidean-Notations-in-HD}
\end{figure}

Recall also that according to the geometric interpretation of $\omega$ presented in Theorem~\ref{Geometric-Interpretation-theorem}, p.~\pageref{Geometric-Interpretation-theorem}, we have
\begin{equation}\label{Three_D_Dirichlet-20}
    \omega(u,m)=\frac{|u^*-m|}{|m-u|}=\frac{l}{q}\,,
\end{equation}
where $u^*$ is the intersection of $S(\rho)$ and the line defined by $m$ and $u$.

Using the representation \eqref{Radial-func-presentation} of p.~\pageref{Radial-func-presentation} for a radial eigenfunction,  Lemma~\ref{Basic_lemma} of p.~\pageref{Basic_lemma} and the second exchange rule of~\eqref{VarExchange-1}, p.~\pageref{VarExchange-1}, we obtain the following sequence of equalities
\begin{equation}\label{Three_D_Dirichlet-21}
\begin{split}
     \varphi_\lambda(r)
    & =V_\alpha(\eta(r))=
    \frac{1}{|S(\eta)|}\int\limits_{S(\eta)}\omega^\alpha(u_1,m_1) dS_{m_1}
    \\& =\frac{1}{|S(\rho)|}\int\limits_{S(\rho)}\omega^\alpha(u,m) dS_{u}=
         \frac{1}{|S(\rho)|}\int\limits_\Sigma\left(\frac{l}{q}\right)^\alpha\frac{2\rho}{l+q}q^k d\Sigma_{\widetilde{u}}\,,
\end{split}
\end{equation}
where the third equality is a consequence of Lemma~\ref{Basic_lemma}. This completes the proof of the Statement~(A) of Lemma~\ref{Lemma-532}.
\end{proof}

\begin{proof}[Proof of Lemma \ref{Lemma-532}, Statement (B)]

Using the representation of a radial eigenfunction vanishing at some finite point obtained in Theorem~\ref{Radial-Vanishing-EF-Theorem}, see formula~\eqref{Radial-Vanishing-EF-cos}, p.~\pageref{Radial-Vanishing-EF-cos}, and then Lemma~\ref{Basic_lemma}, p.~\pageref{Basic_lemma}, we have
\begin{equation}\label{Three_D_Dirichlet-22}
\begin{split}
    & \varphi_\lambda(r(\eta))=\frac{1}{|S(\eta)|}\int\limits_{S(\eta)}\omega^{k/2}(u_1,m_1)\cos\left( \sqrt{\frac{\lambda}{-\kappa}-\frac{k^2}{4}}\ln\omega \right) dS_{m_1}
    \\& =\frac{1}{|S(\rho)|}\int\limits_{S(\rho)}\omega^{k/2}(u,m)\cos\left( \sqrt{\frac{\lambda}{-\kappa}-\frac{k^2}{4}}\ln\omega \right) dS_{u}    \,,
\end{split}
\end{equation}
where, according to \eqref{Three_D_Dirichlet-20}, $\omega=l/q$. This completes the proof of Statement~(B).
\end{proof}

\begin{proof}[Proof of Lemma \ref{Lemma-532}, Statement (C)]

 As we saw in Statement (B), a radial eigenfunction vanishing at some finite point and corresponding to a real eigenvalue can be written as
\begin{equation}\label{Three_D_Dirichlet-23}
    \varphi_\lambda(r(\eta))
    =\frac{1}{|S(\rho)|}\int\limits_{S(\rho)}\left(\frac{l}{q}\right)^{k/2}
    \cos\left(b\ln\frac{l}{q}\right)dS_{u} \,,
\end{equation}
where $b=\sqrt{(-\lambda/\kappa)-k^2/4}$ and the variables $l,q,u,m$ were introduced above. If we apply the argument used in the proof of Theorem~\ref{Geometric-Interpretation-theorem}, p.~\pageref{Geometric-Interpretation-theorem}, see formulas~\eqref{GeoInte-5} and~\eqref{GeoInte-6}, we observe that $lq=\rho^2-\eta^2$. According to the second formula of~\eqref{VarExchange-1}, p.~\pageref{VarExchange-1},
\begin{equation}\label{Var-Excha-New-Notation-dS_u}
    dS_u=\frac{2\rho}{l+q}q^k d\Sigma_{\widetilde{u}}\,.
\end{equation}
These two expressions for $lq$ and for $dS_u$ allow us to rewrite the integral formula \eqref{Three_D_Dirichlet-23} for $\varphi_\lambda(r(\eta))$ in the following way.
\begin{equation}\label{Three_D_Dirichlet-24}
\begin{split}
    \varphi_\lambda(r(\eta))
    & =\frac{1}{|S(\rho)|}\int\limits_{\Sigma}\left(\frac{l}{q}\right)^{k/2}
    \cos\left(b\ln\frac{l}{q}\right) \frac{2\rho}{l+q}q^k d\Sigma_{\widetilde{u}} \,,
    \\& =\frac{2\rho(\rho^2-\eta^2)^{k/2}}{|S(\rho)|}\int\limits_{\Sigma}
    \frac{\cos(b\ln l/q)}{l+q}d\Sigma_{\widetilde{u}}\,.
\end{split}
\end{equation}
For the next step we need the following observation. For an angle $\psi\in[0,\pi]$ define $\Sigma(\psi)$ as follows.
\begin{equation}\label{Three_D_Dirichlet-25}
    \Sigma(\psi)=\{\widetilde{u}\in\Sigma \mid \angle Om\widetilde{u}=\psi\}\,.
\end{equation}
It is clear that $\Sigma(\psi)$ is a $(k-1)-$dimensional sphere of radius $\sin\psi$ and for any fixed $\psi\in[0,\pi]$ the last integrand in \eqref{Three_D_Dirichlet-24} does not depend on $\widetilde{u}\in\Sigma(\psi)$. Therefore,
\begin{equation}\label{Three_D_Dirichlet-26}
    \int\limits_{\Sigma}\frac{\cos(b\ln l/q)}{l+q}d\Sigma_{\widetilde{u}}
    =\int\limits_0^{\pi}\frac{\cos(b\ln l/q)}{l+q}|\Sigma(\psi)|d\psi\,,
\end{equation}
where $|\Sigma(\psi)|=\sigma_{k-1}(\sin\psi)^{k-1}$ is the volume of the sphere $\Sigma(\psi)$ and $\sigma_{k-1}$ is the volume of $(k-1)-$dimensional unit sphere. Therefore, using \eqref{Three_D_Dirichlet-26}, we may continue the sequence of equalities in \eqref{Three_D_Dirichlet-24}, which yields
\begin{equation}\label{Three_D_Dirichlet-27}
    \varphi_\lambda(r(\eta))
    =\frac{2\rho(\rho^2-\eta^2)^{k/2}\sigma_{k-1}}{|S(\rho)|}\int\limits_0^\pi
    \frac{\sin^{k-1}\psi}{l+q}\cos\left(b\ln\frac{l}{q}\right)d\psi\,.
\end{equation}
According to \eqref{Three_D_Dirichlet-29}, p.~\pageref{Three_D_Dirichlet-29}, the last integrand is symmetric with respect to $\psi=\pi/2$. Therefore, the integral formula for a radial function presented in \eqref{Three_D_Dirichlet-27} can be written as
\begin{equation}\label{Three_D_Dirichlet-30}
    \varphi_\lambda(r(\eta))
    =\frac{4\rho(\rho^2-\eta^2)^{k/2}\sigma_{k-1}}{|S(\rho)|}\int\limits_0^{\pi/2}
    \frac{\sin^{k-1}\psi}{l+q}\cos\left(b\ln\frac{l}{q}\right)d\psi\,,
\end{equation}
which completes the proof of Statement~(C) and the proof of Lemma~\ref{Lemma-532}.
\end{proof}



\subsection {Explicit solutions.}

The following theorem gives us explicit solutions for a Dirichlet Eigenvalue Problem with a non-zero condition at the origin in $B^3_\rho$.

\begin{theorem}\label{Explicit-Solution-Theorem}

Let $B^3_\rho$, as before, be the ball model of 3-dimensional hyperbolic space with a constant sectional curvature $\kappa=-1/\rho^2$ and
\begin{equation}\label{Three_D_Dirichlet-31}
\left\{
  \begin{array}{ll}
     & \hbox{$\triangle\varphi_\lambda(\upsilon,r)+\lambda \varphi_\lambda(\upsilon,r)=0 \quad \forall r\in [0,\delta], \,\,\lambda - \text{real}$;} \\
     & \varphi_\lambda(\upsilon,0)=1\quad\text{and}\quad \hbox{$\varphi_\lambda(\upsilon, \delta)=0 \quad \text{for some}\,\, \delta\in(0,\infty)$,}
  \end{array}
\right.
\end{equation}
where $\upsilon$ is a point of the unit $k-$dimensional sphere centered at the origin and $r$ is the geodesic distance between a point in $B^3_\rho$ and the origin.

Then the system above has a solution if and only if
\begin{equation}\label{Three_D_Dirichlet-32}
    \lambda=\lambda_j=-\kappa+\left(\frac{\pi j}{\delta}\right)^2 \quad\text{for some}\,\, j=1,2,3,\cdots\,.
\end{equation}
In addition, for each such $\lambda_j$ there is a unique radial eigenfunction given by
\begin{equation}\label{Expli-Diri-Solu-3-D}
    \varphi_{\lambda_j}(r)=\frac{\delta(\rho^2-\eta^2)}{2\pi\rho^2 \eta j}\cdot\sin\left(\frac{\pi jr}{\delta}\right)
    =\frac{\delta}{\pi j\rho \sinh(r/\rho)}\cdot\sin\left(\frac{\pi jr}{\delta}\right)\,,
\end{equation}
where $0\leq r\leq\delta<\infty$ and $\eta=\rho\tanh(r/2\rho)$.
\end{theorem}

\begin{proof}[Proof of Theorem \ref{Explicit-Solution-Theorem}]
First of all let us observe that
\begin{equation}\label{eta-rho-equivalence}
    \eta=\rho\tanh\left(\frac{r}{2\rho}\right)\quad\text{is equivalent to}\quad
    \frac{\rho^2-\eta^2}{2\eta\rho}=\frac{1}{\sinh(r/\rho)}
\end{equation}
and then, two expression for $\varphi_{\lambda_j}$ given in~\eqref{Expli-Diri-Solu-3-D} are equivalent. Take any solution of \eqref{Three_D_Dirichlet-31} and radialize it. Since we are looking for a vanishing radial eigenfunction, we can combine~\eqref{Radial-Vanishing-EF}, p.~\pageref{Radial-Vanishing-EF} and~\eqref{Basic-Th-Integrable-Case-1}, p.~\pageref{Basic-Th-Integrable-Case-1}, which for the properly chosen $b$ yields
\begin{equation}\label{Varphi-of-b-Explicit-3-D}
    \varphi_\lambda(r(\eta))=\frac{\rho^2-\eta^2}{2\eta\rho b}\cdot\sin\left(\frac{br}{\rho}\right)
    \quad\text{and}\quad\varphi_\lambda(\delta)=0\,,
\end{equation}
where $b=\sqrt{-\lambda/\kappa-1}$ and $\kappa=-1/\rho^2$. Therefore, to satisfy the second equation in~\eqref{Varphi-of-b-Explicit-3-D} we have to set
$$\frac{\delta}{\rho}\sqrt{\frac{\lambda_j}{-\kappa}-1}=\pi j\,,\quad j=1,2,3,\cdots\,,$$
which leads directly to \eqref{Three_D_Dirichlet-32}. Note that in this case $b_j=\pi\rho j/\delta$ and~\eqref{Expli-Diri-Solu-3-D} follows. This completes the proof of Theorem~\ref{Explicit-Solution-Theorem}.
\end{proof}

\begin{remark}[An alternative proof of Theorem~\ref{Explicit-Solution-Theorem}]
    The other way to obtain the proof of Theorem~\ref{Explicit-Solution-Theorem} is to observe, using~\eqref{Expression-q-after-substitution}, p.~\pageref{Expression-q-after-substitution}, that for $k=2$, $Q(r)=\lambda+\kappa=\text{const}$ and the solution $Y(r)$ in~\eqref{radial-laplac-Recall}, p.~\pageref{radial-laplac-Recall}, satisfying the initial condition $Y(0)=1$ can be written explicitly as
    \begin{equation}\label{Varphi-with-C0}
        \varphi_\lambda(r)=Y(r)=C_0\cdot(\sinh(r/\rho))^{-1}\sin(r\sqrt{\lambda+\kappa})\,,
    \end{equation}
    where $C_0$ is to be defined. Since $\sqrt{\lambda+\kappa}=b/\rho$, we must set $C_0=1/b$ and observe that, according to~\eqref{eta-rho-equivalence}, the expression of $\varphi_\lambda(r)$ in~\eqref{Varphi-with-C0} is equivalent to the expression of $\varphi$ presented in~\eqref{Varphi-of-b-Explicit-3-D}. Repeating the argument related to the boundary condition we complete the alternative proof of Theorem~\ref{Explicit-Solution-Theorem}.
\end{remark}

\begin{corollary}\label{Minimal-EV-for-3-disc}
 Recall that the standard Dirichlet Eigenvalue Problem does not have the requirement about a non-zero initial condition at the origin. If we consider the Dirichlet Eigenvalue Problem \eqref{Three_D_Dirichlet-31} without the condition $\varphi_\lambda(0)=1$, then formula \eqref{Three_D_Dirichlet-32} gives for $j=1$ the precise value of the minimal eigenvalue.
\end{corollary}

\begin{proof} According to Theorem 2 of \cite{Chavel}, p.44, the lowest positive Dirichlet eigenvalue in the $n-$disk of radius $\delta$ must have a non-trivial radial eigenfunction. According to \cite{Chavel}, p.272, a non-trivial radial eigenfunction can not vanish at the origin. This is why the non-zero condition at the origin does not affect the estimation of the lower bound and the smallest positive eigenvalue corresponding to a radial eigenfunction is the smallest one for the standard Dirichlet Eigenvalue Problem, i.e., without a non-zero initial condition.
\end{proof}

\section {Lower and upper bounds for the minimal eigenvalue in a Dirichlet Eigenvalue Problem.}

In this section we obtain the lower and the upper bounds for the minimal positive eigenvalue of a Dirichlet Eigenvalue Problem stated in a connected compact domain $\overline{M^n}\subseteq H^n$ with $\partial M^n\neq\emptyset$. According to~\cite{Chavel}, p.~8, the Dirichlet Eigenvalue Problem for $\overline{M^n}$ is stated as follows.

\textbf{\underline{Dirichlet Eigenvalue Problem:}}\label{Dirichlet-Eigen-General-Chavel} Let $M^n$ be relatively compact and connected domain with smooth boundary $\partial M^n\neq\emptyset$ and $\lambda_{\min}$ denotes the minimal eigenvalue. We are looking for all real numbers $\lambda$ for which there exists a nontrivial solution $\varphi\in C^2(M^n)\cap C^0(\overline{M^n})$ satisfying the following system of equations.
\begin{equation}
    \triangle\varphi+\lambda\varphi=0\quad\text{and}\quad\left.\varphi\right|_{\partial M}=0\,.
\end{equation}

First, in Theorem~\ref{Lower-and-Upper-bound-Theorem} below, we prove the set of inequalities for $M^n=D^n(\delta)\subseteq B^n_\rho$, which is a hyperbolic disc of radius $\delta$ centered at the origin and considered in the hyperbolic space model $B^n_\rho$, where $n=k+1$. Then, in Theorem~\ref{Arbitrary-Domain-Estimation-Thm}, p.~\pageref{Arbitrary-Domain-Estimation-Thm}, we obtain a set of inequalities for an arbitrary relatively compact and connected domain $\overline{M^n}\subseteq H^n$ with $\partial M^n\neq\emptyset$.

\begin{theorem}\label{Lower-and-Upper-bound-Theorem}

Let $B^{k+1}_\rho$ be the ball model of $(k+1)-$dimensional hyperbolic space with a constant sectional curvature $\kappa=-1/\rho^2$ and $\varphi_\lambda(\upsilon, r)$ satisfies the following conditions:
\begin{equation}\label{Lower_Bound-1}
\left\{
  \begin{array}{ll}
     & \hbox{$\triangle\varphi_\lambda(\upsilon,r)+\lambda \varphi_\lambda(\upsilon,r)=0 \quad
     \text{for every}\,\,\, r\in [0,\delta], \,\,\lambda - \text{real}$;} \\
     & \hbox{$\varphi_\lambda(\upsilon, \delta)=0 \quad \text{for some}\,\, \delta\in(0,\infty)$,}
  \end{array}
\right.
\end{equation}
where $\upsilon$ is a point of the unit $k-$dimensional sphere centered at the origin and $r$ is the geodesic distance between a point in $B^{k+1}_\rho$ and the origin, i.e., $(\upsilon, \delta)$ are the geodesic polar coordinates. Then the following statements hold.

\textbf{(A)} If $k=1$, then
\begin{equation}\label{Lower_Bound-2}
    -\kappa\frac{k^2}{4}+\left(\frac{\pi}{2\delta}\right)^2<\lambda_{\min}<
    -\kappa\frac{k^2}{4}+\left(\frac{\pi}{\delta}\right)^2 \,.
\end{equation}

\textbf{(B)} If $k=2$, then
\begin{equation}\label{Lower_Bound-3}
    \lambda_{\min} =-\kappa+\left(\frac{\pi}{\delta}\right)^2\,.
\end{equation}

\textbf{(C)} If $k\geq 3$, then
\begin{equation}\label{Lower_Bound-4}
    \lambda_{\min} >-\kappa\frac{k^2}{4}+\left(\frac{\pi}{\delta}\right)^2\,.
\end{equation}

\begin{remark} For $k=1$, a stronger upper bound was obtained by Gage, see~\cite{Gage} or see~\cite{Chavel}, p.~80, but the lower bound given in \eqref{Lower_Bound-2} is new.
\end{remark}
\end{theorem}


\begin{proof}[Proof of Statement (A) of Theorem \ref{Lower-and-Upper-bound-Theorem}]

First we derive the lower bound of \eqref{Lower_Bound-2}. Recall that according to Theorem 2 of \cite{Chavel}, p.44, the lowest positive Dirichlet eigenvalue in the $n-$disk of radius $\delta$ must have a non-trivial radial eigenfunction satisfying \eqref{Lower_Bound-1}. We have already seen in \eqref{Three_D_Dirichlet-18}, p.~\pageref{Three_D_Dirichlet-18}, any radial eigenfunction assuming value one at the origin and vanishing at some finite point $r=r_0$, can be expressed as
\begin{equation}\label{Lower_Bound-5}
    \varphi_\lambda(r(\eta))
    =\frac{4\rho(\rho^2-\eta^2)^{k/2}\sigma_{k-1}}{|S^k(\rho)|}\int\limits_0^{\pi/2}
    \frac{\sin^{k-1}\psi}{l+q}\cos\left(b\ln\frac{l}{q}\right)d\psi\,,
\end{equation}
where
\begin{equation}\label{Lower_Bound-6}
    b=\pm\sqrt{\frac{\lambda}{-\kappa}-\frac{k^2}{4}}\,,\,\,\,l=|m-u^*|\,,\,\,\,q=|m-u|\,.
\end{equation}
For reader's convenience the notations $\eta, \rho, l, q, u, u^*$ are pictured on Figure~\ref{Maximal-position-for-Omega-figure} below. To obtain the minimal Dirichlet Eigenvalue for problem \eqref{Lower_Bound-1} it is enough to find the minimal $\lambda$ for which $\varphi_\lambda(r(\eta))$ represented by \eqref{Lower_Bound-5} vanishes at $r=\delta$.

\begin{proposition}\label{Maximal-position-for-Omega} If
\begin{equation}\label{Lower_Bound-7}
    f(\eta,\psi)=\left| b\ln\frac{l}{q} \right|=\left| b\ln\frac{\rho^2-\eta^2}{\rho^2+\eta^2+2\rho\eta\cos(\theta(\psi))}\right|
\end{equation}
and
\begin{equation}\label{Lower_Bound-8}
    \Gamma=\{ (\eta,\psi) \mid \eta\in[0,\delta_E]\,\,\, \text{and} \,\,\,\psi\in[0,\pi/2] \}\,,
\end{equation}
then
\begin{equation}\label{Lower_Bound-9}
    \max\limits_{\Gamma}f(\eta,\psi)=\left|\frac{b\delta}{\rho}\right|\,,\quad\text{where}\,\,
    \delta=\delta_H=\rho\ln\frac{\rho+\delta_E}{\rho-\delta_E}\,,
\end{equation}
\end{proposition}
where $\eta=|Om|=r_E=\rho\tanh(r/(2\rho))$ and $\delta_E=\rho\tanh(\delta/(2\rho))$.

\begin{figure}[!h]
    \centering
    \epsfig{figure=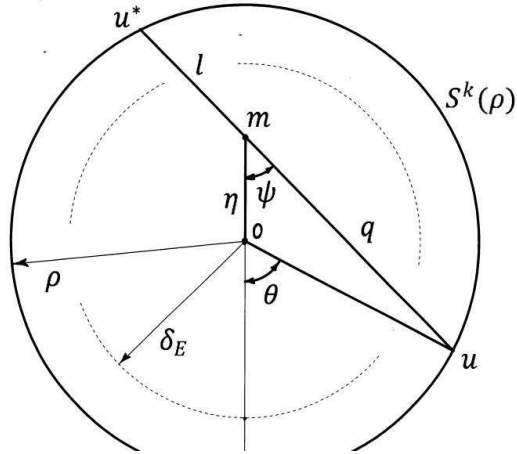,height=6cm}\\
  \caption{Euclidean variables affecting $\omega$.}\label{Maximal-position-for-Omega-figure}
\end{figure}

\begin{proof}[Proof of Proposition \ref{Maximal-position-for-Omega}]
 Notice first that
\begin{equation}\label{Lower_Bound-10}
    \frac{d}{d\theta}\left(\frac{l}{q}\right)\geq0\quad\text{for}\,\,\,(\eta, \theta)\in[0,\rho)\times[0,\pi]\,,
\end{equation}
since the denominator in \eqref{Lower_Bound-7} is a strictly decreasing function for all $\theta\in[0,\pi]$ and for $\eta\in(0,\rho)$. Thus,
\begin{equation}\label{Lower_Bound-11}
    \frac{d}{d\psi}\left(\frac{l}{q}\right)=\frac{d(l/q)}{d\theta}\frac{d\theta}{d\psi}\geq0
    \quad\text{for}\,\,\,(\eta, \psi)\in[0,\rho)\times[0,\pi]\,,
\end{equation}
since $\theta$ is increasing whenever $\psi$ is increasing or by \eqref{VarExchange-1}, p.~\pageref{VarExchange-1},
\begin{equation}\label{Lower_Bound-12}
    \frac{d\theta}{d\psi}=\frac{2q}{l+q}>0\quad\text{for all}
    \,\,\,\psi\in[0,\pi]\,\,\,\text{and}\,\,\,\eta\in[0,\rho)\,.
\end{equation}
Therefore, taking into account that
\begin{equation}\label{Lower_Bound-13}
    \frac{l}{q}<1\quad\text{for}\,\,\,(\eta,\psi)\in(0,\rho)\times[0,\pi/2)\,,
\end{equation}
we have
\begin{equation}\label{Lower_Bound-14}
    \frac{df(\eta,\psi)}{d\psi}=\frac{d}{d\psi}\left|b\ln\frac{l}{q}\right|\leq0\quad
    \text{for all}\,\,\,(\eta,\psi)\in[0,\rho)\times[0,\pi/2]\,.
\end{equation}
Thus, for any $\eta\in[0,\rho)$,
\begin{equation}\label{Lower_Bound-15}
    \max\limits_{\psi\in[0,\pi/2]}f(\eta,\psi)=f(\eta,0)=
    \left|\frac{b}{\rho}\cdot\rho\ln\frac{\rho+\eta}{\rho-\eta}\right|=\left|\frac{br}{\rho}\right|\,.
\end{equation}
Hence,
\begin{equation}\label{Lower_Bound-16}
    \max\limits_{\Gamma}f(\eta,\psi)=\max\limits_{\eta\in[0,\delta_E]}f(\eta,0)=
    \max\limits_{\eta\in[0,\delta_E]} \left|\frac{b\cdot r(\eta)}{\rho}\right|
    =\left|\frac{b\delta}{\rho}\right|\,,
\end{equation}
which completes the proof of Proposition~\ref{Maximal-position-for-Omega}.
\end{proof}

Using Proposition \ref{Maximal-position-for-Omega} above, we conclude that
\begin{equation}\label{Lower_Bound-17}
    \left|\frac{b\delta}{\rho}\right|\leq\frac{\pi}{2}\quad\text{implies}\quad
    \left|b\ln\frac{l}{q}\right|_{\Gamma}\leq\frac{\pi}{2}\,,
\end{equation}
which means that the integrand in \eqref{Lower_Bound-5} remains non-negative for all $(\eta,\psi)\in\Gamma$, and therefore, $\varphi_\lambda(r)$ remains positive for all $r\in[0, \delta]$. Thus, because of \eqref{Lower_Bound-17}, the boundary condition in the Dirichlet Eigenvalue Problem \eqref{Lower_Bound-1} will fail. Conversely, if $\varphi(\delta)=0$, then the first inequality in \eqref{Lower_Bound-17} must fail, i.e,
\begin{equation}\label{Lower_Bound-18}
    \left|\frac{b\delta}{\rho}\right|=
    \frac{\delta}{\rho}\sqrt{\frac{\lambda}{-\kappa}-\frac{k^2}{4}}>\frac{\pi}{2},\quad
    \text{with}\,\,\,\kappa=-\frac{1}{\rho^2}
\end{equation}
or, equivalently,
\begin{equation}\label{Lower_Bound-19}
    \lambda>-\kappa\frac{k^2}{4}+\left(\frac{\pi}{2\delta}\right)^2
\end{equation}
is a necessary condition for $\varphi_\lambda(\delta)=0$. This completes the proof of the lower bound in Statement (A). Note that we did not use the assumption that $k=1$ and therefore, this result holds for all $k\geq1$. On the other hand, for $k>1$ we have much stronger statements given in (B) and (C) of Theorem~\ref{Lower-and-Upper-bound-Theorem}, p.~\pageref{Lower-and-Upper-bound-Theorem}. To obtain the upper bound in \eqref{Lower_Bound-2} of the Statement (A), the assumption $k=1$ is essential. Fix $k=1$, $\eta=\delta_E$ and recall that $\sigma_0=2$. Then, from \eqref{Lower_Bound-5}, the value of radial eigenfunction at $r=\delta$ or, equivalently at $r_E=\delta_E$ can be written as
\begin{equation}\label{Upper_Bound-1}
    \varphi_{\lambda}(\delta)=\frac{4\rho}{\pi}(\rho^2-\delta_E^2)^{1/2}\cdot
    \int\limits_0^{\pi/2}\frac{1}{l+q}\cos\left(b\ln\frac{l}{q}\right)d\psi\,.
\end{equation}
Recall from \eqref{Substitution-for-d-psi}, p.~\pageref{Substitution-for-d-psi} that
\begin{equation}\label{Upper_Bound-2-1}
    d\psi=\frac{l+q}{4\eta b\sin\psi}\,\,d\left(b\ln\frac{l}{q}\right)\,.
\end{equation}
If we use the substitution \eqref{Upper_Bound-2-1} for $d\psi$, then the integration by parts applied to \eqref{Upper_Bound-1} and the fact that $l=q$ for $\psi=\pi/2$ yield the following expression for $\varphi_\lambda(\delta)=\varphi(\delta, b^*(\lambda))$.
\begin{equation}\label{Upper_Bound-3}
    \varphi(\delta, b^*(\lambda))= W\left[ \left. \frac{\sin(b^*\ln q/l)}{\sin\psi} \right|_{\psi\rightarrow0}-
    \int\limits_0^{\pi/2}\frac{\cos\psi}{\sin^2\psi}\sin\left( b^*\ln\frac{q}{l}\right)d\psi \right]\,,
\end{equation}
where
\begin{equation}\label{Upper_Bound-4}
    W=\frac{\rho(\rho^2-\delta_E^2)^{1/2}}{\pi\delta_E b^*}\quad\text{and}\quad b^*=|b|=\sqrt{\frac{\lambda}{-\kappa}-\frac{k^2}{4}}\,.
\end{equation}
We are looking for $b_0$ for which $\varphi(\delta, b_0(\lambda_{\min}))=0$, where $\lambda_{\min}$ denotes the minimal eigenvalue. As we saw in the proof of the lower bound, see \eqref{Lower_Bound-17},
\begin{equation}\label{Upper_Bound-5}
    b_1^*\leq\frac{\pi\rho}{2\delta}\quad\text{implies}\quad\varphi(\delta, b_1^*(\lambda_1))>0\,.
\end{equation}
Our goal now is to find $b_2^*>b_1^*$ such that $\varphi(\delta, b_2^*(\lambda_2))<0$. It will follow that  $b_1^*<b_0^*<b_2^*$, since, according to \eqref{Upper_Bound-1} $\varphi(\delta, b^*)$ is $\mathbf{C}^{\infty}$ with respect to $b^*$. For the next step we need the following proposition.

\begin{proposition}\label{EF-runs-under-zero}
\begin{equation}\label{Upper_Bound-6}
    \text{If}\quad b_2^*=\frac{\pi\rho}{\delta}\,,\quad\text{then}\quad\varphi(\delta,b_2^*)
    =\varphi\left(\delta, \frac{\pi\rho}{\delta}\right)<0\,.
\end{equation}
\end{proposition}

\begin{proof}[Proof of Proposition~\ref{EF-runs-under-zero}]

We are going to use formula \eqref{Upper_Bound-3}. Notice that
\begin{equation}\label{Upper_Bound-7}
    \left. \left(b_2^*\ln\frac{q}{l}\right)\right|_{\eta=\delta_E}=
    \frac{\pi\rho}{\delta}\ln\frac{\rho+\delta_E}{\rho-\delta_E}=\pi\,.
\end{equation}
and then, we apply the L'H\^opital's rule to compute all necessary limits in \eqref{Upper_Bound-3}. Recall from \eqref{Upper_Bound-2-1} that
\begin{equation}\label{Upper_Bound-8}
    \left. \frac{d(b^*\ln l/q)}{d\psi}\right|_{\eta=\delta_E}=\frac{4\delta_E b^*\sin\psi}{l+q}\,.
\end{equation}
So, the L'H\^opital's rule together with \eqref{Upper_Bound-8} yields
\begin{equation}\label{Upper_Bound-9}
    \left.\lim\limits_{\psi\rightarrow0}\frac{\sin(b_2^*\ln l/q)}{\sin\psi}\right|_{\eta=\delta_E}=
    \lim\limits_{\psi\rightarrow0}\frac{4\delta_E b_2^*}{l+q}\frac{\cos(b_2^*\ln l/q)\sin\psi}{\cos\psi}=0\,.
\end{equation}
Therefore,
\begin{equation}\label{Upper_Bound-10}
    \varphi(\delta, b_2^*)=-\frac{\rho(\rho^2-\delta_E^2)^{1/2}}{\delta_E\pi b_2^*}
    \int\limits_0^{\pi/2}\frac{\cos\psi}{\sin^2\psi}\sin\left(b_2^*\ln\frac{q}{l}\right)d\psi\,.
\end{equation}
The integral in \eqref{Upper_Bound-10} is well defined, since again, using \eqref{Upper_Bound-8} and the L'H\^opital's rule, we have:
\begin{equation}\label{Upper_Bound-11}
    \left.\lim\limits_{\psi\rightarrow0}\frac{\sin(b_2^*\ln q/l)}{\sin^2\psi}\right|_{\eta=\delta_E}
    =\frac{\pi\delta_E}{\delta}<\infty\,.
\end{equation}
We also see that the integral in \eqref{Upper_Bound-10} is positive, since $q/l$ is a decreasing function, such that
\begin{equation}\label{Upper_Bound-12}
    \frac{\rho+\delta_E}{\rho-\delta_E}=\left.\frac{q}{l}\right|_{\psi=0}\geq\frac{q(\psi)}{l(\psi)}
    \geq\left.\frac{q}{l}\right|_{\psi=\pi/2}=1\,,
\end{equation}
while $0\leq\psi\leq\pi/2$ and $\eta=\delta_E$. Hence,
\begin{equation}\label{Upper_Bound-13}
    b_2^*\ln\frac{\rho+\delta_E}{\rho-\delta_E}=b_2^*\frac{\delta}{\rho}
    =\pi\geq b_2^*\ln\frac{q}{l}\geq0\,,
\end{equation}
which implies that
\begin{equation}\label{Upper_Bound-14}
    \sin\left(b_2^*\ln\frac{q}{l}\right)\geq0\quad\text{for}\quad\psi\in\left[0,\frac{\pi}{2}\right]\,.
\end{equation}
Therefore, the integrand in \eqref{Upper_Bound-10} is strictly positive for all $\psi\in(0, \pi/2)$, and then, $\varphi(\delta, b_2^*)<0$. This completes the proof of the Proposition~\ref{EF-runs-under-zero}.
\end{proof}

Note now that if we choose $b_1^*=\pi\rho/(2\delta)$, then the claim of the Proposition~\ref{EF-runs-under-zero} together with \eqref{Upper_Bound-5} implies that
\begin{equation}\label{Upper_Bound-15}
    b_1^*(\lambda_1)<b_0^*(\lambda_{\min})<b_2^*(\lambda_2)\,,
\end{equation}
where $b^*$ and $\lambda$ are related as usual, by the following formula
\begin{equation}\label{Upper_Bound-16}
    b^*(\lambda)=\sqrt{\frac{\lambda}{-\kappa}-\frac{k^2}{4}}
\end{equation}
and $\lambda_{\min}$ is the minimal eigenvalue, such that
$\varphi(\delta, b_0^*(\lambda_{\min}))=0$.
Therefore,
\begin{equation}\label{Upper_Bound-17}
   \frac{\pi\rho}{2\delta}=b_1^*<b_0^*(\lambda_{\min})=
   \sqrt{\frac{\lambda_{\min}}{-\kappa}-\frac{k^2}{4}}<b_2^*=\frac{\pi\rho}{\delta}\,,
\end{equation}
which leads exactly to what we need, i.e.,
\begin{equation}\label{Upper_Bound-18}
    -\kappa\frac{k^2}{4}+\left(\frac{\pi}{2\delta}\right)^2<\lambda_{\min}<
    -\kappa\frac{k^2}{4}+\left(\frac{\pi}{\delta}\right)^2\,,
\end{equation}
and then, the proof of Statement (A) of Theorem~\ref{Lower-and-Upper-bound-Theorem} is complete.
\end{proof}

Note that Statement (B) was obtained in previous section, see Corollary~\ref{Minimal-EV-for-3-disc}, p.~\pageref{Minimal-EV-for-3-disc}.

\begin{proof}[Proof of Statement (C) of Theorem \ref{Lower-and-Upper-bound-Theorem}]

We shall see here that using the integration by parts we can improve the inequality \eqref{Lower_Bound-19} for $k\geq3$. Recall again from \eqref{Lower_Bound-5}, p.~\pageref{Lower_Bound-5} that a vanishing radial eigenfunction can be expressed as
\begin{equation}\label{Lower_Bound-20}
    \varphi_\lambda(r(\eta))
    =\frac{4\rho(\rho^2-\eta^2)^{k/2}\sigma_{k-1}}{|S^k(\rho)|}\int\limits_0^{\pi/2}
    \frac{\sin^{k-1}\psi}{l+q}\cos\left(b\ln\frac{l}{q}\right)d\psi
\end{equation}
and by \eqref{Substitution-for-d-psi}, p.~\pageref{Substitution-for-d-psi} that
\begin{equation}\label{Lower_Bound-21}
    d\psi=\frac{l+q}{4\eta b\sin\psi}\cdot d\left( b\ln\frac{l}{q}\right)\,.
\end{equation}
If we use the substitution \eqref{Lower_Bound-21} for $d\psi$, then the integration by parts applied to \eqref{Lower_Bound-20} and the fact that $l=q$ for $\psi=\pi/2$ yield the following expression for $\varphi_{\lambda}(r(\eta))$.
\begin{equation}\label{Lower_Bound-22}
    \varphi_\lambda(r(\eta))=\frac{\rho(\rho^2-\eta^2)^{k/2}\sigma_{k-1}}{\eta|S^k(\rho)|\cdot|b|}
    \int\limits_{\psi=0}^{\psi=\pi/2}\sin\left(|b|\ln\frac{q}{l}\right)d\sin^{k-2}\psi\,.
\end{equation}
Notice that $q\geq l$ for $(\eta,\psi)\in[0,\rho)\times[0,\pi/2]$ and $\sin\psi$ is strictly increasing function for $0<\psi<\pi/2$. Hence, the integrand in \eqref{Lower_Bound-22} remains non-negative as long as
\begin{equation}\label{Lower_Bound-23}
    0\leq|b|\ln\frac{q}{l}=\left|b\ln\frac{q}{l}\right|\leq\pi\,.
\end{equation}
Recall from \eqref{Lower_Bound-9}, p.~\pageref{Lower_Bound-9} that
\begin{equation}\label{Lower_Bound-24}
    \max\limits_{\Gamma}\left|b\ln\frac{q}{l}\right|=\left|\frac{b\delta}{\rho}\right|\,,
\end{equation}
where
\begin{equation}\label{Lower_Bound-25}
    \Gamma=\{ (\eta,\psi) \mid \eta\in[0,\delta_E]\,\,\, \text{and} \,\,\,\psi\in[0,\pi/2] \}\,.
\end{equation}
Therefore,
\begin{equation}\label{Lower_Bound-26}
    \left|\frac{b\delta}{\rho}\right|\leq\pi\quad\text{implies}\quad\left|b\ln\frac{q}{l}\right|_{\Gamma}\leq\pi\,.
\end{equation}
Then, using the integral representation \eqref{Lower_Bound-22}, we conclude that $\varphi_\lambda(r)>0$ for all $r\in[0,\delta]$. This means that the boundary condition in the Dirichlet Eigenvalue Problem \eqref{Lower_Bound-1}, p.~\pageref{Lower_Bound-1}, can not be satisfied. Thus, finally,
\begin{equation}\label{Lower_Bound-27}
    \left|\frac{b\delta}{\rho}\right|=\frac{\delta}{\rho}\sqrt{\frac{\lambda}{-\kappa}-\frac{k^2}{4}}>\pi
    \quad\text{with}\,\,\,\kappa=-\frac{1}{\rho^2}\,,
\end{equation}
or, equivalently,
\begin{equation}\label{Lower_Bound-28}
    \lambda>-\kappa\frac{k^2}{4}+\left(\frac{\pi}{\delta}\right)^2
\end{equation}
is necessary condition for $\varphi_\lambda(\delta)=0$ to hold. This completes the proof of Statement~(C) of Theorem~\ref{Lower-and-Upper-bound-Theorem} and completes the proof of the Theorem as well.
\end{proof}

Now we are ready to describe the set of inequalities for an arbitrary domain $M^n$ stated in Theorem~\ref{Arbitrary-Domain-Estimation-Thm} below.

\begin{theorem}\label{Arbitrary-Domain-Estimation-Thm}
    Let $M^n$ and $\lambda\in\mathbb{R}$ be as defined in the Dirihlet Eigenvalue Problem, p.~\pageref{Dirichlet-Eigen-General-Chavel}.
    Let $D^n_1$ and $D^n_2$ be two disks in $\mathbb{H}^n$ such that
    \begin{equation}\label{Disc-Manifold-Inclusion-Relationship}
        D^n_1\subseteq M^n\subseteq D^n_2
    \end{equation}
    and let $d_1$ and $d_2$ be the diameters of $D^n_1$ and $D^n_2$ respectively. Then
    \begin{description}
      \item[(A)] For $n=2$
        \begin{equation}
            -\frac\kappa4+\left(\frac{\pi}{d_2}\right)^2\leq\lambda_{\min}(M^2)\leq
            -\frac\kappa4+\left(\frac{2\pi}{d_1}\right)^2\,.
        \end{equation}
      \item[(B)] For $n=3$
        \begin{equation}
            -\kappa+\left(\frac{2\pi}{d_2}\right)^2\leq\lambda_{\min}(M^3)\leq
            -\kappa+\left(\frac{2\pi}{d_1}\right)^2\,.
        \end{equation}
      \item[(C)] For $n>3$
        \begin{equation}
            -\kappa\frac{(n-1)^2}{4}+\left(\frac{2\pi}{d_2}\right)^2\leq\lambda_{\min}(M^n)\,.
        \end{equation}
    \end{description}
\end{theorem}

\begin{proof}[Proof of Theorem \ref{Arbitrary-Domain-Estimation-Thm}]\label{Proof-of-Arbitrary-Estimation}
    By the Rayleigh's Theorem in \cite{Chavel}, p.~16, the relationship \eqref{Disc-Manifold-Inclusion-Relationship}, p.~\pageref{Disc-Manifold-Inclusion-Relationship} yields
    \begin{equation}\label{Eigenvalue-Disc-Manifold-Inequalities}
        \lambda_{\min}(D^n_2)\leq\lambda_{\min}(M^n)\leq\lambda_{\min}(D^n_1)\,.
    \end{equation}
    For $n>3$, using \eqref{Lower_Bound-4} from p.~\pageref{Lower_Bound-4} together with \eqref{Eigenvalue-Disc-Manifold-Inequalities}, we have
    \begin{equation}
        -\kappa\frac{(n-1)^2}{4}+\left(\frac{2\pi}{d_2}\right)^2\leq
        \lambda_{\min}(D^n_2)\leq\lambda_{\min}(M^n)\,,
    \end{equation}
    which completes the proof of Statement~(C) of Theorem~\ref{Arbitrary-Domain-Estimation-Thm}, p.~\pageref{Arbitrary-Domain-Estimation-Thm}.

    The similar argument together with formulae~\eqref{Lower_Bound-2}, \eqref{Lower_Bound-3}, p.~\pageref{Lower_Bound-2} and~\eqref{Eigenvalue-Disc-Manifold-Inequalities} yields directly Statements (A) and (B) of Theorem~\ref{Arbitrary-Domain-Estimation-Thm}, p.~\pageref{Arbitrary-Domain-Estimation-Thm}. This completes the proof of Theorem~\ref{Arbitrary-Domain-Estimation-Thm}.
\end{proof}


\section {Dirichlet eigenvalue problem at $\infty$ with a non-zero condition at the origin.}

Recall that according to Theorem \ref{Lower-and-Upper-bound-Theorem}, p.~\pageref{Lower-and-Upper-bound-Theorem}, the lower bound for the minimal eigenvalue in the Dirichlet Eigenvalue Problem in a $(k+1)-$dimensional disk of radius $\delta=\delta_H<\infty$ is
\begin{equation}\label{Infinity_Case-1}
    \lambda_{\min}\geq\frac{-\kappa k^2}{4}+\left(\frac{\pi}{\delta(2+\text{sign}(1-k))}\right)^2\,,
\end{equation}
where $\kappa=-1/\rho^2$ and $k=1,2,3,\ldots$.

Can we use this formula to obtain the lower bound for the minimal eigenvalue in the case when $\delta=\infty$? The answer is no. I.e., the lower bound for the minimal eigenvalue in the limit case of the Dirichlet Eigenvalue Problem \eqref{Lower_Bound-1} is not the same as the limit for the lower bound in the \eqref{Infinity_Case-1} when $\delta\rightarrow\infty$.

In this section we describe the behavior at $\infty$ of all radial eigenfunctions with real eigenvalues that do not vanish at the origin.


\begin{theorem}\label{Infinity-Behavior}

By analogy with \eqref{Lower_Bound-1}, let us call the following system of conditions in $B^{k+1}_\rho$ the Dirichlet eigenvalue problem at $\infty$.
\begin{equation}\label{Th4-1}
\left\{
  \begin{array}{ll}
     & \hbox{$\triangle\varphi_\lambda(\upsilon,r)+\lambda \varphi_\lambda(\upsilon,r)=0 \quad
      \text{for every}\,\,\, r\in [0,\infty), \,\,\lambda - \text{real}$;} \\
     & \varphi_\lambda(\upsilon,0)=1\quad\text{and}\quad \hbox{$\lim\varphi_\lambda(\upsilon, r)=0 \quad \text{as}\,\, r\rightarrow\infty\,.$}
  \end{array}
\right.
\end{equation}
The system above has a solution if and only if $\lambda\in(0,\infty)$.
\end{theorem}

\begin{remark}
Note that the system \eqref{Th4-1} may have a solution only if it has a radial solution, since any solution can be radialized. Therefore, Theorem~\ref{Infinity-Behavior} is a consequence of Theorem~\ref{Infinity-Behavior-radial-Theorem} below, which describes the behavior of all radial eigenfunctions at $\infty$.
\end{remark}

\begin{theorem}\label{Infinity-Behavior-radial-Theorem}
\begin{equation}\label{Infinity-Behavior-radial-Formula}
\lim\limits_{r\rightarrow\infty}\varphi_\lambda(r)=
\left\{
  \begin{array}{ll}
    \infty, & \hbox{if $\quad \lambda<0$;} \\
    1, & \hbox{if $\quad \lambda=0$;} \\
    0, & \hbox{if $\quad \lambda>0$.}
  \end{array}
\right.
\end{equation}
\end{theorem}

\begin{remark}
Recall that as it was mentioned in the proof of Theorem~\ref{Radial-Vanishing-EF-Theorem}, p.~\pageref{Radial-Vanishing-EF-Theorem}, $\varphi_\lambda(r)$ satisfies the following differential equation
$$\left[(\sinh(r/\rho))^{k/2}\varphi_\lambda(r)\right]''=
\left(\frac{k^2}{4\rho^2}-\lambda+\frac{k(k-2)}{4\rho^2\sinh^2(r/\rho)}\right)
\left[(\sinh(r/\rho))^{k/2}\varphi_\lambda(r)\right],$$
which can be investigated by the standard theory of ODE. Theorem~2.1 from \cite{Olver}, p.~193 yields the behavior of $\varphi_\lambda(r)$ for $k^2/4\rho^2>\lambda$ and $r\rightarrow\infty$. Theorem~2.2 from~\cite{Olver}, p.~196 yields the behavior of $\varphi_\lambda(r)$ for $k^2/4\rho^2<\lambda$. To describe the behavior of $\varphi_\lambda(r)$ for $\lambda=k^2/4\rho^2$ it is enough to apply Corollary~9.1 from~\cite{Hartman}, p.~380. Here we shall see the elementary proof of Theorem~\ref{Infinity-Behavior-radial-Theorem} relying only on the geometric interpretation of $\omega$ introduced in Theorem~\ref{Geometric-Interpretation-theorem}, p.~\pageref{Geometric-Interpretation-theorem} and the elementary geometry.
\end{remark}

\begin{proof}[Proof of Theorem \ref{Infinity-Behavior-radial-Theorem}]
The proof of Theorem~\ref{Infinity-Behavior-radial-Theorem} can be split into four Propositions~\ref{Limit-for-alpha-(0,k)} - \ref{Limit-for-alpha=k/2+ib} below. Combining all of them we arrive at~\eqref{Infinity-Behavior-radial-Formula}, which completes the proof of Theorem~\ref{Infinity-Behavior-radial-Theorem}.
\end{proof}

\begin{proposition}\label{Limit-for-alpha-(0,k)}
\begin{equation}\label{Th4-2}
   \lim\limits_{r\rightarrow\infty}\varphi_\lambda(r(\eta))=
   \lim\limits_{\eta\rightarrow\rho}\frac{1}{|S(\rho)|}
   \int\limits_{S(\rho)}\left(\frac{\rho^2-\eta^2}{|u-m|^2}\right)^\alpha dS_u=0
\end{equation}
for all $\alpha\in(0,k)$ or, equivalently, for all $\lambda\in(0,-\kappa k^2/4]$.
\end{proposition}

\begin{proposition}\label{Limit-for-alpha-0-and-k}
\begin{equation}\label{Th4-3}
   \lim\limits_{r\rightarrow\infty}\varphi_\lambda(r(\eta))=
   \lim\limits_{\eta\rightarrow\rho}\frac{1}{|S(\rho)|}
   \int\limits_{S(\rho)}\left(\frac{\rho^2-\eta^2}{|u-m|^2}\right)^\alpha dS_u=1
\end{equation}
for $\alpha=0,k$ or, equivalently, for $\lambda=0$.
\end{proposition}

\begin{proposition}\label{Limit-for-alpha-outside-[0,k]}
\begin{equation}\label{Th4-4}
   \lim\limits_{r\rightarrow\infty}\varphi_\lambda(r(\eta))=
   \lim\limits_{\eta\rightarrow\rho}\frac{1}{|S(\rho)|}
   \int\limits_{S(\rho)}\left(\frac{\rho^2-\eta^2}{|u-m|^2}\right)^\alpha dS_u=\infty
\end{equation}
for all $\alpha\in\mathbb{R}\setminus[0,k]$ or, equivalently, for $\lambda<0$.
\end{proposition}

\begin{proposition}\label{Limit-for-alpha=k/2+ib} If $\alpha=k/2+ib$ with $b\neq0$ or, equivalently, if $\lambda~>~-\kappa k^2/4$, then $\varphi_\lambda(r)$ is a non-trivial radial eigenfunction that vanishes at some finite point and
\begin{equation}\label{Th4-4.1}
   \lim\limits_{r\rightarrow\infty}\varphi_\lambda(r(\eta))=
    \lim\limits_{\eta\rightarrow\rho}\frac{1}{|S(\rho)|}\int\limits_{S(\rho)}\omega^{k/2}\cos(b\ln\omega)dS_u=0
\end{equation}
\end{proposition}

\begin{proof}[Proof of Proposition \eqref{Limit-for-alpha=k/2+ib}] Note that the integral presentation of $\varphi_\lambda(r)$ is justified in the proof of Lemma~\ref{Lemma-532} from p.~\pageref{Lemma-532}, look at \eqref{Three_D_Dirichlet-16}, p.~\pageref{Three_D_Dirichlet-16}. Therefore, the statement is a simple consequence of Proposition A.
\end{proof}

\begin{proposition}\label{Harold-Donelly-proposition}
For any solution of \eqref{Th4-1}, we have
\begin{equation}\label{Th4-4.2}
    \varphi_\lambda(\upsilon, r)\notin L^2(B_\rho^{k+1}) \quad\forall\,\,\lambda\in R\,.
\end{equation}
\end{proposition}

\begin{proof}[Proof of Proposition \ref{Harold-Donelly-proposition}] 
The statement of the Proposition is a simple consequence of theorem taken from Harold Donnelly's paper, see~\cite{Donelly},
p.366. This theorem restricted to functions in $L^2(H^{k+1})$ says that the Laplacian $\triangle$ defined on $L^2(H^{k+1})$ has no point spectrum, which means that the Laplacian does not have any eigenvalue corresponding to $L^2$ eigenfunction. This completes the proof of Proposition~\ref{Harold-Donelly-proposition}.
\end{proof}

\begin{proof}[Proof of Proposition \ref{Limit-for-alpha-(0,k)}]\label{Proof-of-Limit-for-alpha-(0,k)}

Recall from \eqref{Three_D_Dirichlet-15}, p.~\pageref{Three_D_Dirichlet-15} that a radial eigenfunction can be written as
\begin{equation}\label{Th4-10}
   \varphi_\lambda(r(\eta))=\frac{1}{|S(\rho)|}
   \int\limits_{\Sigma}\left(\frac{l}{q}\right)^\alpha\frac{2\rho}{l+q}q^k d\Sigma_{\widetilde{u}}\,.
\end{equation}
All notations used in the proof are pictured on Figure~\ref{Euclidean-Notations-in-HD}, p.~\pageref{Euclidean-Notations-in-HD}.
It is convenient to split the following argument into two cases.



\textbf{Case 1 $(k\geq 2).$} Formula \eqref{Th4-10} gives
\begin{equation}\label{Th4-11}
\begin{split}
   \varphi_\lambda(r(\eta)) & =\frac{2\rho}{|S(\rho)|}
   \int\limits_{\Sigma}\frac{l}{l+q}\,l^{\alpha-1}q^{k-\alpha} d\Sigma_{\widetilde{u}}
   \\& =\frac{2\rho}{|S(\rho)|}
   \int\limits_{\Sigma}\frac{q}{l+q}\,l^{\alpha}q^{k-1-\alpha} d\Sigma_{\widetilde{u}}\,.
\end{split}
\end{equation}
From the first integral in \eqref{Th4-11}, it is clear that

\begin{equation}\label{Th4-12}
\begin{split}
   & \frac{l}{l+q}< 1\quad\text{for every}\quad\frac{l}{q}\quad\text{and}
   \\& \lim\limits_{m\rightarrow A}l^{\alpha-1}q^{k-\alpha}=0\quad
   \text{for every}\,\,\,\psi\,\,\text{and every}\,\,\alpha\in(1,k)\,,
\end{split}
\end{equation}
Recall that $|m|=\eta$, $|A|=\rho$. Therefore, $m\rightarrow A$ if and only if $\eta\rightarrow\rho$. Hence, $\lim\limits_{r\rightarrow\infty}\varphi_\lambda(r(\eta))=0$ for every $\alpha\in(1,k)$. In a similar way, using the second integral in \eqref{Th4-11}, we can see that $\lim\limits_{r\rightarrow\infty}\varphi_\lambda(r(\eta))=0$ for every $\alpha\in(0,k-1)$. Therefore, the proof of Proposition~\ref{Limit-for-alpha-(0,k)} is complete for $k\geq 2$.

\textbf{Case 2 $(k=1).$} This case requires more detailed analysis. We are going to use notations from the Figure~\ref{Limit-Behavior-picture} below. Choose an arbitrary small $\tau>0$. Our goal is to show that for every point $m$ sufficiently close to $A$ $\varphi_\lambda(r(\eta))<3\tau\cdot 4\rho/|S(\rho)|$. Let us set $\psi_0=\pi/2-\tau$ and define $\Omega_i$ as follows.

\begin{figure}[!h]
    \centering
    \epsfig{figure=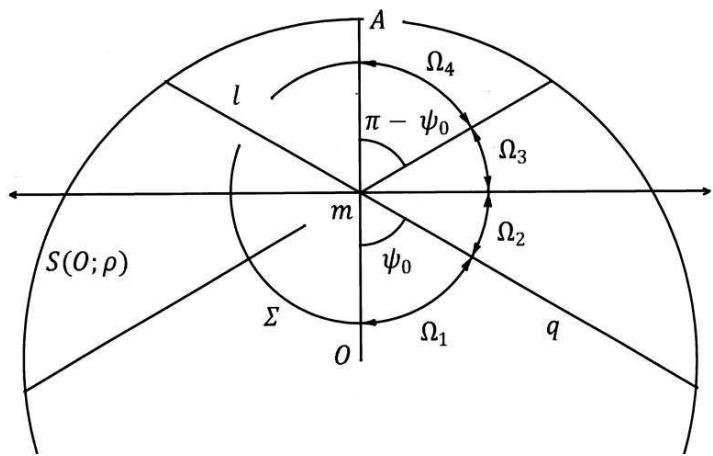,height=6cm}\\
  \caption{The limit behavior in $B^2_\rho$.}\label{Limit-Behavior-picture}
\end{figure}

\begin{equation}\label{Th4-13}
\begin{split}
   & \Omega_1=\{\psi\mid0<\psi<\psi_0=\pi/2-\tau\};\,\Omega_2=\{\psi\mid\pi/2-\tau<\psi<\pi/2\};
   \\& \Omega_3=\{\psi\mid\pi/2<\psi<\pi/2+\tau\};\,\,\Omega_4=\{\psi\mid\pi/2+\tau<\psi<\pi\}\,.
\end{split}
\end{equation}
It is clear that $|\Omega_2|=|\Omega_3|=\tau=\pi/2-\psi_0$. Applying \eqref{Th4-11} for the case $k=1$, we have
\begin{equation}\label{Th4-14}
   \varphi_\lambda(r(\eta))=\frac{4\rho}{|S(O;\rho)|}
   \left( \int\limits_{\Omega_1}+\int\limits_{\Omega_2}+\int\limits_{\Omega_3}+\int\limits_{\Omega_4} \right)
   \frac{l^\alpha q^{1-\alpha}}{l+q}\,d\sigma_\psi\,,
\end{equation}
and we need to study each of the integrals separately. Let $\epsilon_1, \epsilon_2$ be auxiliary variables introduced below by formulae \eqref{Th4-15} and \eqref{Th4-16} respectively. Recall that according to Lemma~\ref{Minimum-for-(l+q)}, p.~\pageref{Minimum-for-(l+q)} we have
$\l+q\geq 2\rho\cos\psi_0$ and then
\begin{equation}\label{Th4-15}
   \int\limits_{\Omega_1}\frac{l^\alpha q^{1-\alpha}}{l+q}\,d\sigma_\psi\leq
   \frac{l^\alpha(2\rho)^{1-\alpha}}{2\rho\cos\psi_0}|\Omega_1|=\epsilon_1\longrightarrow0
\end{equation}
for every $\tau>0$ as long as $\alpha>0$ and $m\rightarrow A$. Similarly,
\begin{equation}\label{Th4-16}
   \int\limits_{\Omega_4}\frac{l^\alpha q^{1-\alpha}}{l+q}\,d\sigma_\psi
   \leq\frac{(2\rho)^\alpha q^{1-\alpha}}{2\rho\cos\psi_0}|\Omega_4|=\epsilon_2
   \longrightarrow0
\end{equation}
for every $\tau>0$ as long as $\alpha<1$ and $m\rightarrow A$. Further,
\begin{equation}\label{Th4-17}
   \int\limits_{\Omega_2}\frac{l^\alpha q^{1-\alpha}}{l+q}\,d\sigma_\psi
   =\int\limits_{\Omega_2}\frac{q}{q+l}\left(\frac{l}{q}\right)^\alpha d\sigma_\psi
   \leq|\Omega_2|=\tau\quad\forall\,\alpha>0\,,
\end{equation}
since $q/(l+q)\leq1$ and $l/q\leq1$ while $\psi\in\Omega_2$. Similarly,
\begin{equation}\label{Th4-18}
   \int\limits_{\Omega_2}\frac{l^\alpha q^{1-\alpha}}{l+q}\,d\sigma_\psi
   =\int\limits_{\Omega_2}\frac{l}{q+l}\left(\frac{q}{l}\right)^{1-\alpha} d\sigma_\psi
   \leq|\Omega_3|=\tau
\end{equation}
for every $\alpha\leq1$. Note now that according to \eqref{Th4-15} and \eqref{Th4-16}, for every $\tau>0$ and for every $m$ sufficiently close to $A$, $\epsilon_1+\epsilon_2\leq\tau$. Therefore, \eqref{Th4-14} together with the inequalities above, implies that for every $\tau>0$ and for every $m$ sufficiently close to $A$ we have
\begin{equation}\label{Th4-19}
   \frac{|S(\rho)|}{4\rho}\varphi_\lambda(r(\eta))\leq\epsilon_1+\epsilon_2+2\tau\leq3\tau\,.
\end{equation}
This means that $\varphi_\lambda(r(\eta))\rightarrow0$ as long as $m\rightarrow A$ and $\alpha\in(0,1)$. Therefore, the description of Case 2 as well as the proof of Proposition~\ref{Limit-for-alpha-(0,k)} is complete.
\end{proof}

\begin{proof}[Proof of Proposition \ref{Limit-for-alpha-outside-[0,k]}]\label{Proof-of-Limit-for-alpha-outside-[0,k]}

In this proof we use notations introduced in the proof of Proposition~\ref{Limit-for-alpha-(0,k)} and so, a reader may refer to the Figure~\ref{Limit-Behavior-picture} above. Let us fix $\psi_0=\pi/3$ and let
\begin{equation}\label{Th4-21}
   \Omega_0=\{\widetilde{u}\in\Sigma\mid \pi-\psi_0=2\pi/3\leq\psi(\widetilde{u})\leq\pi \}\,.
\end{equation}
Combining the estimations \eqref{l-q-estimation-formulae}, p.~\pageref{l-q-estimation-formulae} together with the representation of $\varphi_\lambda(r)$ given in \eqref{Th4-10}, p.~\pageref{Th4-10}, and taking into account that $\alpha>k$, we obtain
\begin{equation}\label{Th4-22}
\begin{split}
    \frac{|S(\rho)|}{2\rho}\varphi_\lambda(r(\eta))
   & =\int\limits_{\Sigma}\frac{l^\alpha q^{k-\alpha}}{l+q}\,d\Sigma_{\widetilde{u}}
    \geq\int\limits_{\Omega_0}\frac{\rho^\alpha}{2\rho}
   (2(\rho-\eta))^{k-\alpha}d\Omega_0
   \\& =
   \frac{|\Omega_0|\rho^{\alpha-1}}{2^{\alpha-k+1}}\left(\frac{1}{\rho-\eta}\right)^{\alpha-k}\rightarrow\infty\,,
\end{split}
\end{equation}
as $\eta\rightarrow\rho$ or, equivalently, as $m\rightarrow A$. Thus, for every $\alpha>k$ the proof of Proposition~\ref{Limit-for-alpha-outside-[0,k]} is complete. Note now that the representation of $\varphi_\lambda(r)$ given in \eqref{Th4-10}, p.~\pageref{Th4-10}, together with Statement (A) of Theorem~\ref{Basic-theorem}, p.~\pageref{Basic-theorem}, yields
\begin{equation}\label{Th4-23}
    \varphi_\lambda(r(\eta))=\frac{1}{|S(\rho)|}\int\limits_{S(\rho)}\omega^\alpha(u,m) dS_u
    =\frac{1}{|S(\rho)|}\int\limits_{S(\rho)}\omega^{\alpha-k}(u,m) dS_u
\end{equation}
for all $\eta\neq\rho$. Therefore,
\begin{equation}\label{Th4-24}
\begin{split}
    & \lim\limits_{\eta\rightarrow\rho}\varphi_\lambda(r(\eta))=\infty\quad\text{for}\quad\alpha>k
    \\& \quad\Leftrightarrow\quad\lim\limits_{\eta\rightarrow\rho}\varphi_\lambda(r(\eta))=\infty\quad\text{for}\quad k-\alpha>k\,,
\end{split}
\end{equation}
which completes the proof of Proposition~\ref{Limit-for-alpha-outside-[0,k]}.
\end{proof}

\begin{proof}[Proof of Proposition \ref{Limit-for-alpha-0-and-k}]
From \eqref{Th4-23},
\begin{equation}\label{Th4-25}
    |S(\rho)|\varphi_\lambda(r(\eta))=\int\limits_{S(\rho)}\omega^k(u,m) dS_u
    =\int\limits_{S(\rho)}\omega^0(u,m) dS_u=|S(\rho)|
\end{equation}
for all $\eta\neq\rho$ and then, $\lim\limits_{\eta\rightarrow\rho}\varphi_\lambda(r(\eta))=1$ for $\alpha=0, k$. This completes the proof of Proposition~\ref{Limit-for-alpha-0-and-k}.
\end{proof}

\section{Asymptotic behavior of a radial eigenfunction $\varphi_\lambda(r)$ as $\lambda\rightarrow\pm\infty$ and $r$ is fixed.}

This section is split into two subsections. The first one is devoted to the case when $\lambda\rightarrow\infty$ and the second one describes the case when $\lambda\rightarrow-\infty$.

\subsection{Asymptotic behavior of $\varphi_\lambda(r)$ as $\lambda\rightarrow-\infty$.}

In this subsection we are going to describe the asymptotic behavior of $F(\alpha)=\int\limits_{S^k(\rho)} \omega^\alpha dS_u$, when $\alpha=k/2+s\in\mathbb{R}$ and $s\rightarrow\pm\infty$. Theorem~\ref{Laplace-method-Theorem} taken from~\cite{Olver}, (p.~81) and stated below will be used as the basic tool to obtain the leading term of the asymptotic decomposition of $F(k/2+s)$ and as a consequence, the leading term of $\varphi_\lambda(r)$ as $\lambda\rightarrow-\infty$.

\begin{theorem}[Laplace's Method from \cite{Olver}, p.~81, (Theorem 7.1)]\label{Laplace-method-Theorem} Let
\begin{equation}\label{Patch-16}
    I(s)=\int\limits_c^d e^{-sp(t)}q(t)dt
\end{equation}
and we want to study the behavior of $I(s)$ as $s\rightarrow\infty$. Assume $c$ is finite and $p(t)$ has a minimum over the interval $[c,d]$ at $t=c$. Assume also that
\begin{enumerate}
  \item $p(t)>p(c)$ when $t\in(c,d)$, and for every $\tau\in(a,b)$ the infimum of $p(t)-p(c)$ in $[\tau,d)$ is positive (in other words, the minimum of $p(t)$ is approached only at $c$);
  \item $p'(t)$ and $q(t)$ are continuous in a neighborhood of $c$, except possibly at $c$;
  \item for $t \rightarrow c^{+}$, we have $p(t)-p(c)\sim P(t-c)^\mu$, $q(t)\sim Q(t-c)^{\nu-1}$ and the first of this relation is differentiable. Here $P,\nu,\mu$ are positive real constants, and $Q$ is a real or complex constant;
  \item $I$ converges absolutely throughout its range for all sufficiently large $s$.
\end{enumerate}
Then
\begin{equation}\label{Patch-17}
I(s)\sim \frac{Q}{\mu}\Gamma(\nu/\mu)\frac{e^{-sp(c)}}{(Ps)^{\nu/\mu}}=L(s) \qquad \text{as}\,\,\, s\rightarrow\infty \,,
\end{equation}
which means that $I(s)-L(s)=o(L(s))$ as $s\rightarrow\infty$.
\end{theorem}

From now on we are going to prove the following Theorem.

\begin{theorem}\label{Asymptotic-behavior-Theorem}
If $s\rightarrow\pm\infty$, then
\begin{equation}\label{Asymptotic-behavior-Formula-1}
    F(u+s)\sim \left(\frac{\pi\rho(\rho^2-\eta^2)}{\eta|s|}\right)^{k/2}
    \left(\frac{\rho+\eta}{\rho-\eta}\right)^{|s|}
\end{equation}
or equivalently, if $\lambda\rightarrow-\infty$, then
\begin{equation}\label{Asymptotic-behavior-Formula-2}
    \varphi_\lambda(r)\sim
    \left(\frac 2\rho\right)^{k/2}\frac{\Gamma\left(\frac{k+1}{2}\right)}{\sqrt{\pi}\sinh(r/\rho)}
    \cdot
    \frac{e^{r\sqrt{-\lambda}}}{(-\lambda)^{k/4}}\,.
\end{equation}
\end{theorem}

\begin{proof}[Proof of Theorem \ref{Asymptotic-behavior-Theorem}.]

Fix $r$ and recall that
\begin{equation}
    \varphi_\lambda(r(\eta))=\frac{1}{|S^k(\rho)|}\int\limits_{S^k(\rho)}\omega^{\alpha}(m,u)dS_u\,,
\end{equation}
where $|m|=\eta=\rho\tanh(r/\rho)$. Let us set $F(\alpha)=|S^k(\rho)|\cdot\varphi_\lambda(r(\eta))$ and
\begin{equation}
    \omega(t)=\frac{\rho^2-\eta^2}{\rho^2+\eta^2+2\rho\eta\cos(t)}\,, \quad t\in[0,\pi]\,.
\end{equation}
Recall from \eqref{Three_D_Dirichlet-5} - \eqref{Iff-conditions-for-Lambda-real}, p.~\pageref{Iff-conditions-for-Lambda-real} that $\lambda$ is real if and only if
\begin{equation}\label{Lambda-Real-Conditions}
    \alpha=\frac k2\pm s\quad\text{or}\quad\alpha=\frac k2\pm ib\,,\quad s,b\in\mathbb{R}\,.
\end{equation}
Observe now that $\lambda\rightarrow-\infty$ means that $\alpha=k/2\pm s$ and $s\rightarrow\infty$. As we have seen, $F(k/2+s)=F(k/2-s)$. Hence,
\begin{equation}
    F(k/2+s)=\int\limits_{S^k(\rho)}\omega^{k/2-s}dS_u
    =\rho^k\sigma_{k-1}\int\limits_0^{\pi}e^{-s\ln\omega}\omega^{k/2}\sin^{k-1}(t)\,dt\,.
\end{equation}

Note that the last integral has precisely the same form as the integral in~\eqref{Patch-16} from Laplace's Method Theorem and then, the behavior of \\ $F(k/2\pm{s})$ can be investigated by using Laplace's Method Theorem~\ref{Laplace-method-Theorem}, p.~\pageref{Laplace-method-Theorem}. To get the asymptotic term let us put
$$p(t)=\ln\omega(t),\quad q(t)=\omega^{k/2}\sin^{k-1}(t)\quad\text{and}\quad t\in[0,\pi]\,.$$
Note that $p(t)$ is strictly increasing on $[0,\pi]$ and the Taylor series decomposition yields
\begin{equation}\label{Taylor-for-p}
p(t)-p(0)\sim (1/2)p''(0) (t-0)^2=\frac{\rho\eta}{(\rho+\eta)^2}(t-0)^2 \,,
\end{equation}
 which implies $\mu=2$ and $P=\rho\eta/(\rho+\eta)^2$. The direct computation shows that the relation \eqref{Taylor-for-p} is differentiable. Indeed,
\begin{equation}
    p^{'}(t)=\frac{2\rho\eta\sin(t)}{\rho^2+\eta^2+2\rho\eta\cos(t)}\sim\frac{2\rho\eta t}{(\rho+\eta)^2}
    \quad\text{as}\quad t\rightarrow 0\,,
\end{equation}
since
\begin{equation}
    \frac{2\rho\eta\sin(t)}{\rho^2+\eta^2+2\rho\eta\cos(t)}-\frac{2\rho\eta t}{(\rho+\eta)^2}
    =o\left(\frac{2\rho\eta t}{(\rho+\eta)^2}\right)
    \quad\text{as}\quad t\rightarrow 0\,.
\end{equation}
Note also that if $t\rightarrow0$, then $q(t)\sim [(\rho-\eta)/(\rho+\eta)]^{k/2}t^{k-1}$, which yields
\begin{equation}
    Q=\left(\frac{\rho-\eta}{\rho+\eta}\right)^{k/2}t^{k-1}\quad\text{and}\quad\nu=k\,.
\end{equation}
Conditions~2 and~4 at the Laplace's Method Theorem can also be easily seen. Therefore,
\begin{equation}\label{asymptotic-Last-Formula}
    F\left(\frac k2+s\right)\sim \frac{\rho^k\sigma_{k-1}}{2}\left.
    \left(\frac{\rho-\eta}{\rho+\eta}\right)^{k/2}\Gamma\left(\frac k2\right)e^{-sp(0)}\right/
    \left(\frac{\rho\eta s}{(\rho+\eta)^2}\right)^{k/2}
\end{equation}
as $s\rightarrow\infty$. Since $p(0)=\ln\frac{\rho-\eta}{\rho+\eta}$, $F(k/2+s)=F(k/2-s)$ and $\sigma_{k-1}=2\pi^{k/2}/\Gamma(k/2)$, we can observe that~\eqref{asymptotic-Last-Formula} leads directly to~\eqref{Asymptotic-behavior-Formula-1}.

Recall that $\lambda\rho^2=\alpha k-\alpha^2$ and $\alpha=k/2\pm s$ yields $|s|=\sqrt{k^2/4-\lambda\rho^2}$. Observe also that if $s\rightarrow\pm\infty$, then
\begin{equation}\label{Equivalences}
    |s|\sim\rho\sqrt{-\lambda}\quad\text{and}\quad\left(\frac{\rho+\eta}{\rho-\eta}\right)^{|s|}\sim
    \left(\frac{\rho+\eta}{\rho-\eta}\right)^{\rho\sqrt{-\lambda}}\,.
\end{equation}
If take into account that $\sigma_{k-1}=2\pi^{k/2}/\Gamma(k/2)$, then the combination of \eqref{Equivalences} and \eqref{Asymptotic-behavior-Formula-1} leads to \eqref{Asymptotic-behavior-Formula-2}. This completes the proof of Theorem~\ref{Asymptotic-behavior-Theorem}.
\end{proof}

\subsection{Asymptotic behavior of $\varphi_\lambda(r)$ as $\lambda\rightarrow+\infty$.}

In this subsection we are going to describe the asymptotic behavior of $F(\alpha)=\int\limits_{S^k(\rho)} \omega^\alpha dS_u$, when $\alpha=k/2+ib\in\mathbb{R}$ and $b\rightarrow\pm\infty$. The asymptotic behavior of $\varphi_\lambda(r)=F(\alpha)/|S^k(\rho)|$ depends on the dimension of $H^{k+1}_\rho$. This is why we consider cases $k=1$, $k=2$ and $k\geq3$ separately. Theorem~\ref{Stationary-Phase-Theorem} taken from~\cite{Olver}, (p.~101) and stated below will be used as the basic tool to obtain the leading term of the asymptotic decomposition of $F(1/2+ib)$ and as a consequence, the leading term of $\varphi_\lambda(r)$ as $\lambda\rightarrow+\infty$ and $k=1$. To investigate the asymptotic behavior for $k>1$ we use integration by parts.

\begin{theorem}[Stationary Phase Theorem, \cite{Olver}, p.~101, (Theorem 13.1)]\label{Stationary-Phase-Theorem}
Suppose that in integral
\begin{equation}
    I(b)=\int\limits_c^d e^{ibp(t)}q(t)\,dt
\end{equation}
the limits $c$ and $d$ are independent on $b$. $p(t)$ and $q(t)$ are independent on $x$, $p(t)$ being real and $q(t)$ either real or complex. We assume that in the closure of $(c,d)$, the only possible point at which $p'(t)$ vanishes is $c$. In addition:
\begin{enumerate}
  \item In $(c,d)$, the functions $p'(t)$ and $q(t)$ are continuous, $p'(t)>0$, and $p''(t)$ and $q'(t)$ have at most a finite number of discontinuities and infinities.
  \item As $t\rightarrow c^+$ $$p(t)-p(c)\sim P(t-c)^\mu,\quad q(t)\sim Q(t-c)^{\nu-1},\quad \nu<\mu,$$
    the first of these relations being twice differentiable and the second one is differentiable. Here $P$, $\mu$ and $\nu$ are positive constants, and $Q$ is a real or complex constant. When $\mu=1$ this is to be interpreted as $p'(t)\rightarrow P$ and $p''(t)=o\{(t-c)^{-1}\}$. Similarly, $q'(t)=o\{(t-c)^{-1}\}$ in the case $\nu=1$.
  \item
        $$\int\limits_{\tau}^d \left|\left(\frac{q(t)}{p'(t)}\right)'\right|dt<\infty\quad\text{for each}\quad \tau\in(c,d)\,.$$
  \item As $t\rightarrow d^{-}$, $q(t)/p'(t)$ tends to a finite limit, and this limit is zero when $p(b)=\infty$.
\end{enumerate}
Then $$I(b)\sim e^{\nu\pi i/(2\mu)}\frac{Q}{\mu}\Gamma\left(\frac \nu\mu\right)\frac{e^{ibp(c)}}{(Pb)^{\nu/\mu}}$$
\end{theorem}

The theorem below gives the asymptotic behavior of $\varphi_\lambda(r)$ if $\lambda\rightarrow+\infty$.

\begin{theorem}\label{Lambda-Positively-Large-Theorem} Fix $r$. Then the asymptotic behavior of $\varphi_\lambda(r)$ as $\lambda\rightarrow\infty$ depends on $k$ as follows.
\begin{description}
  \item[(A)] If $k=1$, then
    \begin{equation}\label{Asymptotic-for-k=1-Stationary-Fase}
        \varphi_\lambda(r)=\sqrt{\frac{2}{\pi\rho\sinh(r/\rho)}}\cdot\frac{\cos(\pi/4-r\sqrt{\lambda})}{\lambda^{1/4}}
        +o(1/\lambda^{1/4})\,.
    \end{equation}
  \item[(B)] If $k=2$, then
    \begin{equation}\label{Asymptotic-for-k=2-Stationary-Fase}
        \varphi_\lambda(r)=\frac{\sin(br/\rho)}{b\sinh(r/\rho)}=\frac{1}{\rho\sinh(r/\rho)}\cdot
        \frac{\sin(r\sqrt{\lambda})}{\sqrt{\lambda}}+o(1/\sqrt{\lambda})\,.
    \end{equation}
  \item[(C)] If $k\geq3$, then $\varphi_\lambda(r)=o(1/\sqrt{\lambda})$.
  \item[(D)] In particular, if $k=4$, then
    \begin{equation}\label{Asymptotic-for-k=4-Stationary-Fase}
        \varphi_\lambda(r)=\left(\frac{\sqrt{3}}{\rho\sinh(r/\rho)}\right)^2\frac{\cos(r\sqrt{\lambda})}{\lambda}+o(1/\lambda)\,.
    \end{equation}
\end{description}
\end{theorem}

\begin{proof}[Proof of Theorem~\ref{Lambda-Positively-Large-Theorem}]
Note that according to Theorem~\ref{Radial-Vanishing-EF-Theorem}, p.~\pageref{Radial-Vanishing-EF-Theorem}, the condition that $\lambda$ is real and runs to $\infty$ implies $\alpha=k/2+ib$, where $b\in\mathbb{R}$ and $\lambda\rho^2=\alpha k-\alpha^2$. Recall also from \eqref{Three_D_Dirichlet-30}, p.~\pageref{Three_D_Dirichlet-30} that
\begin{equation}\label{Geometric-Representation-Radial-Solution}
    \varphi_\lambda(r(\eta))=\frac{4\rho(\rho^2-\eta^2)^{k/2}\sigma_{k-1}}{|S^k(\rho)|}
    \int\limits_0^{\pi/2}\frac{\sin^{k-1}\psi}{l+q}\cos\left(b\ln\frac lq\right)d\psi\,.
\end{equation}
Now we are ready to obtain proofs of all statements in Theorem~\ref{Lambda-Positively-Large-Theorem}.

\textbf{Proof of Statement (A)}.
If $k=1$, then $\sigma_0=2$ and then \eqref{Geometric-Representation-Radial-Solution} yields
\begin{equation}
    \varphi_\lambda(r(\eta))=\frac{4(\rho^2-\eta^2)^{1/2}}{\pi}\cdot
    \Re\int\limits_0^{\pi/2}\frac{e^{ib\ln(l/q)}}{l+q}\,d\psi\,.
\end{equation}
Let us set
\begin{equation}
    q(\psi)=\frac{1}{l+q},\quad p(\psi)=\ln\omega=\ln\frac lq,\quad I(b)=\int\limits_0^{\pi/2}\frac{e^{ib\ln(l/q)}}{l+q}\,d\psi
\end{equation}
and observe that $I(b)$ can be investigated using Theorem~\ref{Stationary-Phase-Theorem} above. Indeed the direct computation shows that $q(\psi)$ as well as $p(\psi)$ satisfy all of the conditions of Theorem~\ref{Stationary-Phase-Theorem} and we can compute the coefficients $P,\mu,Q,\nu$ as follows. First we observe that $p'(0)=0$ and $p''(0)=2\eta/\rho$, which yields
\begin{equation}
    p(\psi)-p(0)\sim\frac \eta\rho(\psi-0)^2\quad\text{and}\quad q(\psi)\sim\frac{1}{2\rho}(\psi-0)^{1-1}\,.
\end{equation}
Therefore, $$P=\frac \eta\rho,\quad \mu=2,\quad Q=\frac{1}{2\rho},\quad\nu=1\,,$$
and then
\begin{equation}
    I(b)\sim\frac{e^{i\pi/4}}{4\rho}\Gamma\left(\frac12\right)\left.e^{ib\ln\frac{\rho-\eta}{\rho+\eta}}\right/(b\eta/\rho)^{1/2}\,.
\end{equation}
If we take the real part of $I(b)$ and observe that
$$\Gamma(1/2)=\sqrt{\pi},\quad\ln\frac{\rho-\eta}{\rho+\eta}=-r/\rho,\quad b=\sqrt{\lambda\rho^2-k^2/4}\sim\rho\sqrt{\lambda}$$
as $\lambda\rightarrow\infty$, we arrive to \eqref{Asymptotic-for-k=1-Stationary-Fase}. This completes the proof of Statement~(A) of Theorem~\ref{Lambda-Positively-Large-Theorem}.

\textbf{Proof of Statement (B)}. The proof of Statement~(B) can be obtained as a combination of the asymptotic relationship $b\sim\rho\sqrt{\lambda}$ and the explicit solution obtained for the case $k=2$. Indeed, according to~\eqref{Varphi-of-b-Explicit-3-D}, p.~\pageref{Varphi-of-b-Explicit-3-D},
\begin{equation}
    \varphi_\lambda(r(\eta))=\frac{\rho^2-\eta^2}{2\eta\rho b}\sin\left(\frac{br}{\rho}\right)=
    \frac{\sin(br/\rho)}{b\sinh(r/\rho)}\,,
\end{equation}
which leads directly to \eqref{Asymptotic-for-k=2-Stationary-Fase}. This completes the proof of Statement~(B) of Theorem~\ref{Lambda-Positively-Large-Theorem}.

\textbf{Proof of Statement (C)}. If $k\geq3$, then recall from \eqref{Lower_Bound-22}, p.~\pageref{Lower_Bound-22} that
\begin{equation}
    \varphi_\lambda(r(\eta))=\frac{\rho(\rho^2-\eta^2)^{k/2}\sigma_{k-1}}{\eta|S^k(\rho)|\cdot|b|}
    \int\limits_0^{\pi/2}\sin\left(b\ln\frac ql\right) d\sin^{k-2}\psi\,.
\end{equation}
If we make the substitution $\tau=\ln(q/l)$, we can apply Riemann-Lebesque Lemma to show that the integral above vanishes as $b\rightarrow\infty$. If we recall that $b\sim\rho\sqrt{\lambda}$, then
\begin{equation}
    \varphi_\lambda(r(\eta))=\frac{(\rho^2-\eta^2)^{k/2}\sigma_{k-1}}{\eta|S^k(\rho)|}\cdot
    \frac{\tau(\sqrt{\lambda})}{\sqrt{\lambda}}\,,
\end{equation}
where $\tau(\sqrt{\lambda})\rightarrow0$ as $\lambda\rightarrow\infty$. This completes the proof of Statement~(C) of Theorem~\ref{Lambda-Positively-Large-Theorem}.

\textbf{Proof of Statement (D)}. If $k=4$, then the leading term can be computed using the integration by parts another time. Indeed, using the substitution for $d\psi$ obtained in~\eqref{Lower_Bound-21}, p.~\pageref{Lower_Bound-21} and the integration by parts, we have
\begin{equation}
\begin{split}
     \frac{\eta\rho^4\sigma_4\cdot|b|}{\rho(\rho^2-\eta^2)^2\sigma_3}\cdot\varphi_\lambda(r(\eta))  & =
     \left.\frac{(l+q)\cos\psi\cos(b\ln\omega)}{2\eta\rho}\right|_{\psi=0}^{\psi=\pi/2}
    \\& -\frac{1}{2\eta b}\int\limits_0^{\pi/2}\cos(b\ln\omega)[(l+q)\cos\psi]'d\psi\,.
\end{split}
\end{equation}
Note that according to~\eqref{(l+q)-Differentiation-formula}, p.~\pageref{(l+q)-Differentiation-formula}, $[(l+q)\cos\psi]'$ is an integrable function and then, we can apply Riemann-Lebesque Lemma as above and see that the last integral vanishes as $\lambda\rightarrow\infty$. If we take into account that
\begin{equation}
\begin{split}
     & b\sim\rho\sqrt{\lambda}\quad \text{and}\quad\frac{\rho^2-\eta^2}{\rho\eta}=\frac{2}{\sinh(r/\rho)}\quad\text{as well as}
     \quad\Gamma(2)=1,
     \\& \Gamma\left(\frac 52\right)=\frac{3\sqrt{\pi}}{4},\quad
     \sigma_{k-1}=\frac{2\pi^{k/2}}{\Gamma(k/2)}\quad\text{and then}\quad\frac{\sigma_3}{\sigma_4}=\frac34,
\end{split}
\end{equation}
we arrive at \eqref{Asymptotic-for-k=4-Stationary-Fase}. This completes the proof of Statement~(D) as well as the proof of Theorem~\ref{Lambda-Positively-Large-Theorem}.
\end{proof}

\chapter{Appendix.}

\section{The detailed description of level curves.}

In this section we shall see a more detailed description of the level set for the equation~\eqref{Level-set-equation-Appendix} below. Originally, a less detailed description was given in Proposition~\ref{Level-curve-description}, p.~\pageref{Level-curve-description}. Here we are going to use the notation introduced in the proof of Statement~(D) of Main Theorem~\ref{Basic-theorem}, see p.~\pageref{Level-Curve-Story-Proof}. Recall that $p$ is chosen in such a way that
\begin{equation}\label{p-levelset-Condition-Appendix}
        \cos(\zeta\ln\omega(\theta))\geq0\quad\text{for every}\quad(\xi,\theta)\in[-p,p]\times[0,\pi]
\end{equation}
and $Q=(k/2,\infty)\times(0,p)$.

\begin{theorem}[The description of level curves]\label{Level-curve-descri-Appendix} \
    \begin{enumerate}
      \item For every fixed $\beta=a+ib\in\overline{Q}$ the set of solutions $S(a,b)$ for the equation
      \begin{equation}\label{Level-set-equation-Appendix}
        W(\xi,\zeta)=\int\limits_{S^k}\omega^\xi\cos(\zeta\ln\omega)dS_y=W(a,b)
      \end{equation}
      is a $\mathcal{C}^1$ curve $\gamma_{(a,b)}(t)=\gamma(t)=(\xi(t),\zeta(t))\subseteq\overline{Q}$.
      \item No two of these curves have common points in $\overline{Q}$.
      \item There are three types of such level curves in $\overline{Q}$ sketched on the Figure below and described as follows.

\begin{figure}[!h]
    \centering
    \epsfig{figure=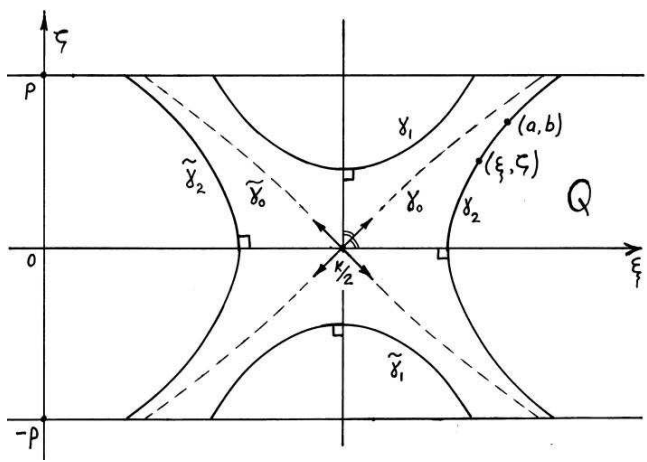,height=6cm}\\
    \caption{All possible level curves.}\label{Level-Curves-fig}
\end{figure}

      \begin{description}
        \item[First type $\gamma_1$,] a $\mathcal{C}^{\infty}$ curve $\gamma_1(t)=(t,v(t))$ defined on some finite interval $k/2\leq t\leq a_1$ such that $v(a_1)=p$. This curve starts on the vertical edge of $\overline{Q}\setminus\{(k/2,0)\}$ and is orthogonal to the edge, i.e., $v'(k/2)=0$. The function $v(t)$ is strictly increasing on the interval $[k/2,a_1]$ and $v'(t)>0$ for every $t\in(k/2,a_1]$.
        \item[Second type $\gamma_2$,] a $\mathcal{C}^{\infty}$ curve $\gamma_2(t)=(u(t),t)$ defined for every $t\in[0,p]$. This curve starts on the lower edge of $\overline{Q}\setminus\{(k/2,0)\}$ and is orthogonal to the edge, i.e., $u'(0)=0$. The function $u(t)$ is a strictly increasing on the interval $[0,p]$ and $u'(t)>0$ for every $t\in(0,p]$.
        \item[Third type $\gamma_0$,] a unique curve $\gamma_0(t,v(t))$ defined on some finite interval $k/2\leq t\leq a_1$ such that $v(a_1)=p$. This curve starts at the lower corner $(k/2,0)$ and it bisects the angle at the corner, i.e., $v'(k/2)=1$. The function $v(t)\in\mathcal{C}^{\infty}(k/2,a_1]$ is strictly increasing on the interval $[k/2,a_1]$ and $v'(t)$ is a positive continuous function of $t\in[k/2,a_1]$. This curve separates the curves of the first and the second types in $\overline{Q}$.
      \end{description}
    \end{enumerate}
\end{theorem}

\begin{proof}[\textbf{Proof of Theorem \ref{Level-curve-descri-Appendix}}]

Note that $\nabla W(\xi,\zeta)\neq0$ for every $(\xi,\zeta)\in \overline{Q}\setminus\{(k/2,0)\}$. Therefore, if a level curve starts at the vertical or at the lower horizontal edge of $\overline{Q}\setminus{(k/2,0)}$, then its behavior can be obtained as the combination of Claims~\ref{Partial-dW-d-xi-claim},~\ref{partial-vertical-claim} (pp.~\pageref{Partial-dW-d-xi-claim},~\pageref{partial-vertical-claim}), the Implicit Function Theorem and the argument used in the proof of Proposition~\ref{Level-curve-description}, p.~\pageref{Level-curve-description}.

To complete the proof of Theorem~\ref{Level-curve-descri-Appendix} we only need to describe the behavior of the level curve starting with the corner $C=(k/2,0)$. This must be the curve of the third type $\gamma_0$. Let $a_0=k/2$ and $C=(a_0, 0)$. It is clear that by the Implicit Function Theorem combined with argument used in the proof of Proposition~\ref{Level-curve-description}, p.~\pageref{Level-curve-description}, there exists a $\mathcal{C}^{\infty}$ function $v(t)$ for $t\in(k/2,a_1]$ such that $\gamma_0(t)=(t,v(t))$ is the level curve for $W(\xi,\zeta)$, $W(\gamma_0(t))=W(a,b)$ for every $t\in[k/2,a_1]$ and $v(t)$ is strictly increasing function for $t\in[k/2, a_1]$. Thus, to complete the proof of Theorem~\ref{Level-curve-descri-Appendix} we need to show that $v(t)$ is continuously differentiable at $t=k/2=a_0$ and $v'(a_0)=1$. Note that The Implicit Function Theorem does not help to describe the behavior of $\gamma_0(t)$ at the corner $C$, since $\nabla W(a_0,0)=0$. So, we need to develop a different approach.
      \begin{claim}\label{upper-line-claim}
      Let $t\in[a_0,a_1]$. Then $v(t)$ is continuous at $t=a_0^{+}$ and $v(t)<t$ for every $t\in(a_0, a_1]$.
      \end{claim}
      \begin{proof}[Proof of Claim \ref{upper-line-claim}.] Let us introduce a new parameter $s=t-a_0$ and let $\psi_1(s)=(a_0+s,s)$ be the bisector of the left lower corner $C$. If we take into account \eqref{logarithm-integral}, p.~\pageref{logarithm-integral}, the direct computation yields
      \begin{equation}\label{Three-derivative-of-W}
        \frac{dW(\psi(0))}{ds}=\frac{d\,^2W(\psi(0))}{ds^2}=\frac{d\,^3W(\psi(0))}{ds^3}=0
      \end{equation}
      and
      \begin{equation}\label{Fourth-derivative-of-W}
        \frac{d\,^4W(\psi(s))}{ds^4}=-4\int\limits_{S^k}\omega^{a_0+s}(\ln\omega)^4\cos(s\ln\omega)dS_y<0
      \end{equation}
      for every $s\in[0, p]$. This implies that $W(\psi_1(s))$ is strictly decreasing for $s\in[0,p]$. Therefore, $v(t)<t-a_0$ for $t\in(a_0, a_1]$ and then
      \begin{equation}
        0\leq\lim\limits_{t\rightarrow a_0^{+}}v(t)\leq\lim\limits\limits_{t\rightarrow a_0^{+}}(t-a_0)=0\,.
      \end{equation}
      Hence, the continuity of $v(t)$ at $t=a_0$ follows. This completes the proof of the Claim \ref{upper-line-claim}.
      \end{proof}

      The following claim establishes the derivative of $v(t)$ at $t=0$.

      \begin{claim}\label{lower-line-claim}
      For any real number $h\in(0,1)$ there exists an interval $[a_0, a_0+d_h]$ such that $d_h>0$ and
      \begin{equation}\label{lower-line-estimation}
        h\leq\frac{v(t)-v(a_0)}{t-a_0}\leq1\quad\text{for every}\,\,\,t\in[a_0, a_0+d_h]\,.
      \end{equation}
      \end{claim}
      \begin{proof}[Proof the Claim \ref{lower-line-claim}.] Let us introduce the line $\psi_h=(a_0+s,hs)$. This line together with the function $v(t)$ are pictured on Figure~\ref{Curve-bisector} below.

\begin{figure}[!h]
    \centering
    \epsfig{figure=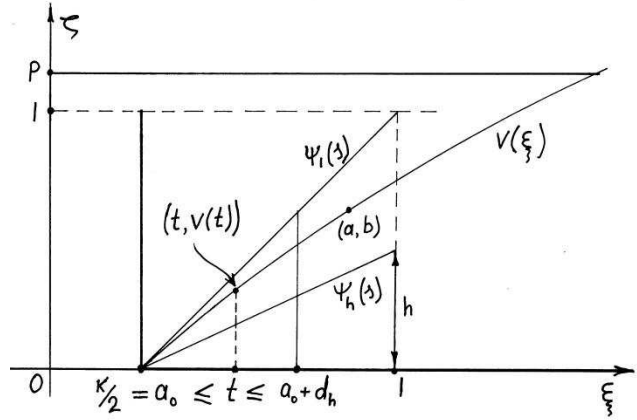,height=6cm}\\
    \caption{The bisector curve $v(t)$.}\label{Curve-bisector}
\end{figure}

      Note that $dW(\psi_h(0))/ds=0$, while
      \begin{equation}
        \frac{d^2W(\psi_h(0))}{ds^2}=(1-h^2)\int\limits_{S^k}\omega^{k/2}(\ln\omega)^2dS_y>0.
      \end{equation}
      This means that for every $h\in(0,1)$ the value of $W(\psi_h(s))$ is strictly increasing within some interval $[0,d_h]$. Clearly, $d_h$ can be chosen as the minimum between $p$ and the right boundary of the interval where $d^2W(\psi_h(s))/ds^2$ remains non-negative. Therefore,
      \begin{equation}
        W(\psi_h(s))>W(a_0,0)\quad\text{for every}\,\,\,s\in(0,d_h]\,,
      \end{equation}
      while, by \eqref{Three-derivative-of-W} and \eqref{Fourth-derivative-of-W},
      \begin{equation}
        W(\psi_1)<W(a_0,0)\quad\text{for every}\,\,\,s\in(0,p]\,.
      \end{equation}
      Hence, for every $s\in(0,d_h]$ we have
      \begin{equation}
        W(a_0+s,s)<W(a_0,0)=W(a,b)<W(a_0+s,hs)\,,
      \end{equation}
      which yields
      \begin{equation}
        h(t-a_0)\leq v(t)\leq t-a_0\quad\text{for every}\,\,\,t\in[a_0,a_0+d_h]\,.
      \end{equation}
      If we take into account that $v(a_0)=0$, we can see that the last sequence of inequalities completes the proof of the Claim~\ref{lower-line-claim}.
      \end{proof}
      \begin{corollary}\label{derivative-v-at-k/2-claim}
        \begin{equation}\label{derivative-v-at-k/2}
            v'(a_0^{+})=\lim\limits_{t\rightarrow a_0^{+}}\frac{v(t)-v(a_0)}{t-a_0}=1\,.
        \end{equation}
      \end{corollary}
      \begin{proof}[Proof of the Corollary \ref{derivative-v-at-k/2-claim}.] Formula \eqref{derivative-v-at-k/2} is a simple consequence of \eqref{lower-line-estimation} from Claim \ref{lower-line-claim}, p.~\pageref{lower-line-claim}.
      \end{proof}


      The last step is to show that $v'(t)$ is continuously differentiable at $t=a_0$, i.e.,
      \begin{equation}
            \lim\limits_{t\rightarrow a_0^{+}}v'(t)=v'(a_0)=1\,.
      \end{equation}
      Recall that using Implicit Function Theorem the following formula
      \begin{equation}\label{derivative-v-by-Implicit-FT-Appendix}
        \frac{dv(t)}{dt}=\frac{W_\xi(t,v(t))}{-W_\zeta(t,v(t))}>0\quad\text{for every}\,\,t\in(a_0,a_1)
      \end{equation}
 was obtained as formula \eqref{derivative-v-by-Implicit-FT}, p.~\pageref{derivative-v-by-Implicit-FT}. Recall also that $t=a_0+s$. If we denote the numerator and the denominator in \eqref{derivative-v-by-Implicit-FT-Appendix} as
      \begin{equation}\label{Numerator-Denominator-Expressions}
      \begin{split}
        & N(\xi,\zeta)=\int\limits_{S^k}\omega^\xi\ln\omega\cos(\zeta\ln\omega)dS_y\quad\text{and}
        \\& D(\xi,\zeta)=\int\limits_{S^k}\omega^\xi\ln\omega\sin(\zeta\ln\omega)dS_y
      \end{split}
      \end{equation}
      respectively, we can observe that all of the following partial derivatives $N_\xi, N_\zeta, D_\xi, D_\zeta$ are continuous at $(a_0, 0)$. As we saw in \eqref{derivative-v-at-k/2}, $v(t)$ is differentiable at $a_0^{+}$. Therefore, both of the following compound functions $N(t,v(t))$ and $D(t,v(t))$ are differentiable at $a_0^{+}$ and
      \begin{equation}
      \begin{split}
        & \frac{dN(a_0^+,0)}{dt}=N_\xi(a_0,0)+\frac{dv(a_0^{+})}{dt}N_\zeta(a_0,0)\,;
        \\& \frac{dD(a_0^+,0)}{dt}=D_\xi(a_0,0)+\frac{dv(a_0^{+})}{dt}D_\zeta(a_0,0)\,.
      \end{split}
      \end{equation}
      Therefore, the values of these derivatives are computable at $a_0^{+}$ and the direct computation, using \eqref{Numerator-Denominator-Expressions}, yields
      \begin{equation}
        \frac{dN(a_0^+,0)}{dt}=\frac{dD(a_0^+,0)}{dt}=\int\limits_{S^k}\omega^{a_0}(\ln\omega)^2dS_y>0\,.
      \end{equation}
      Thus, the L'Hospital,s Rule is applicable to the expression \eqref{derivative-v-by-Implicit-FT-Appendix} for $v'(t)$ and again, the direct computation yields
      \begin{equation}
        \lim\limits_{t\rightarrow a_0^{+}}v'(t)=\lim\limits_{t\rightarrow a_0^{+}}\frac{N(t,v(t))}{D(t,v(t))}
        =\left.   \frac{dN(a_0^+,0)}{dt}   \right/     \frac{dD(a_0^+,0)}{dt}=1\,,
      \end{equation}
      which completes the description of the separation curve $\gamma_0(t)$ and the proof of Theorem~\ref{Level-curve-descri-Appendix}.
\end{proof}

\section{Alternative proofs.}

\subsection{The angle introduced by Hermann Schwarz.}

In this subsection we obtain an alternative proof of Statement~(A) of Theorem~\ref{Basic-theorem}, p.~\pageref{Basic-theorem}. We shall use here the angle and the change of variables introduced by Hermann Schwarz, see~\cite{Ahlfors}, p.~168.

\begin{theorem}[Statement (A) of Main Theorem~\ref{Basic-theorem}]\label{Schwartz-Angle-Theorem} Let $S^k$ be the $k-$dimensional sphere of radius $R$ centered at the origin. If $\alpha, \beta\in\mathbb{C}$ and $\alpha+\beta=k$, then
\begin{equation}\label{Schwartz-Angle-Formula}
    \int\limits_{S^k}\omega^\alpha dS_y=
    \int\limits_{S^k}\omega^\beta dS_y\,.
\end{equation}
\end{theorem}

\begin{proof}[Proof of Theorem \ref{Schwartz-Angle-Theorem}]


Let us assume first that $x$ is inside the sphere $S^k$. For the
notations see Figure~\ref{Schwartz-Angle-Picture} below. Let $y$ be an arbitrary
point of $S^k$. Here, $y^\star$ is defined as an
intersection of the sphere $S^k$ and line $xy$. Angle
$\theta$ is defined as $\pi - \angle xOy$ and $\theta^\star$ is
defined as $\pi-\angle xOy^\star$. The circle $S^1$ shown
in the left part of Figure~\ref{Schwartz-Angle-Picture} is the cross-section of the sphere with the plane defined by the
points $O$, $x$ and $y\in S^k$. Let $q(\theta)=|xy|$,
$l(\theta)=|y^\star x|$ and let the segments $a$ shown in Figure~\ref{Schwartz-Angle-Picture} be
perpendicular to $xO$.

\begin{figure}[!h]
    \centering
    \epsfig{figure=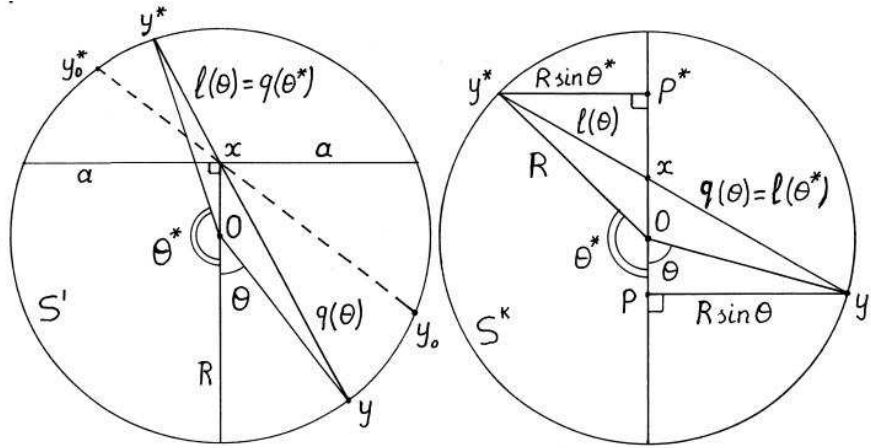,height=6cm}\\
    \caption{Schwarz Angle $\theta^*$.}\label{Schwartz-Angle-Picture}
\end{figure}


Hermann Schwarz observed that for every $y_0\in S^k$, triangles $\triangle y_0^*xy^*$ and $\triangle xy_0y$ are similar and then, $l\cdot|d\theta^*|=q\cdot|d\theta|$, see~\cite{Schwarz}, p.~260. Note also that angle $\theta^*$ is decreasing, while angle $\theta$ is increasing. Therefore,
\begin{equation}\label{SpheProDod-4}
    \frac{d\theta^*}{d\theta}=-\frac{|d\theta^*|}{|d\theta|}=
    -\frac{l(\theta)}{q(\theta)}\,.
\end{equation}
The next important observation is presented in the following Proposition.

\begin{proposition}\label{l(theta)=q(theta-star)-proposition}
    \begin{equation}\label{l(theta)=q(theta-star)-formula}
        l(\theta)=q(\theta^*)\quad\text{and}\quad q(\theta)=l(\theta^*)\,.
    \end{equation}
\end{proposition}

\begin{proof}[Proof of Proposition \ref{l(theta)=q(theta-star)-proposition}]
In the proof of the Proposition a reader may refer to the Figure~\ref{Schwartz-l-q-relation-figure} below.

\begin{figure}[!h]
    \centering
    \epsfig{figure=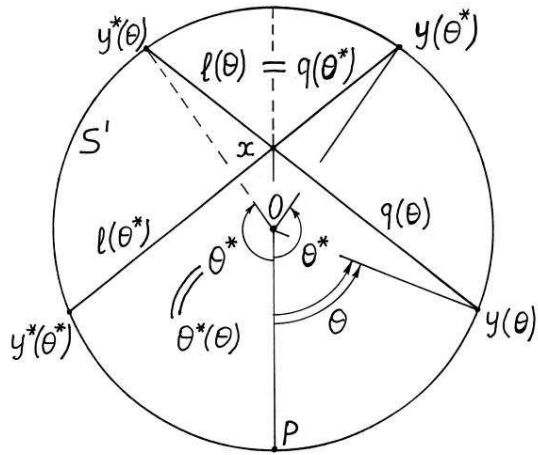,height=6cm}\\
    \caption{$l, q$ relation with respect to Schwarz angle.}\label{Schwartz-l-q-relation-figure}
\end{figure}

Note that both values $l$ and $q$ can be pictured on a one dimensional sphere $S^1$ of radius $R$ centered at the origin $O$. Let $P$ be the intersection of the ray $\overrightarrow{xO}$ and $S^1$. Let $\theta$ counted counterclockwise from the segment $OP$ be the initial argument. It is clear that for every $\theta\in[0,\pi]$ there exists the unique point $y=y(\theta)\in S^1$ such that $\angle POy(\theta)=\theta$. Thus
\begin{equation}\label{l(theta)=q(theta-star)-formu-Prelimi}
    l(\theta)=|xy^*(\theta)|\quad\text{and}\quad q(\theta^*)=q(\theta)|_{\theta=\theta^*}=|xy(\theta^*)|\,.
\end{equation}
Therefore, if $\theta^*$ is defined as $\theta^*=\angle y^*(\theta)OP$, then $y(\theta)|_{\theta=\theta^*}=y(\theta^*)$ must be symmetric to $y^*(\theta)$ with respect to the line $Ox$. This, according to~\eqref{l(theta)=q(theta-star)-formu-Prelimi}, leads to the first identity in~\eqref{l(theta)=q(theta-star)-formula}. The same argument can be used to obtain the second identity in~\eqref{l(theta)=q(theta-star)-formula}, which completes the proof of Proposition~\ref{l(theta)=q(theta-star)-proposition}.
\end{proof}

From now on a reader can be referred back to Figure~\ref{Schwartz-Angle-Picture}. Observe that \eqref{SpheProDod-4} combined with \eqref{l(theta)=q(theta-star)-formula} yields
\begin{equation}\label{3-z}
    d\theta=-\frac{q(\theta)}{l(\theta)}d\theta^\star=
    -\frac{l(\theta^\star)}{q(\theta^\star)} d\theta^\star\,.
\end{equation}
Let $P$ and $P^\star$ be orthogonal projections of $y$ and
$y^\star$ to line $xO$ respectively, see the right part of Figure~\ref{Schwartz-Angle-Picture}. Clearly, $\triangle
xy^\star P^\star$ is similar to $\triangle xyP$ and thus
\begin{equation}\label{4-z}
    l(\theta)\,
    R \sin\theta=q(\theta)\,
    R\sin\theta^\star=
    l(\theta^\star)
    \, R\sin\theta^\star\,.
\end{equation}
Recalling that $\omega=l/q$ and applying \eqref{l(theta)=q(theta-star)-formula}, \eqref{3-z}, \eqref{4-z}, we have
\begin{equation}\label{5-z}
\begin{split}
  & \int\limits_{S^k} \omega^\alpha d\,S_y=
    \int\limits_0^\pi
    \left(\frac{l(\theta)}{q(\theta)}\right)^\alpha
    (R\sin\theta))^{k-1} \sigma_{k-1} Rd\theta
    \\& =
    \int\limits_0^\pi
    \left(\frac{l(\theta^\star)}{q(\theta^\star)}\right)^{k-\alpha}(R\sin\theta^\star)^{k-1}
    \sigma_{k-1}Rd\theta^\star
    = \int\limits_{S^k}
    \omega^{k-\alpha} d\,S_y\,,
\end{split}
\end{equation}
where $\sigma_{k-1}$ is the area of $(k-1)$-dimensional unit
sphere. The formula
$d\,S_y=(R\sin\theta)^{k-1}\sigma_{k-1}R\,d\theta$
holds, since our integrand $\omega(x,y)$ depends only on angle
$\theta$ when $x$ is fixed. This completes the proof of Theorem~\ref{Schwartz-Angle-Theorem} for the case when $|x|<R$, i.e., when $x$ is inside the sphere $S^k$. To see that formula~\eqref{Schwartz-Angle-Formula} holds for $|x|>R$ we need to apply the argument used in the proof of Case~2 of Statement~(A) of Theorem~\ref{Basic-theorem}, see p.~\pageref{x-is-outside-proof-Main-Thm}. This completes the proof of Theorem~\ref{Schwartz-Angle-Theorem}.
\end{proof}

\subsection{Laplacian and Spherical radialization.}

In this subsection we shall see the alternative proof of Lemma~\ref{Laplacian-Commutes-lemma}, p.~\pageref{Laplacian-Commutes-lemma}. Partly, this proof relies on the technique developed in the Randol's paper of \cite{Chavel}, pp.~ 271-273. Let $x$ and $y$ be two arbitrary points in $H^n_\rho$; $\tau=d(x,y)$ denotes the hyperbolic distance between $x$ and $y$ and $\beth^{n-1}(x,\tau)$ be the geodesic sphere of radius $\tau$ centered at $x$. We define
\begin{equation}\label{Radialization-Definition-Appendix}
    f^\sharp_x(y)=\frac{1}{\text{vol}(\beth^{n-1}(x,\tau))}\int\limits_{\beth^{n-1}(x,\tau)} f(y_1) d\,\beth_{y_1}\,,
\end{equation}
where the integration is considered with respect to the measure on $\beth^{n-1}(x,\tau)$ induced by the hyperbolic measure of $H^n_\rho$ and $\text{vol}(\beth^{n-1}(x,\tau))$ is the hyperbolic volume of $\beth^{n-1}(x,\tau)$. Note that we can modify the inverse of Cayley transform given in \eqref{Uniq-13}, p.~\pageref{Uniq-13}, in such a way to get an isometry of $H^n_\rho$ to $B^n_\rho$ which sends $x\in H^n_\rho$ to the origin $O\in B^n_\rho$. It is clear in this case that $\beth^{n-1}(x,\tau)\subset H^n_\rho$ will be mapped to the Euclidean sphere centered at the origin $S^{n-1}(R)\subset\mathbb{R}^n$ as long as $R=\rho\tanh(\tau/(2\rho))$. If we take into account that the hyperbolic metric in $B^n_\rho$ at some point depends only on the distance to the origin, we can observe that the radialization defined in \eqref{Laplacian-Commutes}, p.~\pageref{Laplacian-Commutes}, gives the same result as the radialization defined in \eqref{Radialization-Definition-Appendix} above and then, Lemma~\ref{Laplacian-Commutes-lemma}, p.~\pageref{Laplacian-Commutes-lemma} is equivalent to Lemma~\ref{Laplacian-Commutes-lemma-Appendix} below.

\begin{lemma}\label{Laplacian-Commutes-lemma-Appendix} Let $f^\sharp_x(y)$ is defined in \eqref{Radialization-Definition-Appendix}. Therefore,
\begin{equation}\label{Laplacian-Commutes-Appendix}
    \triangle_y f^\sharp_x(y)=(\triangle_y f(y))^\sharp_x(y)\,.
\end{equation}
\end{lemma}

\begin{proof}[Proof of Lemma \ref{Laplacian-Commutes-lemma-Appendix}]
    Let $K(p) (-\infty<p<\infty)$ be an even $\mathcal{C}^{\infty}$ function having compact support. By setting $K(x,y)=K(d(x,y))$, where $(x,y)\in H^{n}_\rho\times H^{n}_\rho$, we obtain a smooth function of $x$ and $y$ which depends only on the hyperbolic distance between $x$ and $y$. Let $f(y)\in\mathcal{C}^2(H^{n}_\rho)$. Then, the direct computations yield the following identities:
\begin{equation}\label{Formula-Base-1-Appendix}
    \int\limits_{H^n_\rho}K(x,y)\triangle_y f(y)dy=\int\limits_{H^n_\rho}\triangle_yK(x,y)f(y)dy\,;
\end{equation}
\begin{equation}\label{Formula-Base-2-Appendix}
    \int\limits_{H^n_\rho}K(x,y)f^\sharp_x(y) dy=\int\limits_{H^n_\rho}K(x,y)f(y)dy\,;
\end{equation}
\begin{equation}\label{Formula-Base-3-Appendix}
    \triangle_yK(x,y)=\widetilde{K}(x,y)\,,
\end{equation}
where $\widetilde{K}(x,y)$ is also a function depending only on the hyperbolic distance between $x$ and $y$.
Combining \eqref{Formula-Base-1-Appendix}, \eqref{Formula-Base-2-Appendix} and \eqref{Formula-Base-3-Appendix} we can write the following sequence of equalities.
$$ \int\limits_{H^n_\rho}K(x,y)\triangle_y[f^\sharp_x(y)]dy=\int\limits_{H^n_\rho}[\triangle_y K(x,y)] f(y)dy= $$
$$ =\int\limits_{H^n_\rho}K(x,y)\triangle_y f(y)dy=\int\limits_{H^n_\rho}K(x,y)[\triangle_yf(y)]^\sharp_x(y)dy\,, $$
which yields

\begin{equation}\label{Sequence-result-Laplacian-Appendix}
    \int\limits_{H^n_\rho}K(x,y)\{\triangle_y[f^\sharp_x(y)]-[\triangle_yf(y)]^\sharp_x(y) \}dy=0\,.
\end{equation}
Fix any $x\in H^n_\rho$ and observe that the integrand in \eqref{Sequence-result-Laplacian-Appendix} depends only on the hyperbolic distance between $x$ and $y$. Since $K(x,y)$ is an arbitrary non-negative function of $d(x,y)$, we may conclude that $$\triangle_y[f^\sharp_x(y)]-[\triangle_yf(y)]^\sharp_x(y)\equiv0\quad\text{for every}\quad y\in H^n_\rho\,,$$
which completes the proof of Lemma~\ref{Laplacian-Commutes-lemma-Appendix}.
\end{proof}

\begin{remark}
It is clear that the argument used in the proof of Lemma~\ref{Laplacian-Commutes-lemma-Appendix} yields formula~\eqref{Laplacian-Commutes-Appendix} for the Laplacian of any rotationally symmetric metric in $\mathbb{R}^n$. In particular, if $\triangle$ is chosen as the Euclidean Laplacian, formula \eqref{Laplacian-Commutes-Appendix} leads to the Euler-Poisson-Darboux equation. The standard way to obtain such an equation is presented in \cite{FritzWaves}, pp.~88-89 and relies on the Stokes' formula.  The proof of the following corollary derives the Euler-Poisson-Darboux equation as the consequence of formula \eqref{Laplacian-Commutes-Appendix}.
\end{remark}

\begin{corollary}[Euler-Poisson-Darboux formula]\label{Radi-fu-Wave-equa-lemma}
Let $f(x)$ be any function in $\mathbb{R}^n$, $y=x+z$ and
\begin{equation}\label{Eucli-Radialization}
    F(x,r)=\frac{1}{|S(r)|}\int\limits_{|z|=r} f(x+z)dz\,.
\end{equation}
Then
\begin{equation}\label{Radi-function-Wave-equation-formula}
\triangle_x F(x,r)=F_{rr}+\frac{n-1}{r} F_r 
\end{equation}
which is known as the Euler-Poisson-Darboux formula.
\end{corollary}

\begin{proof} Combining the Leibnitz rule of differentiation under integral sign together with formula \eqref{Laplacian-Commutes-Appendix}, we may write the following sequence of equalities.
\begin{equation}
\begin{split}
    \triangle_x F(x,r)=\frac{1}{|S(r)|}\int\limits_{|x-y|=r} \triangle_y f(y)dy=\triangle_y F(x,r)=\triangle_r F(x,r)\,.
\end{split}
\end{equation}
where the second equality is obtained by \eqref{Laplacian-Commutes-Appendix}. This completes the proof of Corollary~\ref{Radi-fu-Wave-equa-lemma}.
\end{proof}






\backmatter

\bibliographystyle{alpha}

(Sergei Artamoshin) The Graduate Center, The City University of New York, Department of Mathematics, 365 Fifth Avenue, New York, NY, 10016.

\emph{E-mail address:} kolmorpos@yahoo.com

\end{spacing}

\end{document}